\documentclass{amsart}

\address{Simons Center for Geometry and Physics,
State University of New York, Stony Brook, NY 11794-3636 U.S.A.} \email{kfukaya@scgp.stonybrook.edu}

\usepackage[dvipdfmx]{graphicx}
\usepackage{amsmath}
\usepackage{amscd}
\usepackage{amssymb}
\usepackage{amstext}
\usepackage{amsmath}
\usepackage[all]{xy}
\usepackage{yhmath}
\usepackage{mathrsfs}
\usepackage{bbding}

\makeatletter
\def\l@section{\@tocline{1}{0pt}{3mm}{8mm}{}}
\def\l@subsection{\@tocline{2}{0pt}{6mm}{10mm}{}}
\def\l@subsubsection{\@tocline{3}{0pt}{9mm}{11mm}{}}
\makeatother

\def\E{\ifmmode{\mathbb E}\else{$\mathbb E$}\fi} 
\def\N{\ifmmode{\mathbb N}\else{$\mathbb N$}\fi} 
\def\R{\ifmmode{\mathbb R}\else{$\mathbb R$}\fi} 
\def\Q{\ifmmode{\mathbb Q}\else{$\mathbb Q$}\fi} 
\def\C{\ifmmode{\mathbb C}\else{$\mathbb C$}\fi} 
\def\H{\ifmmode{\mathbb H}\else{$\mathbb H$}\fi} 
\def\Z{\ifmmode{\mathbb Z}\else{$\mathbb Z$}\fi} 
\def\P{\ifmmode{\mathbb P}\else{$\mathbb P$}\fi} 
\def\T{\ifmmode{\mathbb T}\else{$\mathbb T$}\fi} 
\def\SS{\ifmmode{\mathbb S}\else{$\mathbb S$}\fi} 
\def\DD{\ifmmode{\mathbb D}\else{$\mathbb D$}\fi} 
\def\K{\ifmmode{\mathbb K}\else{$\mathbb K$}\fi}

\theoremstyle{theorem}
\newtheorem{thm}{Theorem}[section]
\newtheorem{cor}[thm]{Corollary}
\newtheorem{lem}[thm]{Lemma}
\newtheorem{sublem}[thm]{Sublemma}

\newtheorem{prop}[thm]{Proposition}

\newtheorem{conj}[thm]{Conjecture}

\theoremstyle{definition}
\newtheorem{defn}[thm]{Definition}
\newtheorem{rem}[thm]{Remark}

\newtheorem{exm}[thm]{Example}

\newtheorem{defnlem}[thm]{Lemma-Definition}

\numberwithin{equation}{section}
\makeindex
\begin{document}

\title[Gromov-Hausdorff distance  \linebreak between $A_{\infty}$ categories 1]{
Gromov-Hausdorff distance between filtered $A_{\infty}$ categories 1: 
Lagrangian Floer theory}
\author{Kenji Fukaya}

\thanks{Kenji Fukaya is supported partially by Simons Collaboration on homological Mirror symmetry.
}

\begin{abstract}
In this paper we introduce and study a distance, Gromov-Hausdorff distance, 
which measures how two filtered $A_{\infty}$ categories 
are far away each other.
\par
In symplectic geometry the author associated a 
filtered $A_{\infty}$ category, Fukaya category, to a finite set 
of Lagrangian submanifolds.
The Gromov-Hausdorff distance then gives a new 
invariant of a finite set of Lagrangian submanifolds.
\par
One can estimate it by the Hofer distance of 
Hamiltonian diffeomorphisms needed to send 
one Lagrangain submanifold to the other.
\par
A motivation to introduce Gromov-Hausdorff distance
is to obtain a certain completion of Fukaya category.
If we have a sequence of sets of Lagrangian submanifolds, 
which is a Cauchy sequence in the sense of Hofer metric,
then the associated filtered $A_{\infty}$ categories 
also form a Cauchy sequence in Gromov-Hausdorff distance.
In this paper we develop a theory to obtain an 
inductive limit of such a sequence of 
filtered $A_{\infty}$ categories.
In other words, we  give an affirmative answer to  
\cite[Conjecture 15.34]{FuFu6}.
\par
We can use this result to associate a 
filtered $A_{\infty}$ category to a finite set of 
Lagrangian submanifolds, {\it without} assuming transversality 
conditions between Lagrangian submanifolds. 
\par
The author believes that the story of Gromov-Hausdorff distance 
between filtered $A_{\infty}$ categories will play an important role 
in the study of homological Mirror symmetry over Novikov {\it ring}.
Recently homological Mirror symmetry over Novikov field is much studied.
However there is practically no study of homological Mirror symmetry over Novikov ring
at this stage.
\end{abstract}
\maketitle

\date{Feb., 2021}


\maketitle
\newpage
\tableofcontents
\newpage

\section{Introduction}
\label{sec;intro}

Let $(X,\omega)$ be a symplectic manifold that is compact or tame.
We consider a finite set $\mathbb L$ of relatively 
spin (embedded and compact) Lagrangian submanifolds. 
When we assume that, for $L,L' \in \mathbb L$ with $L\ne L'$,
$L$ is transversal to $L'$,
we obtain a filtered $A_{\infty}$ category 
$\frak{F}(\mathbb L)$ whose 
object is a pair $(L,b)$ such that $L \in \mathbb L$
and $b$ is a bounding cochain of $L$ in the sense of 
\cite[Definition 3.6.4]{fooobook}.
(\cite{AFOOO,FuFu6}.)

In this paper we introduce a way to measure 
how two filtered $A_{\infty}$ categories 
are close each other.
It is an analogue of 
Gromov-Hausdorff distance \cite{greenbook} 
between metric spaces and we call 
it Gromov-Hausdorff distance between filtered 
$A_{\infty}$ categories.
We denote by
$
d_{\rm GH}(\mathscr C_1,\mathscr C_2)
$ 
the Gromov-Hausdorff distance 
between filtered $A_{\infty}$ categories 
$\mathscr C_1$ and $\mathscr C_2$.
There are a few different versions of Gromov-Hausdorff distance.
See Definitions \ref{GromovHausdorff}, \ref{weakGH}, \ref{GHinftygapped},
and \ref{inftywaekGH}.
\par
Gromov-Hausdorff distance satisfies the triangle 
inequality.
It is actually a pseudo-distance. 
Namely 
$d_{\rm GH}(\mathscr C_1,\mathscr C_2) = 0$ 
may not 
imply that $\mathscr C_1$ is isomorphic to $\mathscr C_2$.
\par
We prove (in Corollaries \ref{corghchikachika} and \ref{corghchikachik2a}) that 
if $\mathbb L = \{L_1,\dots,L_K\}$ is a finite set of 
relatively spin Lagrangian submanifolds which are 
mutually transversal 
and if $\varphi_i$, $i=1,\dots,K$ is a 
finite set of Hamiltonian diffeomorphisms
then
$$
d_{\rm GH}(\frak F(L_1,\dots,L_K),\frak F(\varphi_1(L_1),\dots,
\varphi_K(L_K)))
\le
\max\{ d_{\rm Hof}({\rm id},\varphi_i) \mid i=1,\dots, K\},
$$
where $d_{\rm Hof}$ is  Hofer distance \cite{hofer}.
In other words, the map which associates a filtered $A_{\infty}$ category
$\frak F(L_1,\dots,L_K)$ to $L_1,\dots,L_K$ is 1-Lipschitz with respect 
to Gromov-Hausdorff distance and the Hofer distance 
of $L_i$.
A similar 1-Lipschitz property was proved
for Floer homology in \cite{fooobook,fooo:bidisk}.
In this paper we enhance it so that multiplicative 
structures are included.

\par
One motivation for the author to introduce
Gromov-Hausdorff distance 
between filtered $A_{\infty}$ categories 
is to extend the definition of 
$\frak F(L_1,\dots,L_n)$ 
to the case when $L_i$ is a limit of Lagrangian submanifolds 
with respect to Hofer distance.
\begin{defn}
For a positive integer $K$ and a symplectic manifold $(X,\omega)$ 
that is compact or tame, let $\mathfrak{Lag}^K$ be the set of 
all $k$-tuples $\mathbb L = (L_1,\dots,L_K)$ of immersed relatively spin Lagrangian submanifolds 
of $X$, such that $L_i$ has clean self intersection.
\footnote{More precisely we fix a back ground datum of $X$ (that is, a real vector 
bundle $V$ on a 3 skelton of $X$ such that $w^2(V) = w^2(TX)$, where 
$w^2$ stands for Stiefel-Whiteny class.) and 
$L_i$ is relatively spin with respect to it.  See \cite[Definition 8.1.2]{fooobook2}.} 
\par
We denote by $\mathfrak{Lag}^K_{\rm trans}$ its subset
of elements $(L_1,\dots,L_K) \in \mathfrak{Lag}^K_{\rm trans}$ 
such that $L_i$ is transversal to $L_j$ for all $i,j$ with $L_i\ne L_j$.
\end{defn}
For $\mathbb L \in \mathfrak{Lag}^k_{\rm trans}$
the filtered $A_{\infty}$ category $\frak F(\mathbb L)$ 
is defined (\cite{AFOOO,AJ,FuFu6,fooobook}) such that its object is a pair
$(L,b)$, where $L \in \mathbb L$ and $b$ is its bounding cochain.
The next definition is a variant of one in \cite{chekanov2}.
\begin{defn}
When $\mathbb L^1,\mathbb L^2 \in \mathfrak{Lag}^K$ we define the Hofer distance
between them $d_{\rm Hof}(\mathbb L^1,\mathbb L^2)$ to be the 
infimum of the positive numbers $\epsilon$ with the following 
properties.  There exists $k$-tuple of Hamiltonian diffeomorphisms $\varphi_i$,
$i=1,\dots,k$ such that:
\begin{enumerate}
\item $\varphi_i(L^1_i) = L^2_i$.
\item $d_{\rm Hof}(\varphi_i,{\rm identity}) < \epsilon$.
\end{enumerate}
Here $d_{\rm Hof}$ is the Hofer metric (\cite{hofer}) of the group of Hamiltonian diffeomorphisms.
\end{defn}
\begin{thm}\label{Lagcatecomplete}
For a Cauchy sequence $\{\mathbb L^n\}$ of elements of $\mathfrak{Lag}^K_{\rm trans}$
with respect to  Hofer distance, the sequence $\{\frak F(\mathbb L^n)\}$
of gapped unital filtered $A_{\infty}$ categories 
converges to a unital completed DG-category in Gromov-Hausdorff infinite distance.
The limit is unique up to equivalence. 
\end{thm}
Let $\overline{\mathfrak{Lag}^K}$ be the completion of $\mathfrak{Lag}^K$
with respect to the Hofer metric.
\begin{cor}\label{cor144}
To an element $\mathbb L$ of $\overline{\mathfrak{Lag}^K}$ we can associate a unital completed DG-category
$\frak F(\mathbb L)$ such that $\mathbb L \mapsto \frak F(\mathbb L)$ is 
$1$-Lipschitz 
with respect to the Hofer distance and Gromov-Hausdorff infinite distance.
\end{cor}
The proof of Theorem \ref{Lagcatecomplete} and Corollary \ref{cor144} 
is in Section \ref{sec;compgeo}.
The notion of a Gromov-Hausdorff infinite distance is defined in 
Section \ref{Hofinfhomoto}.\footnote{It is a stronger version 
of Gromov-Hausdorff distance which contains 
an extra information related to the unitality of 
filtered $A_{\infty}$ functors.}
\par
A competed DG-category is a filtered DG-category $\mathscr C$ together with 
an additional structure which gives a 
subcomplex $\overset{\circ}{\mathscr C}(a,b)$
of a morphism complex ${\mathscr C}(a,b)$ 
such that ${\mathscr C}(a,b)$ is its completion.
(See Definition \ref{degb219}.)
We will use this extra structure to avoid 
certain pathological phenomena which could occur 
when Floer homology is not finitely generated.
\par
The equivalence 
is defined in Definition \ref{inftywaekGH}.
If $\mathscr C$ is equivalent to $\mathscr C'$
then it implies, for example, that there exists an isometry 
$I: \frak{OB}(\mathscr C)  \to \frak{OB}(\mathscr C')$
between metric spaces of objects,
such that $H(\mathscr C(a,b),\frak m_1)$ is almost isomorphic\footnote{Definition 
\ref{defn113}.} 
to $H(\mathscr C'(I(a),I(b)),\frak m_1)$ (Theorem \ref{them123333}).
\par
In \cite[Section 6.5.4]{fooobook}  we generalized 
the definition of  Floer cohomology 
$HF((L_1,b_1),(L_2,b_2);\Lambda_0)$
to the case when the intersection of $L_1$ and $L_2$ 
is not transversal. 
(We do not assume it is clean either. Namely 
there is {\it no} assumption on the way how $L_1$ and $L_2$
intersect.)
Corollary \ref{maintheorem1} generalizes this construction to a construction 
of a filtered $A_{\infty}$ category 
{\it without} assuming tranversality.

\begin{cor}\label{maintheorem1}
To $\mathbb L \in \mathfrak{Lag}^K$, we can associate a unital completed DG-category 
$\frak{F}(\mathbb L)$ with the following properties.
\begin{enumerate}
\item
An object $\frak{OB(\frak{F}(\mathbb L))}$
is a pair $(L,b)$ such that $L \in \mathbb L$
and $b$ is a bounding cochain of $L$.
\item
For objects $(L_1,b_1),(L_2,b_2)$ of $\frak{F}(\mathbb L)$,
the completed $\frak m_1$ cohomology group
${HF}((L_1,b_1),(L_2,b_2))$ 
of the morphism complex is 
almost isomorphic to the Floer cohomology
$HF((L_1,b_1),(L_2,b_2);\Lambda_0)$
defined in \cite[Section 6.5.4]{fooobook}.\footnote{
The notions appearing in this item will be defined below.}
\end{enumerate}
\par
The filtered $A_{\infty}$ category $\frak{F}(\mathbb L)$ is independent 
of the choices up to equivalence.
\end{cor}
Corollary \ref{maintheorem1}  is proved at the end of Section \ref{sec;compgeo}.
\begin{rem}
Note that it is important that we take the Novikov {\it ring}
$\Lambda_0$ as the coefficient ring of  Floer 
cohomology. (See for example \cite[Definition 2.1]{FuFu6} 
for the definition of the (universal) Novikov ring $\Lambda_0$.)
When the coefficient ring is the Novikov {\it field} 
$\Lambda$, this result was known already by the following reason.
The Floer cohomology $HF((L_1,b_1),(L_2,b_2);\Lambda)$ 
over $\Lambda$ is invariant under  Hamiltonian diffeomorphisms.
(This is  \cite[Theorem G (G4)]{fooobook}.)
Therefore we can replace $\mathbb L$ by 
$\{\varphi_L(L) \mid L \in \mathbb L\}$ by taking an appropriate Hamiltonian 
diffeomorphism $\varphi_L$ (which depends on $L$) for $L \in \mathbb L$
so that  $\varphi_L(L)$ is transversal to $\varphi_{L'}(L')$
for two elements $L,L' \in \mathbb L$.
\par
The Floer cohomology $HF((L_1,b_1),(L_2,b_2);\Lambda_0)$ 
is {\it not} invariant under Hamiltonian diffeomorphisms.
We also like to mention that the torsion part 
(the part which is killed by multiplying $T^{\alpha}$)
contains many important informations of Lagrangian submanifolds.
\end{rem}
\par
We also provide an example of 
$n$-tuples of Lagrangian submanifolds $L_1,\dots,L_n$ 
and $L'_1,\dots,L'_n$ of a certain symplectic manifold $X$
such that:
\begin{enumerate}
\item 
For each $i$, $L_i$ is Hamiltonian isotopic to $L'_i$.
\item
For each $i$, there exists symplectic diffeomorphism 
$\varphi_i : X \to X$ such that
$\varphi_i(L_j) = L'_j$ for $j\ne i$.
\item
$d_{\rm GH}(
\frak F(L_1,\dots,L_n),\frak F(L'_1,\dots,L'_n)) \ne 0$. 
\end{enumerate} 
See Theorem \ref{themexample}.
\par
To define the notion appearing in Corollary \ref{maintheorem1} (2) above, 
we first review a certain part of \cite[Section 6.5.4]{fooobook}.
\begin{defn}
Let $\vec a = (a_1,a_2,\dots)$
be a sequence such that
$a_i \in \Z_{\ge 0} \cup \{\infty\}$
and that $a_i \ge a_{i+1}$.
We then define a $\Lambda_0$ module $\Lambda_0(\vec a)$ by
$$
\Lambda_0(\vec a)
= \underset{i=1}{\widehat{\overset{\infty}{\bigoplus}}} \frac{\Lambda_0}{T^{a_i}\Lambda_0}.
$$
Here $\widehat{\bigoplus}$ is the $T$-adic completion of the direct sum.
We set 
${\Lambda_0}/{T^{a_i}\Lambda_0} = \Lambda_0$ if $a_i = \infty$
and 
${\Lambda_0}/{T^{a_i}\Lambda_0} = 0$ if $a_i = 0$.
\end{defn}
\par
For two Lagrangian submanifolds $L_1,L_2$ which are transversal 
and $b_1$, $b_2$ bounding cochains of $L_1,L_2$ respectively, it is proved in 
\cite[Theorem 6.1.20]{fooobook} that there exists $\vec a$ such that:
\begin{equation}\label{struFl}
HF((L_1,b_1),(L_2,b_2);\Lambda_0)
\cong \Lambda_0(\vec a).
\end{equation}
Moreover $a_i = 0$ except finitely many natural numbers $i$.
Furthermore, the following result is proved in \cite{fooobook,fooo:bidisk}. 
We assume $L_1$ is transversal to $L_2$. Let
$\varphi_1$, $\varphi_2$ be Hamiltonian diffeomorphisms 
such that $\varphi_1(L_1)$ is transversal to $\varphi_2(L_2)$.
We put
$$
\aligned
HF((L_1,b_1),(L_2,b_2);\Lambda_0)
&\cong \Lambda_0(\vec a),\\
HF((\varphi_1(L_1),(\varphi_1)_{*}(b_1)),(\varphi_2(L_2),(\varphi_2)_{*}(b_2));\Lambda_0)
&\cong \Lambda_0(\vec a').
\endaligned$$
Then we have: 
\begin{equation}\label{Lipshich}
\vert a_i - a'_i \vert \le d_{\rm Hof}(\varphi_1,\varphi_2).
\end{equation}
Now the definition of Floer cohomology in the case when $L_1$ is not transversal to $L_2$ 
is as follows.
We take a sequence of Hamiltonian diffeomorphisms $\varphi_n$ such that
$\varphi_n(L_1)$ is transversal to $L_2$ 
and that $\epsilon_n = d_{\rm Hof}(\varphi_n,\varphi_{n+1})$ 
satisfies $\sum \epsilon_n < \infty$.
We assume furthermore that $\lim_{n\to \infty}d_{\rm Hof}({\rm id},\varphi_n) = 0$.
\par
We put
$$
HF((\varphi_n(L_1),(\varphi_n)_{*}(b_1)),(L_2,b_2);\Lambda_0)
= \Lambda_0(\vec a^n).
$$
Then by (\ref{Lipshich}) $n \mapsto a^n_i$ is a Cauchy sequence.
We put
$a^{\infty}_i = \lim_{n\to \infty} a^n_i$  and define:
\begin{equation}\label{fingHF}
HF((L_1,b_1),(L_2,b_2);\Lambda_0): = \Lambda_0(\vec a^{\infty}).
\end{equation}
Note that infinitely many numbers $a^{\infty}_i$
can be non-zero. (See \cite[Example 6.5.40]{fooobook}.)
In such a case the Floer cohomology may not be finitely generated.
\begin{conj}
If $L_1,L_2$ are real analytic then the Floer cohomology 
(\ref{fingHF}) is finitely generated.
The filtered $A_{\infty}$ category $\frak{F}(\mathbb L)$ is gapped 
if all the elements of $\mathbb L$ are real analytic.
Moreover the `almost isomorphic' in Theorem \ref{maintheorem1} (2)
can be replaced by `isomorphic'.
\end{conj}
We now define the notion `almost isomorphic' appearing 
in Theorem \ref{maintheorem1} (2).
\begin{defn}\label{energyfilteredmodule}
Let $V$ be a $\Lambda_0$  module. 
Its {\it energy filter} is a system of  submodules
$\frak F^{\lambda}V \subseteq V$ for $\lambda \ge 0$ such that:
\begin{enumerate}
\item 
If $\lambda < \lambda'$ then $\frak F^{\lambda'}V \subseteq \frak F^{\lambda}V$.
\item
For $c>0$, $T^c\frak F^{\lambda}V \subseteq \frak F^{c+\lambda}V$.
\item 
$V$ is complete with respect to the topology induced by the filtration.
\end{enumerate}
We say $V$ is a filtered $\Lambda_0$ module if it has an energy filter.
\par
We consider a cochain complex $(C,d)$ over $\Lambda_0$.
We assume that $C$ is energy filtered and $d$ preserves the energy filtration.
We call such $(C,d)$ a {\it filtered cochain complex over $\Lambda_0$}.
\par
We say a filtered $\Lambda_0$ module $V$ is {\it zero energy generated} if the following holds.
We put $V_{\Lambda}:= V \otimes_{\Lambda_0} \Lambda$.
We assume that there exists a filtration $\frak F^{\lambda}V_{\Lambda}$
for $\lambda \in \mathbb R$ such that items (1)(2)(3) above 
hold. Moreover we require:
\begin{enumerate}
\item[(4)] The $\Lambda_0$ module $V$ coincides with $\frak F^0V_{\Lambda}$.
The energy filter of  $V$ coincides with the restriction of one on $V_{\Lambda}$.
\end{enumerate}
Hereafter when we study filtered chain complex over $\Lambda_0$ we assume that 
it is zero energy generated, unless otherwise mentioned explicitly.
\end{defn}
The cohomology group of  filtered cochain complex over $\Lambda_0$ may not be  zero energy generated
since it has a torsion. However it has the following property.
\begin{defn}
A filtered $\Lambda_0$ module $V$ is said to be {\it divisible} if 
for any $v \in \frak F^{\lambda}(V)$ and $\lambda \ge \epsilon > 0$
there exists $w \in \frak F^{\lambda-\epsilon}(V)$ such that 
$T^{\epsilon}w = v$.
\end{defn}
\begin{defn}
For a divisible filtered $\Lambda_0$ module $V$ 
we define its {\it spectral dimension} $f(\cdot;V) : [0,\infty) \to \Z_{\ge 0} \cup \{\infty\}$
as follows.
We consider a finitely generated submodule $W$ of $\frak F^{\lambda}V$
and denote by $\#W$ the smallest number of generators of $W$.
We define $f(\lambda;V)$ to be the supremum of $\#W$
for all finitely generated submodules $W$ of $\frak F^{\lambda}V$.
\par
The function $f(\cdot;V)$ is non-increasing. 
\end{defn}
\begin{exm}\label{ex6.5.31}
\cite[Lemma 6.5.31]{fooobook} implies
$$
f(\lambda;\Lambda_0(\vec a))
= 
\#\{ i \mid a_i > \lambda\}.
$$
\end{exm}
\begin{defn}\label{defn113}
Let $V_1,V_2$ be energy  filtered $\Lambda_0$ modules.
We say $V_1$ is {\it almost isomorphic} to $V_2$ if 
for any $\lambda$
$$
\aligned
&\lim_{\epsilon>0, \epsilon \to 0} f(\lambda-\epsilon;V_1)
\ge f(\lambda;V_2),
\ge 
\lim_{\epsilon>0, \epsilon \to 0} f(\lambda+\epsilon;V_1),
\\
&\lim_{\epsilon>0, \epsilon \to 0} f(\lambda-\epsilon;V_2)
\ge f(\lambda;V_1)
\ge 
\lim_{\epsilon>0, \epsilon \to 0} f(\lambda+\epsilon;V_2).
\endaligned
$$
We remark that this condition is equivalent to the condition that 
$f(\lambda;V_1)=f(\lambda;V_2)$
for $\lambda \in (0,\infty)$ outside the discrete subset where  
$f(\cdot;V_1)$ and $f(\cdot;V_2)$ are discontinuous.
\end{defn}
Note that Example \ref{ex6.5.31} implies that 
$\Lambda_0(\vec a)$ is almost isomorphic to $\Lambda_0(\vec b)$
if and only if $\vec a = \vec b$.
Note however that $f(\lambda;\Lambda_+) \equiv 1 \equiv f(\lambda;\Lambda_0)$, 
for $\lambda > 0$.\footnote{$\Lambda_+$ is the maximal 
ideal of $\Lambda_0$.}
Therefore `almost isomorphic' does not imply `isomorphic' in general.
\par
Gromov-Hausdorff distance we introduce in this paper is a version with multiplicative
structures of the notion of a convergence of $\Lambda_0$ modules 
used in (\ref{fingHF}).
The author recently learned that 
an idea of the argument of the
proof of inequalities such as (\ref{Lipshich}) 
appeared in an earlier work by 
Ostrover \cite{Os}.
The structure theorem (\ref{struFl}) 
is of rather simple form since in Lagrangian 
Floer theory of \cite{fooobook}, 
a generator of Floer's cochain complex corresponding to an intersection 
point $[p]$ is in $\frak F^0$.
This is the zero-energy-generated-ness
of the Floer's cochain complex. 
In the generality studied in \cite{fooobook}
we do not know any other way.
However in a certain situation 
where action functional is well-defined, 
a natural way is to put the generator 
at $\frak F^{\lambda}$ where $\lambda$ is the 
value of the action functional.
In such a case the Floer cohomology 
has a structure called a persistent module.
Then (\ref{struFl}) becomes 
a normal form theorem of persistent modules
\cite{Ba}.
In such a situation a similar argument as 
\cite{fooobook} is used
by Usher \cite{Us}. Polterovich-Shelukhin \cite{PS} showed the continuity of 
barcodes associated to persistent modules in Floer 
homology of periodic Hamiltonian systems in the situation
when action functional is well-defined.
The idea to study the limit 
using the continuity, which was used 
in (\ref{fingHF}), 
appears in a recent work by 
Roux-Seyfaddini-Viterbo \cite{RSV},
where its important applications 
to Hamiltonian dynamics are given.
The estimate of the Hofer distance (of  
objects of a filtered $A_{\infty}$ category) 
by the Hofer-Chekanov's distance between 
Lagrangian submanifolds 
was proved in \cite{FuFu6} 
using the moduli space defined in \cite{fooobook}.
A similar estimate appears in \cite{BCS}.
Biran-Cornea-Shelvkin \cite{BCS} also
discuss its interesting relation to Lagrangian cobordism. 
\par
The explanation of the structure of the paper and the outline of the sections are in order.
In Section \ref{sec;limit} we define Gromov-Hausdorff distance.
In Section \ref{sec:Exam}
a simple example which shows non-triviality 
of Gromov-Hausdorff distance 
of filtered $A_{\infty}$ categories 
in Lagrangian Floer theory is given.
The algebraic part of the 
construction of the limit 
of a sequence of filtered $A_{\infty}$ categories 
is discussed in Sections 
\ref{sec;energyloss},
\ref{sec;bar},
\ref{sec;back},
and 
\ref{sec;speclo}.
Section \ref{sec;triangle} 
is devoted to the proof of 
the triangle inequality of 
Gromov-Hausdorff distance.
To show the well-defined-ness 
of the limit up to equivalence, 
we need to study unitality of the 
limit.
In Sections \ref{sec;alhomoequi} and 
\ref{sec;indhomoequi} we discuss 
algebraic part of the story of 
unitality of the limit.
Theorem \ref{mainalgtheorem2},
which is an algebraic 
story of inductive 
limit of a sequence of filtered $A_{\infty}$
categories, is proved in Section \ref{sec;indhomoequi}.
Sections \ref{Hofinfhomoto} provides 
an input needed to obtain 
{\it unital} inductive system 
of filtered $A_{\infty}$ categories.
We call it an infinite homotopy equivalence 
(between two objects of a filtered $A_{\infty}$
category).\footnote{See Definition \ref{infhomoequiv}.}
The existence of an infinite homotopy equivalence is 
proved in the situation of Lagrangian 
Floer theory in Section \ref{sec;hoferinf1}.
Section \ref{sec;spedim}
gives a relation between 
spectral dimension and Gromov-Hausdorff 
distance which is used to prove 
Theorem \ref{maintheorem1}
(2).
The proof of Theorem \ref{Lagcatecomplete}
is completed in Section \ref{sec;compgeo}
assuming the results established in Section 
\ref{sec;hoferinf1}.
Section \ref{sec;more} 
discuss a generalization of 
Theorems \ref{Lagcatecomplete}
where `a finite set of Lagrangian submanifolds'
will be replaced by 
`a separable subspace of the 
completion of the set of Lagrangian submanifolds'.
\par
We remark that Sections \ref{sec;compgeo}, \ref{sec;hoferinf1}
are the parts where various techniques and results 
developed to study Lagrangian Floer theory,
such as virtual fundamental chain, is applied.
Sections \ref{sec;limit}-\ref{sec;spedim}
(except Section \ref{sec:Exam}, which is elementary)
are of algebraic nature.
The discussion there is mostly 
self-contained except we use 
freely various results on 
homological algebra of filtered $A_{\infty}$
categories, which are obtained and explained in \cite{fu4,fooobook,FuFu6}.
The author thinks that this paper shows that homological 
algebra of filtered $A_{\infty}$ categories over $\Lambda_0$ is 
much richer that homological algebra of $A_{\infty}$ categories
over a filed.

\section{Gromov-Hausdorff distance between filtered $A_{\infty}$ categories.}
\label{sec;limit}

In this section we introduce the notion of Gromov-Hausdorff distance.
We work on the Novikov ring $\Lambda_0^R$ over a certain ground ring 
$R$.\footnote{We follow \cite[Section 2]{FuFu6}
for the notations related to the homological algebra of 
filtered $A_{\infty}$ categories.} 
In this paper we assume that the ground ring $R$ is always a field
of characteristic $0$.
When we apply it to Lagrangian Floer theory 
we put $R = \R$ since the constructions in Floer theory 
we developed  mostly use the de Rham model.  (We can work over 
the ring of integers or of rational numbers using the singular homology 
model but then the exact unit is hard to obtain.)
We hereafter omit $R$ and write $\Lambda_0$ etc. in place of $\Lambda_0^R$.
The morphism complex of a filtered $A_{\infty}$ category is 
either $\Z$ graded or $\Z/2N$ graded for a certain positive integer $N$.
\par
For an element $x$ of a filtered $\Lambda_0$ module $V$
we define its {\it valuation} $\frak v(x)$ 
to be the infimum of the numbers $\lambda$ 
such that $x \in \frak F^{\lambda}V$.
If $V$ is zero energy generated or divisible 
then $x \in \frak F^{\frak v(x)}V$.
\par
From now on in this paper we use the following convention:
the `curvature' $\frak m_0$ is always zero.
\par
We recall:
\begin{defn}
Let $\mathscr C$ be a filtered $A_{\infty}$ category.
We say $\{{\bf e}_c\}$ for ${\bf e}_c \in \mathscr C(c,c)$ with $\deg {\bf e}_c = 0$
to be a {\it strict unit}, if the following holds.
\begin{enumerate}
\item 
$\frak m_2({\bf e}_c,y) = y$, $(-1)^{\deg x}\frak m_2(x,{\bf e}_c) = x$ 
for $y \in \mathscr C(c,b)$, $x \in \mathscr C(a,c)$.
\item
$\frak m_1({\bf e}_c) = 0$. $\frak m_k(\dots,{\bf e}_c,\dots) = 0$
for $k\ge 3$.
\item $\frak v({\bf e}_c) = 0$.
\end{enumerate}
\end{defn}
We sometimes say unit in place of strict unit.
\par
For a filtered $A_{\infty}$ category $\mathscr C$ 
we  define an $A_{\infty}$ category $\mathscr C_{\Lambda}$
such that 
the sets of its objects is the same as $\mathscr C$ and that for objects $a,b$
$$
\mathscr C_{\Lambda}(a,b) = \mathscr C(a,b) \otimes_{\Lambda_0} \Lambda.
$$
\par
We  recall the next definition.
\begin{defn}{\rm (\cite[Definition 15.2]{FuFu6})}\label{defhoferdist}
Let $\mathscr C$ be a unital filtered $A_{\infty}$ category.
Let $c_1,c_2$ be ojects of $\mathscr C$.
We define the {\it Hofer distance}  $d_{\rm Hof}(c_1,c_2)$
between them to be the infimum of the positive numbers $\epsilon$ 
such that the following holds.
\begin{enumerate}
\item
There exist $t_{12} \in \mathscr C_{\Lambda}(c_1,c_2)$,
$t_{21} \in \mathscr C_{\Lambda}(c_2,c_1)$ 
of degree $0$ and
$s_{1} \in \mathscr C_{\Lambda}(c_1,c_1)$,
$s_{2} \in \mathscr C_{\Lambda}(c_2,c_2)$
of degree $-1$
such that
\begin{enumerate}
\item $\frak m_2(t_{21},t_{12})  + \frak m_1(s_2)= {\bf e}_{c_2}$.
\item $\frak m_2(t_{12},t_{21})  + \frak m_1(s_1)= {\bf e}_{c_1}$.
\item $\frak m_1(t_{21}) = 0$. $\frak m_1(t_{12}) = 0$. 
\end{enumerate}
\item
We require
$\frak v(t_{21}) > -\epsilon_1$,
$\frak v(t_{12}) > -\epsilon_2$
where 
$\epsilon_1,\epsilon_2$ are positive numbers with 
$\epsilon_1 + \epsilon_2 \le \epsilon$.
We also require
$\frak v(s_{1}) > -\epsilon$,
$\frak v(s_{2}) > -\epsilon$.
\end{enumerate}
We call $(t_{12},t_{21},s_{1},s_2)$ satisfying (1) 
a {\it homotopy equivalence} between $c_1,c_2$ 
and $(\epsilon_1,\epsilon_2)$ its {\it energy loss}. 
\end{defn}
Note that the Hofer distance between objects may be infinite.
We also remark that $d_{\rm Hof}(c_1,c_2)=0$ 
may not imply that $c_1 = c_2$.
So  Hofer distance is actually a pseudo-distance.\footnote{
Hofer distance between Hamiltonian diffeomorphisms 
defines a metric. This is an important result by Hofer.}
In the case of  $A_{\infty}$ categories 
in Lagrangian Floer theory (the transversal or clean case)
$d_{\rm Hof}((L_1,b_1),(L_2,b_2)) = 0$ implies that $L_1 =L_2$.
See \cite[Proposition 15.6]{FuFu6}.
\begin{rem}\label{Rem23232}
Let $(t_{12},t_{21},s_{1},s_2)$ be a homotopy equivalence 
with energy loss $(\epsilon_1,\epsilon_2)$.
We may replace $t_{12},t_{21}$ by
$T^{\delta}t_{12},T^{-\delta}t_{21}$ and may assume 
$\epsilon_1 = \epsilon_2 = \epsilon/2$.
In such a case we say 
energy loss of $(t_{12},t_{21},s_{1},s_2)$ is $\epsilon$.
\end{rem}
\begin{lem}\label{Hofertriangle}
For three objects $c_1,c_2,c_3$ of $\mathscr C$ the 
triangle inequality:
$
d_{\rm Hof}(c_1,c_2) + d_{\rm Hof}(c_2,c_3)
\ge d_{\rm Hof}(c_1,c_3)
$ holds.
\end{lem}
\begin{proof}
Let $(t_{12},t_{21},s_1, s_2)$ 
(resp. $(t_{23},t_{32},s'_2, s_3)$) be a homotopy equivalence 
between $c_1$ and $c_2$ (resp. $c_2$ and $c_3$) such that its 
energy loss is 
$\epsilon_{12}$
(resp. 
$\epsilon_{23}$).
We put
$
t_{13} = \frak m_2(t_{12},t_{23})$,  $t_{31} = \frak m_2(t_{32},t_{21}) 
$ and 
$$
s'_1 = s_1 - \frak m_2(t_{12},\frak m_2(s_2,t_{21})) + 
\frak m_3(t_{12},t_{23},\frak m_2(t_{32},t_{12})) 
+ \frak m_2(t_{12},\frak m_3(t_{23},t_{31},t_{12})).
$$
It is easy to see
$
\frak m_2(t_{13},t_{31})  + \frak m_1(s'_1) = {\bf e}_{c_1}
$
and $\frak v(s'_1) \ge -  (\epsilon_{12}+\epsilon_{23})$.
We can calculate in the same way for $\frak m_2(t_{31},t_{13})$.
The proof of Lemma \ref{Hofertriangle} is complete.
\end{proof}
We will prove in Corollary \ref{Hofdishomotopyinv} that Hofer distance 
is invariant under homotopy equivalence of 
filtered $A_{\infty}$ categories.
\par
Let $A,B$ be subsets of a metric space $(Y,d)$.
We recall that the Hausdorff distance $d_{\rm H}(A,B)$ between
$A$ and $B$ is the infimum of  positive numbers $\epsilon$
such that for each $p \in A$ (resp. $q \in B$) there exists 
$q \in B$  (resp. $p \in A$) such that $d(p,q) < \epsilon$.
\par
The gapped condition introduced in \cite[Definition 3.2.26]{fooobook} 
plays an important role in the
homological algebra of filtered $A_{\infty}$ categories.
We write $\mathscr C(a,b) = \overline{\mathscr C}(a,b) \,\,\widehat{\otimes}_R\,\, \Lambda_0$.
The structure operations $\frak m_k$ are written as 
$
\frak m_k =  \sum_{i=1}^{\infty} T^{\lambda_i} \overline{\frak m}_{k,i}
$
where the maps $\overline{\frak m}_{k,i}$ are $R$ linear maps between 
tensor products of $\overline{\mathscr C}(c_i,c_j)$ (over the ground ring $R$) and real numbers
$\lambda_i$ are non-negative and converge to $+\infty$
as $i$ goes to $\infty$.
A filtered $A_{\infty}$ category is called {\it gapped} if 
there exists a discrete (additive) sub-monoid $G$ of $\R_{\ge 0}$ such that 
$\lambda_i$ appearing in the structure operations is always an element of $G$.
The gapped-ness of filtered $A_{\infty}$ functors etc. is defined in the same way. 
\par
Gromov's compactness theorem implies that if $\mathbb L$ is a finite 
set of mutually transversal or clean Lagrangian submanifolds then $\frak F(\mathbb L)$
is gapped.

\begin{defn}\label{GromovHausdorff}
Let $\mathscr C_1$, $\mathscr C_2$ be unital 
and gapped filtered $A_{\infty}$ categories.
We define the {\it Gromov-Hausdorff distance} $d_{\rm GH}(\mathscr C_1,\mathscr C_2)$
between them as follows. The inequality $d_{\rm GH}(\mathscr C_1,\mathscr C_2) < 
\epsilon$
holds if and only if there exists a unital and gapped filtered $A_{\infty}$ category $\mathscr C$ such that:
\begin{enumerate}
\item The set of objects $\frak{OB}(\mathscr C)$ is the disjoint union of $\frak{OB}(\mathscr C_1)$
and $\frak{OB}(\mathscr C_2)$.
\item  There exists a unital and gapped homotopy equivalence from the  filtered $A_{\infty}$ category $\mathscr C_1$
(resp. $\mathscr C_2$)
to the full subcategories of $\mathscr C$ so that it is the inclusion $\frak{OB}(\mathscr C_1) 
\to \frak{OB}(\mathscr C)$ (resp. $\frak{OB}(\mathscr C_2) 
\to \frak{OB}(\mathscr C)$) for objects.
\item  The Haudsorff distance between $\frak{OB}(\mathscr C_1)$ and
$\frak{OB}(\mathscr C_2)$ in $\frak{OB}(\mathscr C)$ (with respect to the Hofer distance)
is smaller than $\epsilon$.
\end{enumerate}
\end{defn}
\begin{rem}
Items (1)(3) imply that the Gromov-Hausdorff distance (\cite[Definition 3.4]{greenbook}) between 
two (pseudo) metric spaces $(\frak{OB}(\mathscr C_1),d_{\rm Hof})$
and $(\frak{OB}(\mathscr C_2),d_{\rm Hof})$ is not greater than
$d_{\rm GH}(\mathscr C_1,\mathscr C_2)$.
\end{rem}
\begin{exm}
Let $\mathscr C$ be a unital and gapped filtered $A_{\infty}$ category and 
$A,B$ two subsets of $\frak{OB}(\mathscr C)$.
Let $\mathscr C(A)$, $\mathscr C(B)$ be the full subcategories of $\mathscr C$
the set of whose objects are $A$, $B$, respectively.
Then $d_{\rm GH}(\mathscr C(A),\mathscr C(B)) \le d_{\rm H}(A,B)$.
\end{exm}
\begin{defn}\label{weakequid}
Two unital and gapped filtered $A_{\infty}$ categories $\mathscr C_1$, $\mathscr C_2$ 
are said to be {\it weakly equivalent} if there exists
a unital and gapped filtered $A_{\infty}$ category $\mathscr C$
such that (1)(2) of Definition \ref{GromovHausdorff} hold 
and 
\begin{enumerate}
\item[(3)']  The Haudsorff distance between $\frak{OB}(\mathscr C_1)$ and
$\frak{OB}(\mathscr C_2)$ in $\frak{OB}(\mathscr C)$ (with respect to  Hofer distance)
is $0$.
\end{enumerate}
\end{defn}
This is slightly stronger that $d_{\rm GH}(\mathscr C_1,\mathscr C_2) = 0$.

\begin{thm}\label{triangleGH}
Gromov-Hausdorff distance satisfies 
the triangle inequality.
Moreover the weak equivalence between gapped and unital filtered $A_{\infty}$
categories is an equivalence relation.
\end{thm}
We will prove Theorem \ref{triangleGH} in Section \ref{sec;triangle}.

The filtered $A_{\infty}$ category which appears in Lagrangian 
Floer theory is gapped when the intersection of Lagragian submanifolds 
are transversal or clean. When we take the limit 
or in the case intersection may not be clean, 
the gapped-ness may not hold.
In such a case various basic results on homological 
algebra of filtered $A_{\infty}$ categories do not hold.
For example the proof of Theorem \ref{triangleGH} we will give 
in Section \ref{sec;triangle}
does not work.
\par
We will modify the definition of Gromov-Hausdorff distance 
and weak equivalence below so that we can  prove 
the triangle inequality without assuming gapped-ness.
We first introduce a few notions.
\begin{defn}
Let $V,W$ be divisible $\Lambda_0$ modules and 
$\varphi: V\to W$  a filtered $\Lambda_0$ homomorphism.
\par
We say that $\varphi$ is {\it almost surjective}
if for any $x \in W$ and $\epsilon > 0$ there exists 
$y \in  V$ such that $\varphi(y) = T^{\epsilon}x$.
We say that $\varphi$ is {\it almost injective}
if for each $y \in V$ with $\varphi(y) = 0$,
and $\epsilon > 0$, $T^{\epsilon}y = 0$.
We say $\varphi$ is an {\it almost isomorphism} if 
$\varphi$ is almost surjective and almost injective.
\end{defn}
\begin{exm}
Let $C: = \Lambda_0/\Lambda_+$. The zero homomorphism from $C$ to $0$
is almost injective but is not injective.
\par
Let $C_1$ be the $T$-adic completion of the direct sum 
$(\Lambda_0)^{\mathbb N}$ and $C_2 := \Lambda_0$.
We define $\varphi: C_1 \to C_2$ by
$
\varphi(x_1,\dots,x_n,\dots) = \sum_{i=1}^{\infty} T^{1/i}x_i.
$
Then $\varphi$ is almost surjective but is not surjective.
(In fact the image is $\Lambda_+$.)
\end{exm}
The next lemma is not difficult to prove. We will prove it in Section \ref{sec;spedim}.
\begin{lem}\label{lem244}
If $\varphi: V\to W$ is an almost isomorphism between 
divisible filtered $\Lambda_0$ modules, then 
$V$ is almost isomorphic to $W$, in the sense of Definition \ref{defn113}. 
\end{lem}
\begin{defn}\label{defn216}
Let $(C,d)$, $(C_1,d_1)$, $(C_2,d_2)$ be filtered cochain complexes over $\Lambda_0$.
(Note that we always assume that they are zero energy generated.)
\begin{enumerate}
\item
We say a pair $(C,C_0)$ is a {\it completed complex}
if $C = (C,d)$ is a filtered $\Lambda_0$ complex and 
$C_0 \subset C$ is its subcomplex such that $C$ is a completion of 
$C_0$ with respect to the $T$-adic topology.
We call $C_0$ the {\it finite part} of $C$.
We require $C_0 \subset C_0 \otimes_{\Lambda_0} \Lambda: 
= (C_0)_{\Lambda}$ 
and $C_0 = \frak F^0(C_0)_{\Lambda}$.
\item
Let $(C_i,C_{i,0})$ be a completed complex
for $i=1,2$.
Let $\frak i: (C_1,C_{1,0}) \to (C_2,C_{2,0})$
be an injective filtered cochain homomorphism.
We say $(C_1,C_{1,0})$ is a {\it completed  subcomplex} of $(C_2,C_{2,0})$ 
and $\frak i$ a {\it completed  injection},
if $\frak i(C_{1,0}) \subseteq C_{2,0}$ and $\frak v(\frak i(x)) = \frak v(x)$.
\footnote{This implies that there is a 
splitting $C_{2,0} \cong C_{1,0}
\oplus C^*_{2,0}$ as $\Lambda_0$ module and $C^*_{2,0}$ 
is zero energy generated.}
\end{enumerate}
\end{defn}
\begin{defn}
Let $(C_1,C_{1,0})$, $(C_2,C_{2,0})$ be completed cochain complexes.
We say that a completed  injection $(C_1,C_{1,0}) \to (C_2,C_{2,0})$ is an 
{\it almost cochain homotopy equivalence} if 
for each finitely generated subcomplex $W$ of $C_{2,0}$ and $\epsilon > 0$
there exists  $G: W \to C_{1,0}$ (which is abbreviated as an $\epsilon$-$W$ 
{\it shrinker}) such that:
\begin{enumerate}
\item The image of $dG + Gd - T^{\epsilon}$ is in $C_{1,0}$.
\item  $G = 0$ on $C_{1,0}$.
\item  $G$ preserves energy filtration.
\end{enumerate}

Suppose that a 
finitely generated 
submodule $W^0 \subseteq C_{2,0}$ is zero energy generated. 
We put $W: = (W^0_{\Lambda} + d_2(W^0_{\Lambda})) \cap C_{2,0}$,
which is a finitely generated subcomplex of $C_{2,0}$.
We call an $\epsilon$-$W$ shrinker $G$ of $(C_1,C_{1,0}) \to (C_2,C_{2,0})$
an $\epsilon$-$W^0$ shrinker,
by an abuse of notation.
When $W^0 = \Lambda y \cap C_{2,0}$ we call it an 
$\epsilon$-$y$ shrinker.
\end{defn}

\begin{lem}\label{lemma217}
If $(C_1,C_{1,0})$ is a completed subcomplex of $(C_2,C_{2,0})$
and if the inclusion $(C_1,C_{1,0}) \to (C_2,C_{2,0})$ is an 
almost cochain homotopy equivalence
then the homomorphism $H(C_1) \to H(C_2)$ induced on 
cohomologies is an almost isomorphism. 
\end{lem}
\begin{proof}
We first prove $H(C_1) \to H(C_2)$ is almost surjective.
Let $x \in C_{2}$ with $dx = 0$.
We take $\lambda_i$ such that $\lambda_i \to \infty$.
\footnote{For example $\lambda_i = i$.}
Let $y_1 \in C_{2,0}$ such that $x - y_1 \in \frak F^{\lambda_1}C_2$.
Since $dx =0$ we have $dy_1 \in \frak F^{\lambda_1} C_2$.
We take $\epsilon-y_1$ shrinker $G_1$.
Then
$
T^{\epsilon}x - dG_1y_1 \in  \frak F^{\lambda_1}C_2 + C_{1,0}.
$
We put $T^{\epsilon}x - dG_1y_1 = z_1 + w_1$,
with $z_1 \in \frak F^{\lambda_1}C_2$ and $w_1 \in C_{1,0}$.
We take $y_2 \in \frak F^{\lambda_1}C_{2,0}$ such that 
$z_1 - y_2 \in \frak F^{\lambda_2}C_{2}$.
Since $dz_1 + dw_1 = 0$, there exists $a_2 \in C_{1,0}$ 
such that $dy_2 - a_2 \in \frak F^{\lambda_2}C_2$.
We take $\epsilon/2-\{y_2,a_2\}$ shrinker $G_2$.
Then
$
T^{\epsilon/2}z_1 - (dG_2+G_2d)y_2 \in \frak F^{\lambda_2}C_{2}
+ \frak F^{\lambda_1}C_{1,0}$.
Since $G_2(dy_2 - a_2) \in \frak F^{\lambda_2}C_2$ and $G_2 a_2 = 0$, 
we have $G_2dy_2 \in \frak F^{\lambda_2}C_2$.

We put $T^{\epsilon/2}z_1 - dG_2y_2 = z_2 + w_2$
with $z_2 \in \frak F^{\lambda_2}C_{2}$, 
$w_2 \in \frak F^{\lambda_1}C_{1,0}$.
Inductively we can find 
$z_k \in \frak F^{\lambda_k}C_{2}$, $y_k \in \frak F^{\lambda_{k-1}}C_{2,0}$,
$w_k \in \frak F^{\lambda_{k-1}}C_{1,0}$
such that
$$
z_{k+1} + w_{k+1} = T^{2^{-k}\epsilon}z_k - dG_{k+1} y_{k+1}.
$$
Therefore
$$
T^{2\epsilon}x = \sum_{i=1}^{\infty} T^{2^{1-i}\epsilon} w_i
+ \sum_{i=1}^{\infty}T^{2^{1-i}\epsilon} d G_i y_i.
$$
The first sum in the right hand side is a cycle in $C_1$
and the second sum in the right hand side is a boundary in $C_2$.
Therefore $[T^{2\epsilon}x]$ is in the image of $H(C_1) \to H(C_2)$
as required.
\par
We next prove that $H(C_1) \to H(C_2)$ is almost injective.
Let $x\in C_{1}$
and $dx = 0$.
Suppose $[x] = 0$ in $H_2(C_2)$.
There exists $y \in C_2$ such that $x = dy$.
We take $\lambda_i$ such that $\lambda_i \to \infty$.
Let $y_1 \in C_{2,0}$ such that $y - y_1 \in \frak F^{\lambda_1}C_2$.
Since $x - dy_1 \in \frak F^{\lambda_1}C_2$ there exists 
$a_1 \in C_{1,0}$ such that $a_1 - dy_1 \in \frak F^{\lambda_1}C_{2,0}$.
We take $\epsilon$-$\{y_1,a_1\}$ shrinker $G_1$.
Then
$$
T^{\epsilon}x = T^{\epsilon}d(y-y_1) + dG_1d y_1 + dw_1
$$
with $w_1 \in C_1$.
Note that $G_1dy_1 - G_1a_1 \in \frak F^{\lambda_1}C_2$ and $G_1a_1 = 0$
(since $a \in C_{1,0}$).
We put $z_1 = T^{\epsilon}(y-y_1) + G_1dy_1 \in \frak F^{\lambda_1}C_2$.
Let $y_2 \in \frak F^{\lambda_1}C_{2,0}$ 
with  $z_1 - y_2 \in \frak F^{\lambda_2}C_2$.
Since $dz_1 = T^{\epsilon}x - dw_1 \in C_1$, we can take $a_2 \in C_{1,0}$
such that $a_2 - dy_2 \in \frak F^{\lambda_2}C_{2,0}$.
We take $\epsilon/2$-$\{y_2,a_2\}$ shrinker $G_2$.
Then
$$
T^{\epsilon/2}dz_1 = T^{\epsilon/2}d(z_1-y_2) + dG_2d y_2 + dw_2
$$
with $w_2 \in \frak F^{\lambda_1}C_1$.
We put $z_2 = T^{\epsilon/2}(z_1-y_2) + G_2dy_2 \in \frak F^{\lambda_2}C_2$.
\par
Inductively we can find $z_k \in \frak F^{\lambda_k}C_2$,
$w_k \in \frak F^{\lambda_{k-1}}C_1$,
$y_k \in \frak F^{\lambda_{k-1}}C_2$
such that
$
T^{2^{-k}\epsilon}dz_{k-1} = dz_k + dw_k
$
and $z_{k} = T^{2^{1-k}\epsilon}(z_{k-1}-y_{k}) + G_{k}dy_{k}$.
Therefore
$$
T^{2\epsilon} x = d \left(
\sum_k T^{2^{1-k}\epsilon} w_k\right),
$$
as required.
\end{proof}
\begin{defn}\label{degb219}
A {\it completed DG-category} is a filterd DG-category $\mathscr C$
together with an assignment of $ \overset{\circ}{\mathscr C}(a,b)$ 
for each $a,b \in \frak{OB}(\mathscr C)$ such that 
$(\mathscr C(a,b),\overset{\circ}{\mathscr C}(a,b))$ is a completed 
complex and 
$
\frak m_2(\overset{\circ}{\mathscr C}(a,c) \otimes \overset{\circ}{\mathscr C}(c,b))
\subseteq
\overset{\circ}{\mathscr C}(a,c).
$
We call $\overset{\circ}{\mathscr C}(a,b)$ the {\it finite part} of $\mathscr C(a,b)$.
When $\mathscr C$ is unital we require that the unit is contained 
in the finite part.
\end{defn}
\begin{defn}\label{weakGH}
Let $\mathscr C_1$, $\mathscr C_2$ be unital 
and completed DG-categories.
We define the {\it Gromov-Hausdorff distance} $d_{\rm GH}(\mathscr C_1,\mathscr C_2)$
between them as follows. The inequality $d_{\rm GH}(\mathscr C_1,\mathscr C_2) < \epsilon$
holds if and only if there exist a unital completed DG-category 
$\mathscr C$ and a $DG$-functor $\mathscr C_i \to \mathscr C$ such that
(1)(3) of Definition \ref{GromovHausdorff} and 
the following hold.
\begin{enumerate}
\item[(2)'] 
For each $a,b \in \frak{OB}(\mathscr C_1)$ (resp. $a',b' \in \frak{OB}(\mathscr C_2)$)
we have:
\begin{enumerate}
\item[(a)]
$\mathscr C_1(a,b)$ (resp. $\mathscr C_2(a',b')$) is a completed subcomplex of 
$\mathscr C(a,b)$  (resp. $\mathscr C(a',b')$). 
\item[(b)]
The cochain maps 
$
(\mathscr C_1(a,b),\frak m_1) \to (\mathscr C(a,b),\frak m_1)$, 
$(\mathscr C_2(a',b'),\frak m_1) \to (\mathscr C(a',b'),\frak m_1)$
are almost cochain homotopy equivalences.
\end{enumerate} 
\end{enumerate} 
Two unital 
and completed DG categories $\mathscr C_1$, $\mathscr C_2$
are said to be {\it weakly equivalent} if 
there exists $\mathscr C_0$ such that (1) of Definition \ref{GromovHausdorff} 
and (2') above hold and 
$$
d_{\rm H}(\frak{OB}(\mathscr C_1),\frak{OB}(\mathscr C_2)) = 0.
$$
\par
A DG-functor $\mathscr C_1 \to \mathscr C_0$ is said 
to be an {\it almost homotopy equivalent injection} 
if (2)' above holds.
\end{defn}
\begin{thm}\label{triangleharder}
Gromov-Hausdorff distance of 
unital completed DG-categories satisfies the 
triangle inequality. 
The weak equivalence defined above is an equivalence relation.
\end{thm}
We will prove Theorem \ref{triangleharder} in Section \ref{sec;triangle}.
\par
\begin{lem}\label{isoembHFH}
If $\mathscr C \to \mathscr C'$ is an almost homotopy equivalent injection
between completed DG-categories, then 
the map $\frak{OB}(\mathscr C) \to \frak{OB}(\mathscr C')$
is an isometric embedding, with respect to  Hofer metric.
\end{lem}
We can use Lemma \ref{lemma217} to prove 
this lemma.
We omit the proof since in Section \ref{Hofinfhomoto}
we prove a stronger result  Theorem \ref{idometric}.
\par
The next result is the reason why the author named the notion in Definition \ref{defhoferdist}
to be the {\it Hofer} distance. 
Let $(X,\omega)$ be a symplectic manifold and 
$H: X\times  [0,1] \to \R$  a time dependent Hamiltonian function.
We denote its (time dependent) Hamiltonian vector field by
$\frak X_{H_t}$.
We take the one parameter group of transformations $\varphi^t_{H}$
generated by $\frak X_{H_t}$.  Namely we require
$
\frac{d\varphi^t_{H}}{dt} = \frak X_{H_t} \circ \varphi^t_{H},
$
and $\varphi^0_{H} = {\rm identity}.
$
Here we put $H_t(x): =H(x,t)$.
We assume that $H$ is normalized, that is, 
$
\int_{X} H(x,t) \omega^n = 0
$ (in the case when $X$ is compact), $\mathcal H =0$ outside a compact set (in the 
case when $X$ is non-compact).
We define:
\begin{equation}
E^-({H}) = \int_0^1 -\min \; {H}_t\, dt, \quad
E^+({H}) = \int_0^1 \max \; {H}_t\, dt,
\end{equation}
and $E(H) = E^-({H}) + E^+({H})$.
We put $\varphi = \varphi^1_{H}$. 
Let $(L,b)$ be a pair of a relatively spin Lagrangian submanifold
and its bounding cochain.
We assume $L$ is transversal to $\varphi(L)$ and
take a finite set $\mathbb L$ of relatively spin Lagrangian submanifolds containing $L$
and $\varphi(L)$ such that $L_1 \pitchfork L_2$ for $L_1, L_2 \in \mathbb L$, 
$L_1 \ne L_2$.
Let $\frak F(\mathbb L)$ be the filtered $A_{\infty}$ category 
whose object  is a pair $(L',b')$ such that $L' \in \mathbb L$ and $b$ is a 
bounding cochain of $L'$.
\begin{thm}\label{estimateHofLag}
The Hofer distance of $\frak F(\mathbb L)$
satisfies the inequality:
$$
d_{\rm Hof}((L,b),(\varphi(L),\varphi_*(b))) \le E({H}).
$$
\end{thm}
Theorem \ref{estimateHofLag} is \cite[Theorem 15.5]{FuFu6}. (The proof there is based on 
\cite[Section 5.3]{fooobook} and \cite{fooo:bidisk}.)
We will prove its stronger version Theorem \ref{estimateHofLaginf2} in Section \ref{sec;hoferinf1}.
We review the proof of Theorem \ref{estimateHofLag} in Section \ref{sec;hoferinf1}.

We remark that the Hofer distance between $\varphi$ and the identity map is 
the infimum of $E({H})$ for $H$ such that 
$\varphi^1_{H} = \varphi$.
\par
For a filtered $A_{\infty}$ category $\mathscr C$ and a subset 
$A$ of $\frak{OB}(\mathscr C)$ we denote by 
$\mathscr C(A)$ the full subcategory of $\mathscr C$ the set of 
whose objects is $A$.
\par
The next corollary is an immediate consequence of Theorem \ref{estimateHofLag}.

\begin{cor}\label{corghchikachika}
Let $\varphi_i: X \to X$ be Hamiltonian diffeomorphisms 
such that $d_{\rm H}(\varphi_i,{\rm id}) < \epsilon$ for $i=1,\dots,k$
and $L_1,\dots,L_k$ be relatively spin Lagrangian submanifolds.
Then
$$
d_{\rm GH}(\frak F(L_1,\dots,L_k), \frak F(\varphi_1(L_1),\dots,
\varphi_k(L_k))
< \epsilon.
$$
\end{cor}

\section{An example.}
\label{sec:Exam}

In this section, we prove Theorem \ref{themexample}, which shows 
a certain non-triviality of Gromov-Hausdorff distance we introduced in 
Section \ref{sec;limit}, in Lagrangian Floer theory.

\begin{thm}\label{themexample}
For each sufficiently small positive number $\epsilon$ and a natural number $n$ there exists a compact 
symplectic manifold $(X,\omega)$
and its spin Lagrangian submanifolds $L_1,\dots,L_n$, $L'_1,\dots,L'_n$
with the following properties.
\begin{enumerate}
\item The $1$ form $0$ is a bounding cochain of $L_1,\dots,L_n$, $L'_1,\dots,L'_n$.
\item The pair $(L_i,0)$ is Hamiltonian isotopic to $(L'_i,0)$.
\item For every $i \in \{1,\dots,n\}$, there exists a symplectic diffeomorphism
$\varphi_i : X \to X$ such that $\varphi_i(L_j) = L'_j$ for $j\ne i$. 
\item
$
d_{\rm GH} (\frak F((L_1,0),\dots,(L_n,0)),\frak F((L'_1,0),\dots,(L'_n,0))
= \epsilon.
$
\end{enumerate}
\end{thm}
\begin{proof}
Let $L_1,\dots,L_n$ be circles (smooth submanifolds diffeomorphic to $S^1$) in $S^2$ which intersect
transversally  each other and that there exists an $n$-gon 
$D_n^2$ that is a connected component of 
$S^2 \setminus \bigcup_iL_i$ such that
$\partial D_n^2$ is a union of arcs $\partial_i D_n^2$ 
$i=1,\dots,n$ and $\partial_i D_n^2 \subset L_i$.
(In particular the boundary of $\partial_i D_n^2$ 
comprises two points, the intersection points of $L_i \cap L_{i-1}$ and 
of $L_i \cap L_{i+1}$.
\par
We add handles to the union of the connected components of 
$S^2 \setminus  \bigcup_iL_i$
other than $D_n^2$ to obtain $\Sigma$ so that 
$\Sigma \setminus  \bigcup_iL_i$ 
is a union of two connected components, one is 
$D_n^2$ and the other is a bordered Riemann surface of higher genus.
We will make a particular choice for the way to add handles later so that 
(3) is satisfied.
\par
We define $L'_1$ by moving $L_1$ via a Hamiltonian isotopy 
as in Figure \ref{Figure1} below.
\begin{figure}[h]
\centering
\includegraphics[scale=0.6]{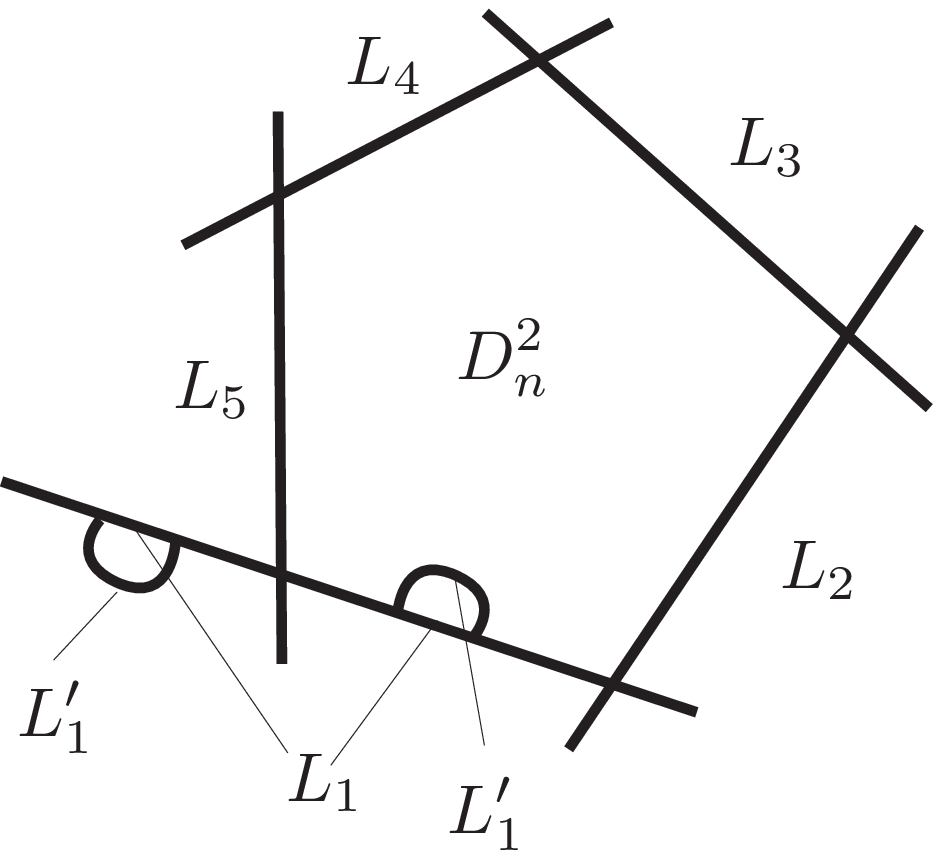}
\caption{$L'_1$}
\label{Figure1}
\end{figure}
We put $L_i = L'_i$ for $i\ne 1$.
We may take $L'_1$ so that (2) holds.
(1) is obvious from construction.
We observe that there is only one non-trivial 
operation of positive energy that is
$$
\frak m_{n-1} : 
CF(L_1,L_2) \otimes \dots \otimes CF(L_{n-1},L_n) 
\to CF(L_1,L_n)
$$ 
which is induced from the holomorphic $n$-gon $D_n^2$.
It gives
\begin{equation}\label{Anformula}
\frak m_{n-1}(p_{1,2},\dots,p_{n-1,n})
= T^{\lambda} p_{1,n}
\end{equation}
where $p_{i,i+1}$ is the generator of $CF(L_i,L_{i+1})$
corresponding to the intersection point that is a vertex 
of $D_n^2$ and $\lambda$ is the area of $D_n^2$.
\par
If we replace $L_1$ by $L'_1$ we obtain the same conclusion 
except (\ref{Anformula}) is replaced by
\begin{equation}\label{Anformula2}
\frak m_{n-1}(p_{1,2},\dots,p_{n-1,n})
= T^{\lambda'} p_{1,n}
\end{equation}
where $\lambda' < \lambda$.
This implies that the Gromov-Hausdorff distance
\begin{equation}\label{GHdist}
d_{\rm GH} (\frak F((L_1,0),\dots,(L_n,0)),\frak F((L'_1,0),\dots,(L'_n,0))
\end{equation} 
is non-zero. 
Moreover it is easy to see that 
(\ref{GHdist}) is a continuous function of $\lambda'$ 
(when we take a $\lambda'$ parametrized family of $L_1'$) and it converges to $0$ as 
$\lambda'$ converges to $\lambda$. 
(It is likely that (\ref{GHdist}) is $\lambda - \lambda'$.
We do not need to prove it to prove Theorem \ref{themexample}.)
Thus (4) holds.
\par
We finally prove (3).
Let us consider the case of the Figure \ref{Figure1} ($n=5$ in particular) 
and $i=3$.
We take domains $A_1,B_1,A_2,B_2$ as in Figure \ref{Figure2}.
\begin{figure}[h]
\centering
\includegraphics[scale=0.6]{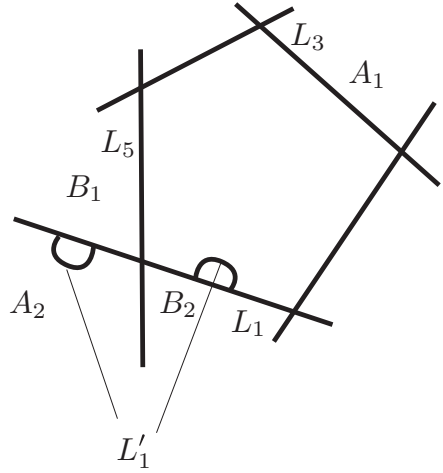}
\caption{$A_1,A_1,B_1,B_2$}
\label{Figure2}
\end{figure}
We may take $\Sigma$ so that there are 1-handles 
$H_1$, $H_2$ such that $H_i$ connects $A_i$ to $B_i$.
\par
We next take an annulus $S^1 \times [0,1]$ as follows.
It starts at the part of $D^2_n$ where we modify $L_1$ to $L'_1$. 
It crosses $L_3$ and enters into the domain $A_1$.
Then it passes through the handle $H_1$ to arrive in $B_1$.
It crosses $L_1$ at the part where we modify $L_1$ to $L'_1$ and 
enters $A_2$.
It passes through the handle $H_2$ to arrive in $B_2$.
It then comes back.
(See Figure \ref{Figure3}.)
\begin{figure}[h]
\centering
\includegraphics[scale=0.6]{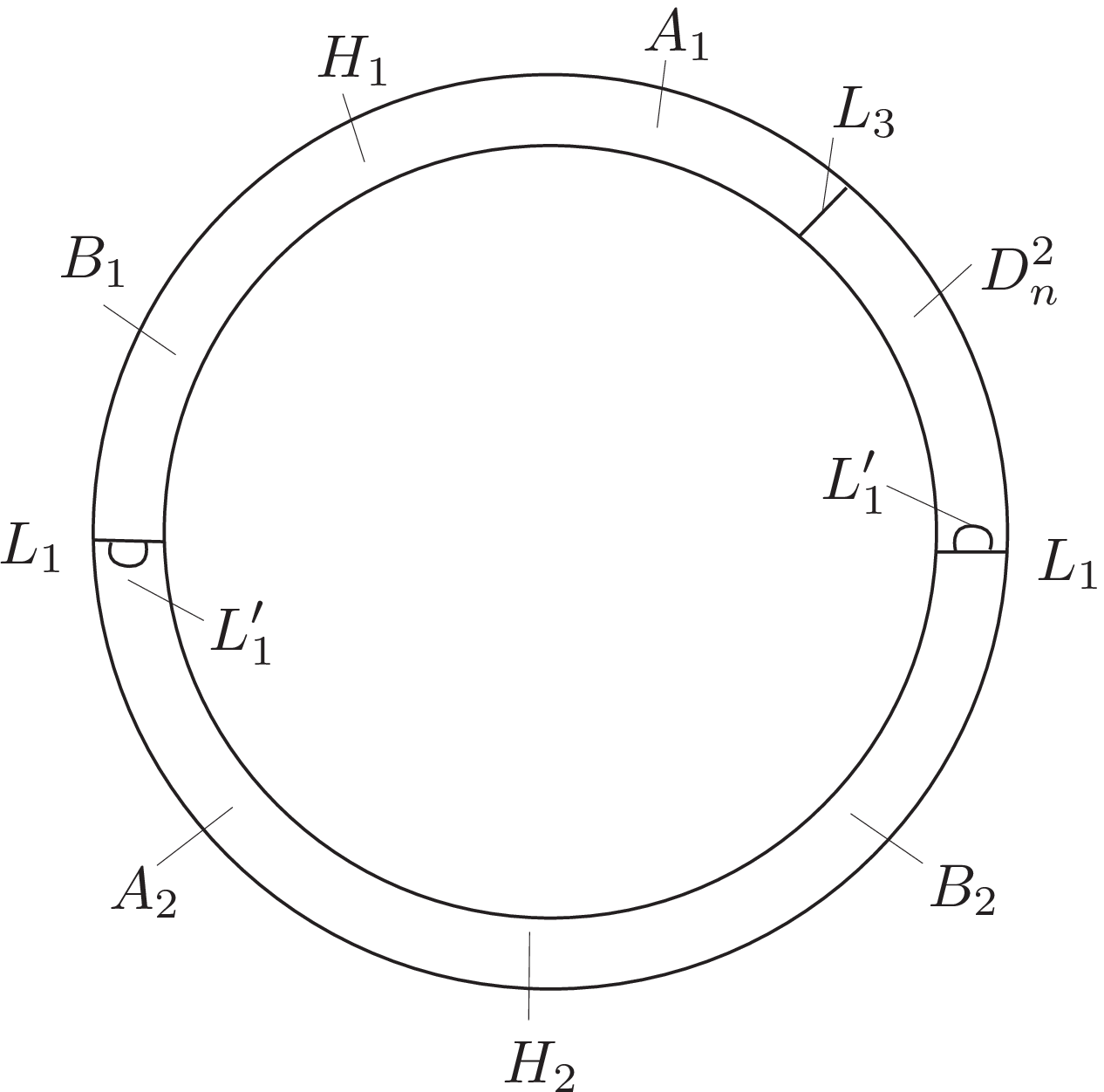}
\caption{Annulus}
\label{Figure3}
\end{figure}
\par
The intersection of the annulus $S^1 \times [0,1]$
with $L_1$ is a union of two arcs and it does not 
intersect with $L_i$ other than $L_1$ and $L_3$.
\par
Now it is easy to see that there exists a 
symplectic diffeomorphism which has a 
compact support in the interior of the annulus 
and which sends $L_1$ to $L'_1$.
The proof of Theorem \ref{themexample} is complete.
\end{proof}

\begin{rem}
We remark that in the example we gave above, 
which is the case of real 2 dimension, the property (4) 
can be checked very easily and  
the non-existence of $\varphi : \Sigma \to \Sigma$
suth that $\varphi(L_i) = L'_i$ for 
$i=1,\dots,n$ can be proved by an elementary method.
However we can replace $\Sigma$ by 
$\Sigma \times S^2$ and $L_i,L'_i$ by 
$L_i \times S^1$, $L'_i \times S^1$, respectively,
(Here $S^1$ is the equator. We take the same equator for $i=1,\dots,n$.)
the properties (1)(2)(3)(4) still hold.
In this $4$ dimensional case, to prove the 
non-existence of $\varphi : \Sigma \to \Sigma$ such that
$\varphi(L_i) = L'_i$ for 
$i=1,\dots,n$, 
it seems necessary to use 
a certain non-elementary method such as pseudo-holomorphic curves, (that is, 
the Floer theory).
\end{rem}

\section{Filtered $A_{\infty}$ functors with energy loss.}
\label{sec;energyloss}

Let us study the situation where
two filtered $A_{\infty}$ categories
are close to each other in Gromov-Hausdorff distance. 
We translate this assumption to another notion, that is, a
filtered $A_{\infty}$ functor with 
energy loss.
\par
Let $\mathscr C$ be a filtered $A_{\infty}$ category 
and $a,b$ its objects.
We define:
\begin{equation}\label{form24new}
\left\{
\aligned
&\overset{\circ}B_k\mathscr C[1](a,b) := \underset{a=c_0,c_1,\cdots,c_{k-1},c_k=b}{\bigoplus}
\mathscr C[1](c_0,c_1) \otimes \cdots \otimes \mathscr C[1](c_{k-1},c_k)
\\
&\overset{\circ}B\mathscr C[1](a,b) := \underset{k=1,2,\dots}{{\bigoplus}} \overset{\circ}B_k\mathscr C[1](a,b).
\endaligned
\right.
\end{equation}
Let 
$B\mathscr C[1](a,b)
$
(resp. $B_k\mathscr C[1](a,b)$)
be the completion of $\overset{\circ}B\mathscr C[1](a,b)
$ 
(resp. $\overset{\circ}B_k\mathscr C[1](a,b)
$) with respect to the energy filtration.
\par
Note that we do not include $k=0$ in (\ref{form24new}). This is different from the convention of \cite[Section 2]{FuFu6}.
In this paper we assume that a filtered $A_{\infty}$ category satisfies $\frak m_0 = 0$.
\par
We define a homomorphism 
\begin{equation}\label{coalgonB}
\Delta : B_k\mathscr C[1](a,b) \to \underset{k_1+k_2=k}{\widehat{\bigoplus}} 
\underset{c}{\widehat{\bigoplus}}\,\,\,
B_{k_1}\mathscr C[1](a,c) \widehat\boxtimes B_{k_2}\mathscr C[1](c,b)
\end{equation}
by
\begin{equation}\label{coalgonBdef}
\Delta(x_1\otimes\cdots\otimes x_k) = \sum_{k'=1}^{k-1} (x_1 \otimes\cdots\otimes x_{k'})
\boxtimes
(x_{k'+1} \otimes\cdots\otimes x_{k}).
\end{equation}
This map is co-associative in an obvious sense.
Here and hereafter we use $\boxtimes$ in place of $\otimes$ in the 
right hand sides of (\ref{coalgonB}),(\ref{coalgonBdef}).
In fact several different kinds of tensor products appear in this paper 
so it seems necessary to introduce several different notations 
for tensor products to distinguish them.
\par
The submodule $\overset{\circ}B\mathscr C[1](a,b)$ is preserved by $\Delta$.
We call $\overset{\circ}B\mathscr C[1](a,b)$ 
the {\it finite part} of $B\mathscr C[1](a,b)$
\par
Let $\mathscr C$, $\mathscr C'$ be filtered $A_{\infty}$
categories and $\Phi_{\rm ob} : \frak{OB}(\mathscr C) \to \frak{OB}(\mathscr C')$
a set theoretical map.
Let $a,b$ be  objects of $\mathscr C$ and 
$a' = \Phi_{\rm ob}(a)$, $b' =\Phi_{\rm ob}(b)$.
When a $\Lambda_0$ filtered module homomorphism
$
\Phi_{k}(a,b) : B_k\mathscr C[1](a,b) \to \mathscr C'[1](a',b'),
$
of degree $0$ is given for each $k,a,b$, there exists uniquely a $\Lambda_0$ module homomorphism
$
\widehat{\Phi}(a,b) : B\mathscr C[1](a,b) \to B\mathscr C'[1](a',b')
$
such that:
\begin{enumerate}
\item The map $\widehat{\Phi}(a,b)$ is a co-homomorphism. Namely 
\begin{equation}
\Delta \circ \widehat{\Phi}(a,b)
= \bigoplus_c   (\widehat{\Phi}(a,c) \boxtimes \widehat{\Phi}(c,b)) \circ  \Delta.
\end{equation}
\item
The composition of $\widehat{\Phi}(a,b)$ with 
the projection $B\mathscr C[1](a,b) \to \mathscr C[1](a,b)$
is $\Phi_{k}(a,b)$ on $B_k\mathscr C[1](a,b)$.
\item
The finite parts are preserved by $\widehat{\Phi}(a,b)$.
\end{enumerate}

\begin{defn}\label{filfuncenergy}
A {\it filtered $A_{\infty}$ functor $\Phi$ from $\mathscr C$ to $\mathscr C'$ with energy loss $\rho$}
is by definition $(\Phi_{\rm ob},\{\Phi_{k}(a,b)\})$ such that:
\begin{enumerate}
\item A set theoretical map $\Phi_{\rm ob} : \frak{OB}(\mathscr C) \to \frak{OB}(\mathscr C')$ 
is given.
\item
For objects $a,b$ and $k=1,2,\dots$ $\Phi_{k}(a,b)$
is a $\Lambda$ module homomorphism:
$$
\Phi_{k}(a,b) : B_k\mathscr C_{\Lambda}[1](a,b) \to \mathscr C'_{\Lambda}[1](a',b').
$$
(Note that here we change the coefficient ring from $\Lambda_0$ to $\Lambda$.)
\item
We obtain $
\widehat{\Phi}(a,b) : \overset{\circ}B\mathscr C_{\Lambda}[1](a,b) \to \overset{\circ}B\mathscr C'_{\Lambda}[1](a',b')
$ as above.
Then 
$$
\widehat{\Phi}(a,b) \circ \hat d = \hat d \circ \widehat{\frak F}(a,b).
$$
Here $\hat d$ is the derivation obtained from $A_{\infty}$ operations.\footnote{
We remark that we use $\overset{\circ}B\mathscr C_{\Lambda}[1](a,b)$ not its completion 
$B\mathscr C_{\Lambda}[1](a,b)$.
Note that $\Phi_{k}(a,b)$ decrease energy filtration by $k\rho$, 
which goes to infinity as $k\to\infty$.
So $\widehat{\Phi}$ cannot be extended to the completion.}
\item
We consider the energy filtration $\frak F^{\lambda}$ of 
$B\mathscr C_{\Lambda}[1](a,b)$ such that
$x_1 \otimes x_2 \otimes \dots \otimes x_k$ is in 
$\frak F^{\lambda}B\mathscr C_{\Lambda}[1](a,b)$ if and only if 
$x_i \in \frak F^{\lambda_i}(c_i,c_{i+1})$ with $\sum \lambda_i = \lambda$.
Then
$$
\Phi_{k}(a,b)(\frak F^{\lambda}B_k\mathscr C_{\Lambda}[1](a,b))
\subset 
\frak F^{\lambda-k\rho}\mathscr C'_{\Lambda}[1](a',b')
$$
\end{enumerate}
\end{defn}
\begin{defn}
We say a filtered $A_{\infty}$ functor $\Phi: \mathscr C \to \mathscr C'$ with energy loss $\rho$
between unital filtered $A_{\infty}$ categories is {\it unital} if
$$
\Phi_1({\bf e}_{a}) = {\bf e}_{\Phi_{\rm ob}(a)},
\quad \Phi_k(x_1,\dots,{\bf e}_{c_i},\dots,x_k) = 0, \quad k\ge 2.
$$ 
Here ${\bf e}_a \in \mathscr C(a,a)$ is the unit. 
\end{defn}
Note that a filtered $A_{\infty}$ functor $\Phi: \mathscr C \to \mathscr C'$ with energy loss 0
is nothing but a filtered $A_{\infty}$ functor.
In \cite[Definition 5.21]{fooobook} we defined the notion of a 
filtered $A_{\infty}$ bi-module homomorphism with energy loss.
Biran-Cornea-Shelukhin \cite{BCS} uses a related notion
which they call a weakly filtered $A_{\infty}$ functor with discrepancy.
\footnote{Biran-Cornea-Shelukhin also use the notion of 
a weakly filtered $A_{\infty}$ structure with discrepancy.
It seems difficult to use a bounding cochain
for such a structure, 
since Maurer-Cartan equation is likely to diverge. 
\cite{BCS} studies the exact or monotone cases and they do not 
use bounding cochain.}

\begin{lem}\label{energylosscup}
Let $\Phi^n : \mathscr C^n \to \mathscr C^{n+1}$ be a filtered $A_{\infty}$
functor with energy loss $\rho_n$ for $n=1,2$.
Then we can compose to obtain $\Phi^2\circ \Phi^1: \mathscr C^1 \to \mathscr C^{3}$
that is a filtered $A_{\infty}$
functor with energy loss $\rho_1 + \rho_2$.
\end{lem}
\begin{proof}
We can define a composition as an $A_{\infty}$
functor linear over $\Lambda$. 
(See for example \cite[Definition 7.34]{fu4}.) 
Then Definition \ref{filfuncenergy} (4) can 
be checked easily from the explicit formula.
In fact $\widehat{\Phi^2\circ \Phi^1} 
= \widehat{\Phi^2}\circ \widehat{\Phi^1}$.
\end{proof}
\begin{lem}\label{lem11300}
If $\Phi : \mathscr C_1 \to \mathscr C_2$ is a unital filtered $A_{\infty}$ functor 
with energy loss $\rho$, then 
$$
d_{{\rm Hof}}(\Phi_{\rm ob}(c_1),\Phi_{\rm ob}(c_2)) \le d_{{\rm Hof}}(c_1,c_2) + 2\rho,
$$
\end{lem}
\begin{proof}
Let $c_1,c_2 \in \frak{OB}(\mathscr C_1)$ and
$(t_{12},t_{21},s_1,s_2)$ be a homotopy equivalence between them with 
energy loss $\epsilon$.
We put
$
t'_{12}: = \Phi_1(t_{12})$, $t'_{21}: =\Phi_2(t_{21})$
$s'_{1} = \Phi_1(s_{1}) + \Phi_2(t_{12},t_{21})$,
$s'_{2} = \Phi_1(s_{2}) + \Phi_2(t_{21},t_{12})$.
It is easy to check that they become a  homotopy equivalence
between $\Phi_{\rm ob}(c_1),\Phi_{\rm ob}(c_2)$ with 
energy loss $\epsilon+2\rho$.
The lemma follows.
\end{proof}
The next corollary is immediate from Lemma \ref{lem11300}.
\begin{cor}\label{Hofdishomotopyinv}
The Hofer distance is invariant under the unital homotopy equivalence 
of unital filtered $A_{\infty}$ categories.
\end{cor}
\begin{defn}\label{defn4646}
Let $(C,d)$, $(C',d')$ be  filtered (resp. completed) cochain complexes over $\Lambda_0$.
A {\it cochain map (resp. completed cochain map) with energy loss $\rho$} from $(C,d)$ to $(C',d')$
is by definition a cochain map (resp. completed cochain map) $\phi : C_{\Lambda} 
\to C'_{\Lambda}$ such that 
$\phi(\frak F^{\lambda}C_{\Lambda}) \subseteq \frak F^{\lambda-\rho}C'_{\Lambda}$.
\par
Let $\phi, \psi : C_{\Lambda} 
\to C'_{\Lambda}$ be cochain maps with energy loss $\rho$.
A {\it cochain homotopy with energy loss $\rho$ between $\phi, \psi$}
is $H: C_{\Lambda} 
\to C'_{\Lambda}$ such that $d'_{\Lambda} \circ H + H \circ d_{\Lambda} =  
\phi - \psi'$ and 
$H(\frak F^{\lambda}C_{\Lambda}) \subseteq \frak F^{\lambda-\rho}C'_{\Lambda}$.
In the case of completed cochain complexes we require that $H$ preserves the finite 
part.
\par
A cochain map $\phi$ with energy loss $\rho_1$ from $(C,d)$ to $(C',d')$
is said to be a {\it cochain homotopy equivalence with energy loss $\rho$}
if there exists a cochain map $\psi$ with energy loss $\rho_2$ from $(C',d')$ to $(C,d)$
with $\rho_1 + \rho_2 \le \rho$ such that the compositions $\phi \circ \psi$ and $\psi \circ \phi$ are 
cochain homotopy equivalent to the identity map 
as cochain maps with energy loss $\rho$.
\end{defn}
\begin{thm}\label{GrokaraInd}
Suppose $\mathscr C_1$ and $\mathscr C_2$ are 
gapped and unital filtered $A_{\infty}$ categories with 
$d_{\rm GH}(\mathscr C_1,\mathscr C_2) < \epsilon$.
Then there exists a gapped filtered $A_{\infty}$
functor $\Phi: \mathscr C_1 \to \mathscr C_2$ 
with energy loss $\epsilon$.
\par
Its linear part $\Phi_1 : \mathscr C_1(a,b) \to \mathscr C_2(\Phi_{\rm ob}(a),\Phi_{\rm ob}(b))$
is a cochain homotopy equivalence with energy loss $2\epsilon$.
\par
The same holds for completed DG-category (which is not necessary gapped).
\end{thm}

\begin{proof}
The main part of the proof is the following:
\begin{lem}\label{Lemma3434}
Let $\mathscr C$ be a gapped and unital filtered $A_{\infty}$
category and $\frak C_{\mathcal I} = \{c_i \mid i \in \mathcal I\}$, $
\frak C'_{\mathcal I} = \{c'_i \mid i \in \mathcal I\}$ 
are sets of objects. 
Here $\mathcal I$ is a certain index set.
We assume that 
$d_{\rm Hof}(c_i,c'_i) < \epsilon$ for $i \in \mathcal I$.
Then there exists a filtered $A_{\infty}$ functor 
with energy loss $\epsilon$,
$
\Phi: {\mathscr C}(\frak C_{\mathcal I})
\to 
{\mathscr C}(\frak C'_{\mathcal I})
$
between full subcategories.
\par
Its linear part is a cochain homotopy equivalence with energy loss $2\epsilon$.
\par
The same holds for completed DG-category (which is not necessary gapped).
\end{lem}
\begin{proof}
We may assume that $\mathscr C$ is a DG-category.
(Here we use gapped-ness.)
Note that there is an embedding of filtered $A_{\infty}$ categories
$
\frak I: {\mathscr C}(\frak C_{\mathcal K})
\to  {\mathscr C}(\frak C_{\mathcal K} \cup \frak C'_{\mathcal K}).
$
By \cite[Lemma 8.45]{fu4} 
$
\frak I_{\Lambda}: {\mathscr C}_{\Lambda}(\frak C_{\mathcal K})
\to  {\mathscr C}_{\Lambda}(\frak C_{\mathcal K} \cup \frak C'_{\mathcal K})
$
is a homotopy equivalence of $A_{\infty}$ categories over $\Lambda$.
Its homotopy inverse $\Psi$ in the case of DG-category is given 
explicitly in \cite[page 115]{fu4}, that is:
\begin{equation}\label{nhomotoyp}
\Psi_k(x_1,\dots,x_k)
= t \circ x_1 \circ s_1 \circ x_2 \circ \dots \circ x_{k-1}\circ  s_{k-1}  \circ x_{k} \circ t'.
\end{equation}
Here the symbol $\circ$ is obtained from the $\frak m_2$ operator 
by putting the sign.
Namely 
\begin{equation}\label{signchangetocirc}
x \circ y = (-1)^{\deg x}\frak m_2(x,y),\qquad
dx = \frak m_1(x).
\end{equation}
Then $\circ$ is associative and $d$ is a coderivation.
(See \cite[Proof of Lemma 8.45]{fu4}.)
\par
The morphisms $t,t',s_j$ are defined as follows. 
Let $x_j \in \mathscr C(c_{i_{j-1}},c_{i_{j}})$, $j=1,\dots,k$
and $t = t_{c'_{i_0},c_{i_0}}\in \mathscr C(c'_{i_0},c_{i_0})$, $t' = t_{c_{i_k},c'_{i_k}} \in \mathscr C(c_{i_k},c'_{i_k})$
and $s_j = s_{c_{i_j}} \in \mathscr C(c_{i_j},c_{i_j})$.
Here $(t_{c_i,c'_i},t_{c'_i,c_i},s_{c_i},s_{c'_i})$ is a 
homotopy equivalence with energy loss $\epsilon$ between $c_i$ and $c'_i$.
It is easy to check
$
\frak v(\Psi_k(x_1,\dots,x_k)) >
\sum \frak v(x_j) - k\epsilon,
$
that is, $\Psi$ has energy loss $\epsilon$.
\par
The linear part $\Psi_1 \circ \frak I_1$ is given by
$x \mapsto t\circ x \circ t'$.
Its cochain homotopy inverse is given by
$y \mapsto t'\circ x \circ t$.
In fact the composition is 
$x \mapsto t\circ t' \circ x \circ t\circ t'$ 
and
$$
t'\circ t\circ x \circ t\circ t' - x 
= -d(s \circ x \circ t\circ t' +  x \circ s')
+ (s \circ dx \circ t\circ t' +  dx \circ s').
$$
\end{proof}
Now we prove Theorem \ref{GrokaraInd}.
We take $\epsilon'$ such that 
$
d_{
\rm GH}(\mathscr C_1,\mathscr C_2) < \epsilon' < \epsilon.
$
We take $\mathscr C$ as in Definition \ref{GromovHausdorff}.
We take $\mathscr C^0_i$ the full subcategory 
of $\mathscr C_i$ with the following properties.
\begin{enumerate}
\item
The $\epsilon-\epsilon'$ neighborhood of $\frak{OB}(\mathscr C^0_i)$
contains $\frak{OB}(\mathscr C_i)$.
\item
There exists a bijection $I : \frak{OB}(\mathscr C^0_i)
\to \frak{OB}(\mathscr C^0_2)$ such that
$
d_{\rm Hof}(c,I(c)) < \epsilon'$ for any $c$.
\end{enumerate}
Applying Lemma \ref{Lemma3434} to 
$\mathscr C^0_1$ and $\mathscr C^0_2$
we obtain a filtered $A_{\infty}$ functor 
with energy loss $\epsilon'$,
$\Phi_0$, from $\mathscr C^0_1$ to $\mathscr C^0_2$.
\par
Let $J : \frak{OB}(\mathscr C_1) \to \frak{OB}(\mathscr C^0_1)$
be a set theoretical map such that $d_{\rm Hof}(c,J(c)) < \epsilon- \epsilon_1$.
(Note that $J$ may  not be continuous.) 
Using the same formula (\ref{nhomotoyp})
we obtain a filtered $A_{\infty}$ functor $\Phi_1: \mathscr C_1 \to \mathscr C_1^0$
of energy loss $\epsilon - \epsilon'$
such that it becomes $J$ for objects.
The composition of $\Phi_0 \circ \Phi_1$ with the 
inclusion $\mathscr C_1^0 \to \mathscr C_1$ is the 
required filtered $A_{\infty}$ functor.
\end{proof}

\begin{defn}\label{defnindAinf}
An {\it inductive system of filtered $A_{\infty}$ categories} is 
a pair $((\mathscr C^1,\mathscr C^2,\dots),\linebreak (\Phi^1,\Phi^2,\dots))$
such that:
\begin{enumerate}
\item
A filtered $A_{\infty}$ category $\mathscr C^n$ is given for $n=1,2,\dots$.
\item
A filtered $A_{\infty}$ functor with 
energy loss $\epsilon_n$, $\Phi^n: \mathscr C^n \to \mathscr C^{n+1}$ is given for $n=1,2,\dots$.
\item The sum $\sum \epsilon_n$ is finite.
\end{enumerate}
\end{defn}
\begin{defn}
An inductive system of filtered $A_{\infty}$ categories
is said to be {\it unital} (resp. {\it gapped}) if 
filtered $A_{\infty}$ categories and filtered $A_{\infty}$ functors 
are all unital (resp. {\it gapped}).
(We remark that the discrete monoid $G$ appearing in the 
definition of gapped-ness may depend on $n$.)
\end{defn}
\begin{thm}\label{mainalgtheorem2}
Let $(\{\mathscr C^n\},\{\Phi^n\})$ be a gapped 
inductive system of filtered $A_{\infty}$ categories.
Then 
there exists a completed DG-category 
$
{\mathscr C}^{\infty}
$
such that:
\begin{enumerate}
\item
The set of objects of ${\mathscr C}^{\infty}$
is identified 
with the 
inductive limit as the sets
$\varinjlim\frak{OB}({\mathscr C}^{n})$.
Here we use $\Phi^n_{\rm ob} : \frak{OB}({\mathscr C}^{n}) 
\to \frak{OB}({\mathscr C}^{n+1})$ to define this inductive system.
\item
We put $\epsilon'_n = \sum_{m=n}^{\infty} \epsilon_m$.
Then there exists a filtered $A_{\infty}$ functor 
$\Upsilon^{\infty,n}:  {\mathcal C}^{n} \to {\mathcal C}^{\infty}$
with energy loss $\epsilon'_n$ such that
$
\Upsilon^{\infty,n+1} \circ \Phi^{n} = \Upsilon^{\infty,n}.
$
\item
Let $a^{\infty},b^{\infty} \in \frak{OB}({\mathcal C}^{\infty})$.
We take $a^m,a^{m+1},\dots$ and $b^m,b^{m+1},\dots$
which represents $a^{\infty},b^{\infty}$.
Then the cohomology group
$
H({\mathscr C}^{\infty}(a^{\infty},b^{\infty}),\frak m_1)
$
is almost isomorphic to 
$H(\varinjlim({\mathscr C}^{n}(a^{n},b^{n}),\frak m_1).
$\footnote{
The definition of the inductive limit 
$\varinjlim({\mathscr C}^{n}(a^{n},b^{n}),\frak m_1)$ will be given in 
Definition \ref{lininductdefn}.}
\item
If ${\mathscr C}^{n}$, $\Phi^n$ are all unital then  
${\mathscr C}^{\infty}$ is homotopically unital.\footnote{
See Section \ref{sec;alhomoequi} for the definition of homotopical unitality.}
Moreover $\Upsilon^{\infty,n}$ is homotopically unital.
\end{enumerate}
\end{thm}
Theorem \ref{mainalgtheorem2} will be proved in Sections \ref{sec;bar}-\ref{sec;indhomoequi}.
We now define the inductive limit appearing in item (3).
\begin{defn}
An {\it inductive system of filtered cochain complexes over $\Lambda_0$} 
(resp. {\it completed cochain complexes})  is $(\{C^n\},\{\varphi^n\})$
where:
\begin{enumerate}
\item
A filtered cochain complex 
 $C^n$ 
 over  $\Lambda_0$ 
  (resp. completed cochain complexes $(C^n,C^n_0)$) is given.
\item
A filtered (resp. completed) cochain map $\varphi^n: C^n \to C^{n+1}$  over $\Lambda_0$
with energy loss $\epsilon_n$ is given.
\item The sum $\sum \epsilon_n$ is finite.
\end{enumerate}
\end{defn}
\begin{defnlem}\label{lininductdefn}
We can define the inductive limit of an inductive system of filtered cochain complexes 
 over $\Lambda_0$ or completed cochain complexes. 
The inductive limit is a completed  cochain complex in the 
sense of Definition \ref{defn216}.
\end{defnlem}
\begin{proof}
We consider the inductive limit
$
\varinjlim\, C^n_{\Lambda}
$
that is a $\Lambda$ vector space, which
we denote  by $\overset{\circ}{C^{\infty}_{\Lambda}}$.
The co-boundary operators $d_n$ on $C^n_{\Lambda}$ induces 
a co-boundary operator $d$ on $\overset{\circ}{C^{\infty}_{\Lambda}}$.
We define a filtration $\frak F^{\lambda}\overset{\circ}{C^{\infty}_{\Lambda}}$
as follows. An element $x^{\infty}$ of $\overset{\circ}{C^{\infty}_{\Lambda}}$
is contained in $\frak F^{\lambda}\overset{\circ}{C^{\infty}_{\Lambda}}$ 
if and only if there exists a sequence $x^n \in C^n_{\Lambda}$
such that $\lim_{n\to \infty} x^n = x^{\infty}$ and that 
$\frak v(x^n) \ge \lambda_n$
with $\lim_{n\to \infty}\lambda_n = \lambda$.
We put $\overset{\circ}{C^{\infty}}: = \frak F^{0}\overset{\circ}{C^{\infty}_{\Lambda}}$.
Let $C^{\infty}$ be the completion of $\overset{\circ}{C^{\infty}}$
with respect to the energy filtration.
Then $(C^{\infty},\overset{\circ}{C^{\infty}})$
is the completed filtered cochain complex we look for.
\par
In the case of inductive system of compelted cochain complexes we first take 
the inductive limit of the finite part and then take the completion.
\end{proof}
We remark that if  $\Phi: \mathscr C \to \mathscr C'$ is a filtered $A_{\infty}$ functor with 
energy loss $\rho$ then, for objects $a,b$ of $\mathscr C$, the map 
$\Phi_1(a,b): \mathscr C(a,b) \to \mathscr C'(\Phi_{\rm ob}(a),\Phi_{\rm ob}(b))$
is a filtered cochain map of energy loss $\rho$.
Therefore the inductive limit 
$\varinjlim({\mathscr C}^{n}(a^{n},b^{n}),\frak m_1)$ 
appearing in Theorem \ref{mainalgtheorem2} (3)(b) 
is defined in Lemma-Definition \ref{lininductdefn}.

\begin{rem}\label{rem41115}
Let $(\{C^n\},\{\varphi^n\})$
be an inductive system of filtered cochain complexes over $\Lambda_0$.
Suppose in addition that the energy loss of $\varphi^n$ is $0$.
We obtain the inductive limit in the usual sense.
Namely we take the set theoretical inductive limit of $C^n$ and define 
the co-boundary operator as the limit.
This inductive limit or its completion, however, may not coincide with one 
in Lemma-Definition \ref{lininductdefn}.
\par
In fact, let us take $C^n = \Lambda_0$ with $0$ 
as co-boundary operators. We put $\varphi^n(x) =  T^{1/2^n}x$.
The inductive limit in Lemma-Definition \ref{lininductdefn} is $\Lambda_0$.
The usual inductive limit is $\Lambda_+$.
We remark that $\Lambda_+$ is closed in $\Lambda_0$ with respect to the 
$T$-adic topology.
\end{rem}

\section{The Bar resolution of a filtered $A_{\infty}$ category.}
\label{sec;bar}

In this section we describe the Bar-resolution of a filtered $A_{\infty}$ category.
The Bar resolution is a well established notion. We give its definition 
here since we need to study various filtrations carefully.

\begin{defn}\label{defn51}
A {\it doubly filtered differential graded co-category} $\mathscr B$
is the following object.
Hereafter we write {\it DFDGC} instead of doubly filtered differential graded co-category.
\begin{enumerate}
\item 
The set of objects, $\frak{OB}(\mathscr B)$, is given.
\item
For each $a,b \in \frak{OB}(\mathscr B)$ a 
completed cochain complex over $\Lambda_0$, abbreviated by 
$\mathscr B(a,b)$, is given.
Its filtration $\frak F^{\frak \lambda}\mathscr B(a,b)$ is called 
the {\it energy filtration}.
We denote the finite part of $\mathscr B(a,b)$ by $\overset{\circ}{\mathscr B}(a,b)$
\item Another filtration $\frak S^k\mathscr B(a,b)$, abbreviated by 
{\it number filtration}, is given such that:
\begin{enumerate}
\item
For $k =1,2,\dots$ $\frak S^k\mathscr B(a,b)$ 
is a completed  subcomplex of $\mathscr B(a,b)$
with $\frak S^k\overset{\circ}{\mathscr B}(a,b)$ being 
its finite part.
\item
$\frak S^{k+1}\mathscr B(a,b) \supseteq \frak S^k\mathscr B(a,b)$.
\item
We require
$
\overset{\circ}{\mathscr B}(a,b): = \bigcup_n\frak S^n\overset{\circ}{\mathscr B}(a,b).
$
\end{enumerate}
\item A map 
$
\Delta : \overset{\circ}{\mathscr B}(a,b) \to  
\underset{c}{{\bigoplus}}
\overset{\circ}{\mathscr B}(a,c) \boxtimes \overset{\circ}{\mathscr B}(c,b).
$
is given. We call it the {\it co-composition}.
It induces 
$
\Delta : {\mathscr B}(a,b) \to  
\underset{c}{\widehat{\bigoplus}}
{\mathscr B}(a,c) \widehat\boxtimes {\mathscr B}(c,b).
$
\item
The co-composition is co-associative and preserves 
co-derivation and two filtrations.
\end{enumerate}
\end{defn}
\begin{defn}
To a filtered $A_{\infty}$ category $\mathscr C$ 
we  associate a DFDGC $B\mathscr C$ as follows.
We call $B\mathscr C$ the {\it Bar resolution} of $\mathscr C$.
\begin{enumerate}
\item
The set of objects is defined by
$\frak{OB}(B\mathscr C) := \frak{OB}(\mathscr C)$.
\item 
The finite part of the morphism space $\overset{\circ}B\mathscr C(a,b)$ 
is defined by (\ref{form24new}).
$B\mathscr C(a,b)$ is its completion.
The co-boundary operator is the co-derivation $\hat d$ 
induced from the $A_{\infty}$ operations $\frak m_k$.
The energy filtration of $B\mathscr C(a,b)$ 
is induced from the energy filtration of 
$\mathscr C(a,b)$ in an obvious way.
\item
The number filtration is defined by:
$
\frak S^k\overset{\circ}B\mathscr C(a,b)
:=
\bigoplus_{\ell=1}^k \overset{\circ}B_{\ell}\mathscr C(a,b).
$
Here the right hand side is defined by (\ref{form24new}).
\item
The co-composition is defined by (\ref{coalgonBdef}).
It is obvious that it is co-associative and 
preserves two filtrations. Moreover $\hat d$ is a co-derivation
and preserves two filtrations.
\end{enumerate}
\end{defn}
\begin{defn}\label{coDBDBdemor}
Let $\mathscr B$, $\mathscr B'$ be DFDGCs.
A {\it DFDGC morphism with  energy loss $\rho$} from 
$\mathscr B$ to $\mathscr B'$ is $\Psi: = (\Psi_{\rm ob},\Psi_\Lambda)$
where:
\begin{enumerate}
\item 
A map $\Psi_{\rm ob}: \frak{OB}(\mathscr B) \to \frak{OB}(\mathscr B')$
is given.
\item
Let $a' = \Psi_{\rm ob}(a)$, $b' = \Psi_{\rm ob}(b)$.
Then a $\Lambda$ linear map 
$\Psi_\Lambda(a,b): \overset{\circ}{\mathscr B}_\Lambda(a,b) \to \overset{\circ}{\mathscr B'}_\Lambda(a',b')$
of degree $0$ is given. The map
$\Psi_\Lambda(a,b)$ preserves the number filtration and the finite part.
\item
The linear map $\Psi_{\Lambda}$ is a co-homomorphism.
Namely
$
\Delta \circ {\Psi_{\Lambda}}(a,b)
= \bigoplus_c   ({\Psi_{\Lambda}}(a,c) \boxtimes {\Psi_{\Lambda}}(c,b)) \circ  \Delta.
$
on 
$\overset{\circ}{\mathscr B}_\Lambda(a,b)$.
\item
We require
\begin{equation}
\Psi_\Lambda(a,b)(\frak F^{\lambda}\frak S^k(\overset{\circ}{\mathscr B}(a,b)))
\subseteq
\frak F^{\lambda-k\rho}\frak S^k(\overset{\circ}{\mathscr B'}(a',b')).
\end{equation}
In particular $\Psi_\Lambda(a,b)$ induces a map: 
$\frak F^{\lambda}\frak S^k({\mathscr B}(a,b))
\to
\frak F^{\lambda-k\rho}\frak S^k({\mathscr B'}(a',b'))$.
\end{enumerate}
\end{defn}
\begin{defn}
Let $\Phi: \mathscr C \to \mathscr C'$ be a filtered $A_{\infty}$ functor  with energy loss $\rho$.
It induces a DFDGC morphism with with energy loss $\rho$, which is abbreviated by
$B\Phi: B\mathscr C \to B\mathscr C'$, as follows.
\begin{enumerate}
\item
The map $(B\Phi)_{\rm ob}: \frak{OB}(B\mathscr C) \to  \frak{OB}(B\mathscr C')$
is $\Phi_{\rm ob}$.
\item
Let $a' = \Phi_{\rm ob}(a)$, $b' = \Phi_{\rm ob}(b)$.
The map $(B\Phi)_{\Lambda}(a,b): \overset{\circ}{\mathscr B}_\Lambda(a,b) 
\to \overset{\circ}{\mathscr B'}_\Lambda(a',b')$
is obtained in Definition \ref{filfuncenergy} (3).
\end{enumerate}
The fact that they satisfy (1)-(4) of Definition \ref{coDBDBdemor} is 
immediate from Definition \ref{filfuncenergy}.
\end{defn}

\begin{defn}
An {\it inductive system of DFDGC} is $(\{\mathscr B^n\},\{\Psi^n\})$
where:
\begin{enumerate}
\item
A DFDGC $\mathscr B^n$ is given for $n=1,2,\dots$.
\item
A DFDGC morphism  with energy loss 
$\epsilon_n$, $\Psi^n: \mathscr B^n \to \mathscr B^{n+1}$, is given for $n=1,2,\dots$.
\item The sum $\sum \epsilon_n$ is finite.
\end{enumerate}
\end{defn}
\begin{lem}\label{lem46}
If $(\{\mathscr C^n\},\{\Phi^n\})$ is an 
inductive system of filtered $A_{\infty}$ categories, then
$(\{B\mathscr C^n\},\{B\Phi^n\})$
is an inductive system of DFDGC.
\end{lem}
The proof is obvious.
Note that the maps appearing in the morphisms of DFDGC are linear, 
while the maps appearing in the filtered $A_{\infty}$ functors 
are multi-linear (or non-linear).
By this reason taking the inductive limit is much easier 
for DFDGC than for filtered $A_{\infty}$ categories.

\begin{defn}\label{indlimDFDF}
Let $(\{\mathscr B^n\},\{\Psi^n\})$
be an inductive system of DFDGC.
We define its {\it inductive limit}, 
abbreviated by $\varinjlim\,{\mathscr B}^{n}$, 
as follows.
We put ${\mathscr B}^{\infty} = \,
\varinjlim{\mathscr B}^{n}$.
\begin{enumerate}
\item
The set $\frak{OB}({\mathscr B}^{\infty})$ of its objects 
is the inductive limit 
$
\varinjlim\,\frak{OB}({\mathscr B}^{n})
$
as sets. Note that we use $\Psi^n_{\rm ob}$ to define this inductive system.
\item
Let $a^{\infty}, b^{\infty}$ be elements of ${\mathscr B}^{\infty}$.
We take sequences $a^n$, $b^n$ of elements $\frak{OB}({\mathscr B}^{n})$
representing $a^{\infty}, b^{\infty}$, respectively.
We put
\begin{equation}\label{formula44}
\frak S^k{\mathscr B}^{\infty}(a^{\infty},b^{\infty}) =
\varinjlim\,\frak S^k{\mathscr B}^{n}(a^n,b^n).
\end{equation}
Note that $\frak S^k{\mathscr B}^{n}(a^n,b^n)$ together with 
$\Psi^n_{\Lambda}(a^n,b^n)$ becomes an 
inductive system of completed cochain complexes. Therefore the inductive limit in the 
right hand side of (\ref{formula44}) is defined by  Lemma-Definition \ref{lininductdefn}.
\item
We put 
$\overset{\circ}{\mathscr B^{\infty}}(a,b): = \bigcup_n\frak S^n\overset{\circ}{\mathscr B^{\infty}}(a,b)
$
and define ${\mathscr B}^{\infty}(a,b)$ to be its completion with respect 
to the energy filtration.
\item
Since $\Psi^n$ preserves co-composition  $\Delta$ 
it induces a co-composition $\Delta$ on the inductive limit.
\end{enumerate}
It is easy to see that the axiom of DFDGC is satisfied.
\end{defn}

\section{Back to $A_{\infty}$ category via Co-Bar resolution.}
\label{sec;back}

Let $(\{\mathscr C^n\},\{\Phi^n\})$ be an 
inductive system of filtered $A_{\infty}$ categories.
By Lemma \ref{lem46} and Definition \ref{indlimDFDF}
we obtain a DFDGC $\mathscr B^{\infty}$
by
\begin{equation}
\mathscr B^{\infty}:= \varinjlim\,B{\mathscr C}^{n}.
\end{equation}
\par
Note that by Lemma-Definition \ref{lininductdefn} we have an inductive limit
\begin{equation}\label{414141}
\mathscr C^{\infty}(a,b):= \varinjlim\,{\mathscr C}^{n}(a^n,b^n),
\end{equation}
which is a completed cochain complex 
if $a = \lim a_n$ and $b = \lim b_n$.
In particular the boundary operator $\frak  m_1$ is induced on it.
However its $A_{\infty}$ operations $\frak m_k$ for $k>1$ is  not given.
\par
We define a few maps between $B{\mathscr C}^{n}(a_n,b_n)$, $\mathscr B^{\infty}(a,b)$.
Let $a_n,b_n \in \frak{OB}(\mathscr C^n)$.
For $n' > n$ we define $a_{n'},b_{n'}$
inductively by $a_{n'+1} = \Phi^{n'+1,n'}_{\rm ob}(a_{n'})$,
$b_{n'+1} = \Phi^{n'+1,n'}_{\rm ob}(b_{n'})$.
We then define
$
\hat\Phi^{n',n}(a_n,b_n) : \overset{\circ}B\mathscr C^{n}(a_n,b_n) \to 
\overset{\circ}B\mathscr C^{n'}(a_{n
'},b_{n'})
$
by
\begin{equation}\label{defnform54}
\hat\Phi^{n',n}(a_n,b_n):= \hat\Phi^{n',n'-1}(a_{n'-1},b_{n'-1}) \circ \dots \circ \hat\Phi^{n+1,n}(a_n,b_n)
\end{equation}
We put $a_{\infty} = \lim a_n$ $b_{\infty}= \lim b_n$.
We define
\begin{equation}\label{iso53444}
\hat\Phi^{\infty,n}(a_n,b_n) : \overset{\circ}B\mathscr C^{n}(a_n,b_n) \to 
\overset{\circ}{\mathscr B^{\infty}}(a_{\infty},b_{\infty})
\end{equation}
by sending ${\bf x}$ to an element represented by
$(\hat\Phi^{n',n}({\bf x}))_{n'=n}^{\infty}$.
We have the following:
\begin{equation}\label{filter666}
\hat\Phi^{\infty,n}(a_n,b_n)(\frak F^{\lambda}\frak S^k \overset{\circ}B\mathscr C^{n}(a_n,b_n))
\subseteq 
\frak F^{\lambda-k\epsilon_n}\frak S^k\overset{\circ}{\mathscr B^{\infty}}(a_{\infty},b_{\infty})
\end{equation}
with $\epsilon_n \to 0$.

We use the Co-Bar resolution to define the inductive limit 
of filtered $A_{\infty}$ categories.
The notion of a Co-Bar resolution is  classical. (See for example \cite{keller}.) We however 
give a detailed  proof since again we need to study its relation to the 
two filtrations.
We also need to use operations and notations we introduce below 
in later sections.
Let $\mathscr B$ be a DFDGC. For its objects $a,b$, we consider 
\begin{equation}\label{form24newnew}
\left\{
\aligned
\overset{\circ}A_k\mathscr B[1](a,b) &= \underset{a=c_0,c_1,\cdots,c_{k-1},c_k=b}{{\bigoplus}}
\overset{\circ}{\mathscr B}[1](c_0,c_1) \boxtimes \cdots \boxtimes \overset{\circ}{\mathscr B}[1](c_{k-1},c_k)
\\
\overset{\circ}A\mathscr B[1](a,b) &= \underset{k=1,2,\dots}{{\bigoplus}} \overset{\circ}A_k\mathscr B[1](a,b).
\endaligned
\right.
\end{equation}
The energy filtration of $\mathscr B[1](a,b)$ induces 
an energy filtration of $\overset{\circ}A\mathscr B[1](a,b)$.
Here we use $A$ and $\boxtimes$ in place of $B$ and $\otimes$ to distinguish 
Co-Bar resolution from Bar resolution. 
(The symbol $A$ stands for Adams, who introduced Co-Bar resolution in \cite{Ada}.
The symbol $F$ is used in some of the references. I avoid it 
since $F$ is used for filtration.) 
Let $A_k\mathscr B[1](a,b)$ be the completion of 
$\overset{\circ}A_k\mathscr B[1](a,b)$ with respect to the energy 
filtration.
The co-boundary operator $d$ and co-composition $\Delta$ induces operations 
\begin{equation}
\left\{
\aligned
\partial : A\mathscr B[-1](a,b) \to A\mathscr B[-1](a,b) \\
\hat \Delta : A\mathscr B[-1](a,b) \to A\mathscr B[-1](a,b)
\endaligned
\right.
\end{equation}
by
\begin{equation}
\left\{
\aligned
\partial(x_1 \boxtimes \dots \boxtimes x_k)
&= 
\sum_{i=1}^k (-1)^{*} x_1 \boxtimes \dots \boxtimes d(x_i) \boxtimes \dots\boxtimes x_k, \\
\hat \Delta(x_1 \boxtimes \dots \boxtimes x_k)
&= 
\sum_{i=1}^k (-1)^{*} x_1 \boxtimes \dots \boxtimes \Delta(x_i) \boxtimes \dots\boxtimes x_k,
\endaligned
\right.
\end{equation}
where $* = \deg'x_1 + \dots + \deg'x_{i-1}$.
Note that the co-derivation $\Delta$ appearing in the right 
hand side is a map
$
\Delta: \mathscr B[-1](a,b)
\to \widehat{\bigoplus_{c}}\,\,\, \mathscr B[-1](a,c) 
\,\,\widehat{\boxtimes}\,\, \mathscr B[-1](c,b).
$
It satisfies
$
\Delta(\overset{\circ}{\mathscr B}[-1](a,b))
\subseteq 
{\bigoplus_{c}}\,\,\, \overset{\circ}{\mathscr B}[-1](a,c) 
\,\,{\boxtimes}\,\, \overset{\circ}{\mathscr B}[-1](c,b).
$
We put
\begin{equation}
\frak M_1 := \partial + \hat\Delta: (A\mathscr B[-1](a,b))[1]
\to( A\mathscr B[-1](a,b))[1].
\end{equation}
We also define:
\begin{equation}
\frak M_2: (A\mathscr B[-1](a,c))[1]\, \oslash (A\mathscr B[-1](c,b))[1]
\to (A\mathscr B[-1](a,b))[1]
\end{equation}
by 
\begin{equation}\label{frakM2}
\frak M_2(x_1 \boxtimes \dots \boxtimes x_k \oslash x_{k+1} \boxtimes \dots \boxtimes x_m)
=
(-1)^* x_1 \boxtimes \dots \boxtimes x_k \boxtimes x_{k+1} \boxtimes \dots \boxtimes x_m
\end{equation}
with $* = \deg'x_1 + \dots + \deg'x_{i-1}+1$.\footnote{$+1$ is put here 
so that (\ref{defnofIk}) becomes an $A_{\infty}$ functor.
It also appears in several other places.
The author does not know the conceptional reason why it appears. 
Since $A\mathscr B$ is a DG-category we can change the sign of $\frak M_2$ 
and still obtain a DG-category.}
Here $\oslash$ is a tensor product. We use this symbol to distinguish it from 
$\otimes$, $\boxtimes$.
The operators $\frak M_1$ and $\frak M_2$ has degree $1$.
It is easy to see  
\begin{equation}
\left\{
\aligned
&\frak M_1 \circ \frak M_1 = 0, \\
&\frak M_1(\frak M_2({\bf x} \oslash {\bf y}))  
+ \frak M_2(\frak M_1({\bf x}) \oslash {\bf y}))
+  (-1)^{\deg'{\bf x}}\frak M_2({\bf x} \oslash \frak M_1({\bf y}))) = 0.
\endaligned
\right.
\end{equation}
Moreover $\frak M_1$ and $\frak M_2$ preserve the energy filtration.
Therefore putting $\frak M_k = 0$ for $k=3,4,\dots$ 
we obtain a filtered $A_{\infty}$ category.
We denote it by $A\mathscr B$ also.
Since $\frak M_k = 0$ for $k=3,4,\dots$, 
$A\mathscr B$ is actually a DG-category. 
However the sign convention is one of an $A_{\infty}$ category.
(In other words we need a sign change to obtain a DG-category.)
\par
We  define the {\it total number filtration} $\frak S^n\overset{\circ}A\mathscr B$
so that $\frak S^n\overset{\circ}{A\mathscr B}$ 
is the sub-$\Lambda_0$-module generated by the union of
$$
\frak S^{n_1}\overset{\circ}{\mathscr B}[1](c_0,c_1) \boxtimes \cdots \boxtimes \frak S^{n_k}\overset{\circ}{\mathscr B}[1](c_{k-1},c_k)
$$
with $n_1 + \dots + n_k \le n$.
We remark that $A\mathscr B$ is the completion of $\overset{\circ}A\mathscr B: 
= \bigcup_n \frak S^n\overset{\circ}A\mathscr B$.
\par
$A\mathscr B$ becomes a completed DG-category.
Its finite part is $\overset{\circ}A\mathscr B$.
\begin{defn}
We call the completed DG-category $A\mathscr B$ the {\it Co-Bar resolution}
of $\mathscr B$.
\end{defn}
Let $(\{\mathscr C^n\},\{\Phi^n\})$ be an 
inductive system of filtered $A_{\infty}$ categories.
We put
$
\mathscr B^{\infty}:= \varinjlim\,B{\mathscr C}^{n}.
$
Then $A\mathscr B^{\infty}$ is a completed DG-category.
\begin{defn}\label{defnindlim}
We define
$$
\varinjlim\,{\mathscr C}^{n}: = A\mathscr B^{\infty}
$$
to be the {\it inductive limit} of an inductive system of filtered $A_{\infty}$ categories.
\end{defn}
We remark that 
$
\varinjlim\,{\mathscr C}^{n}(a,b) = \frak S^1\mathscr B^{\infty}(a,b)
= \frak S^1A\mathscr B^{\infty}(a,b).
$
Here the inductive limit in the left hand side is one as filtered cochain complexes over $\Lambda_0$.
Therefore there exists a completed injection
\begin{equation}
I: ({\mathscr C}^{\infty}(a,b),\frak m_1) \to (A\mathscr B^{\infty}(a,b),\frak M_1).
\end{equation}

\begin{thm}\label{theorem63}
The completed injection $I$  
is an almost cochain homotopy equivalence.
\end{thm}
\begin{cor}
The map
$
I_* : H({\mathscr C}^{\infty}(a,b),\frak m_1) \to
H(A\mathscr B^{\infty}(a,b),\frak M_1)
$
is an almost isomorphism of $\Lambda_0$ modules.
\end{cor}
This is a consequence of Theorem \ref{theorem63} and Lemma \ref{lemma217}.
We prove Theorem \ref{theorem63} in the next section.

\section{Almost cochain homotopy equivalence: the proof of Theorem \ref{theorem63}.}
\label{sec;speclo}

We begin with the following, which seems to be known to experts at 
least in the case of DG or $A_{\infty}$ category over a filed, 
such as $\Lambda$ or $R$.
We remark that if $\mathscr C$ is gapped then $AB\mathscr C$
is gapped.

\begin{thm}\label{Theorem71}
Let $\mathscr C$ be a gapped filtered $A_{\infty}$ category.
Then there exists a gapped homotopy equivalence of 
filtered $A_{\infty}$ categories 
$$
\frak I = (I_1,I_2,.\dots) : \mathscr C \to AB\mathscr C
$$ 
so that the linear part $I_1$ is the inclusion $I: \mathscr C = \frak S^1 AB\mathscr C
\subseteq A_1B\mathscr C$.
\end{thm}
\begin{proof}
We write an element of $AB\mathscr C$ as 
\begin{equation}\label{defnofcirc}
{\bf x} = x_1 \odot_1 x_2 \odot_2 \dots \odot_{m-1} x_{m-1} \odot x_m
\end{equation}
where $\odot_i$ is one of $\boxtimes$ or $\otimes$ and $x_i \in \mathscr C(c,c')$
for certain elements $c,c' \in \frak{OB}(\mathscr C)$.
We define $\deg \otimes =0$, $\deg \boxtimes = 1$.
\par
The co-boundary operator $\frak M_1$ is the sum of the operations of the forms (I), (II) below.
\begin{enumerate}
\item[(I)] Replacing $\otimes$ to $\boxtimes$. Namely 
\begin{equation}
\aligned
&x_1 \odot_1 \dots \odot_{k-1} x_k \otimes x_{k+1} \odot_{k+2} \dots  \odot_k x_m \\
&\mapsto 
(-1)^*  x_1 \odot_1 \dots \odot_{k-1} x_k \boxtimes x_{k+1} \circ_{k+2} \dots \odot_k x_m
\endaligned
\end{equation}
where $* = \deg'x_1 + \deg  \odot_1 + \dots  + \deg' x_{k-1} + \deg\odot_{k-1}$.
The sum of such operations is $\hat\Delta$.
\item[(II)]
Applying $\frak m_k$ to $x_{i+1}  \otimes \dots \otimes x_{i+k}$.
In other words 
\begin{equation}
\aligned
&x_1 \odot_1 \dots \odot_{i-1} x_{i} \odot_{i} x_{i+1} 
\otimes \dots \otimes x_{i+k}\odot_{i+k+1} \dots  \odot_m x_m \\
&\mapsto 
(-1)^*  x_1 \odot_1 \dots \odot_{i} \frak m_{k}(x_{i+1},\dots,x_{i+k}) \odot_{i+k+1} \dots \odot_m x_m
\endaligned\end{equation}
Here we require 
$\odot_{i+1} = \dots = \odot_{i+k}= \otimes$. 
The sign is $* = \deg'x_1 + \deg \odot_1 + \dots +\deg' x_{k-1} +  \deg \odot_{i}$.
\end{enumerate}
For a given $k$ we denote the sum of such operations by $\partial_k$.
The sum $\widehat{\Delta} + \sum_{k=1}^{\infty}\partial_k$ is 
the co-boundary operator of $AB\mathscr C$.
\begin{defn}
We define the operator $S$ as the sum of the following operations
\begin{equation}\label{defnS}
\aligned
&x_1 \odot_1 \dots \odot_{k-1} x_k \boxtimes x_{k+1} \odot_{k+2} \dots  \odot_m x_m \\
&\mapsto 
(-1)^*  x_1 \odot_1 \dots \odot_{k-1} x_k \otimes x_{k+1} \odot_{k+2} \dots \odot_m x_m.
\endaligned
\end{equation}
Namely we replace one of $\boxtimes$ to $\otimes$.
The sign is $* = \deg'x_1 + \deg \odot_1 + \dots + \deg\odot_{i}$.
\end{defn}

One can easily check the next formula:
\begin{equation}\label{new7576}
\left\{
\aligned
&(\hat\Delta \circ S + S \circ \hat\Delta)({\bf x}) = m-1, \\
&\partial_1 \circ S + S \circ \partial_1 = 0.
\endaligned
\right.
\end{equation}
Here $m$ is as in (\ref{defnofcirc}). Moreover we have
\begin{equation}\label{new7676}
\partial_k(\frak S^m AB\mathscr C)
\subseteq \frak S^{m-1} AB\mathscr C
\end{equation}
for $k\ge 2$.
Therefore  
\begin{equation}\label{isomor77}
H\left(\frac{\frak S^m AB\mathscr C}{\frak S^{m-1} AB\mathscr C},\frak M_1\right) = 0
\end{equation}
for $m  > 1$.
(We use the fact that the ground ring is a field of characteristic $0$ 
here.)
It implies that $I = I_1$ induces an isomorphism on cohomologies.
\par
More precisely (\ref{isomor77}) implies that the cohomology of $\overset{\circ}AB\mathscr C(a,b)$ 
with $\frak M_1$ as coboundary operator is isomorphic 
to the cohomology of $\mathscr C(a,b)$ with $\frak m_1$ as coboundary operator.
We need to show that the inclusion induces an isomorphism between 
the cohomology of $\overset{\circ}AB\mathscr C(a,b)$ 
and the cohomology of its completion $AB\mathscr C(a,b)$.
The proof of this fact is similar to the proof of Lemma \ref{lemma217} or 
of Theorem \ref{theorem63}.
So we omit it.
\par
To complete the proof it suffices to find $I_k$ for $k=2,\dots$ such that 
$\frak I = (I_1,I_2,\dots)$ becomes an $A_{\infty}$ functor.
We define:
$
I_k: \mathscr C[1](c_0,c_1) \widehat\oslash \cdots \widehat\oslash \mathscr C[1](c_{k-1},c_k)
\to AB\mathscr C(c_0,c_k)
$
by
\begin{equation}\label{defnofIk}
I_k(x_1 \oslash \cdots \oslash x_k)
:=
x_1 \otimes \cdots \otimes x_k.
\end{equation}
Namely we replace all the $\oslash$ to $\otimes$.
Note that the right hand side is an element of $B_k\mathscr C[1](c_0,c_k) =B\mathscr C(c_0,c_1)
\subseteq A_1B\mathscr C(c_0,c_1)$.
It is straitforward to check that (\ref{defnofIk}) defines an $A_{\infty}$ functor.
\end{proof}
\begin{rem}
The fact that the composition of the Bar and the Co-Bar resolutions is equivalent to the 
identity functor is well established.
(See for example \cite{keller}.)  Since the author could not find the simple formula 
(\ref{defnofIk}) in the literature, he includes it here.
Note that even in the case when $\mathscr C$ is a DG-category the map $I_2$
is non-zero. Namely 
$\frak I$ is not a DG functor. So in the literature the process of inverting the 
weak equivalence is used. By using the language of $A_{\infty}$ category we can 
explicitly construct a homotopy equivalence. 
\end{rem}

\begin{proof}[Proof of Theorem \ref{theorem63}]
We begin with the following definition.
\begin{defn}
Let $(C_i,C_{i,0})$ be a completed cochain complex for $i=0,1,2$
and $\frak i_i: (C_i,C_{i,0}) \to (C_{i+1},C_{i+1,0})$
a completed injection for $i=0,1$.
We say $\frak i_1$ is {\it almost cochain homotopy equivalance 
relative to $C_0$} if the following holds.
For any finitely generated subcomplex $W$ of $C_{2,0}$ and 
$\epsilon>0$ there exists $G: W \to C_{2,0}$ such that
\begin{enumerate}
\item The image of $dG + Gd - T^{\epsilon}$ is in $C_{1,0}$.
\item  $G = 0$ on $C_{0,0}$.
\item  $G$ preserves energy filtration.
\end{enumerate}
We call $G$ a {\it $\epsilon$-$W$ shrinker relative to $C_0$}.
\end{defn}
\begin{lem}\label{lem86}
Let $(C_i,C_{i,0})$ be a completed cochain complex for $i=1,2,3$ ,
and  $\frak i_i: (C_i,C_{i,0}) \to (C_{i+1},C_{i+1,0})$ a
completed injection for $i=0,1,2$.
Suppose $\frak i_1$ and $\frak i_2$ are 
almost cochain homotopy equivalences relative to $C_0$. 
Then the composition 
$\frak i_2\circ\frak i_1: (C_1,C_{1,0}) \to (C_3,C_{3,0})$ is an almost cochain homotopy equivalence relative to $C_0$.
\end{lem}
\begin{proof}
Let $W \subset C_{3,0}$ and $\epsilon > 0$.
We take $\epsilon-W$ shrinker $G_1$ relative to $C_0$ (for $\frak i_2$).
We take $W^+$ which contains $W$, $G_1W$. 
We then take $\epsilon-{W^+}$-shrinker $G_2$
relative to $C_0$ (for $\frak i_1$).
We observe 
$$
(T^{\epsilon} - (dG_2 + G_2d))\circ (T^{\epsilon} - (dG_1 + G_1d))(W)
\subset C_1.
$$
Since
$
(T^{\epsilon} - (dG_2 + G_2d))\circ (T^{\epsilon} - (dG_1 + G_1d))
= 
T^{2\epsilon} - T^{\epsilon}(d (G_1+ G_2) + (G_1+ G_2) d)
+ (dG_2G_1 + G_2G_1d)
$, the map
$G := T^{\epsilon}(G_1+ G_2) - G_2G_1$  is a $2\epsilon-W$ shrinker 
relative to $C_0$ we 
look for.
\end{proof}
\begin{lem}\label{lemma8dayto}
Let $(C_i,C_{i,0})$ $i=1,2,\dots$ and $(C,C_0)$ be completed cochain complexes, 
and the cochain homomorphisms $(C_i,C_{i,0}) \to (C_{i+1},C_{i+1,0})$ completed injections.
We assume
$
C_0  = \bigcup_i  C_{i,0}.
$
We also assume that  $(C_i,C_{i,0}) \to (C_{i+1},C_{i+1,0})$ is an almost cochain homotopy equivalence 
for $i=1,2,\dots$.
\par
Then $(C_1,C_{1,0}) \to (C,C_{0})$ is an almost cochain homotopy equivalence.
\end{lem}
\begin{proof}
By Lemma \ref{lem86}
the inclusion $C_1 \to C_k$ is an almost cochain homotopy equivalence.
Since any finitely generated subcomplex of $C_0$ is contained 
in $C_{k,0}$ for some $k$, the lemma follows.
\end{proof}
Therefore to prove Theorem \ref{theorem63} 
it suffices to show the next lemma.
\begin{lem}\label{lem8888}
The inclusion 
$\frak S^{m-1}A\mathscr B^{\infty}[-1](a,b) \to \frak S^mA\mathscr B^{\infty}[-1](a,b)$
is an almost cochain homotopy equivalence relative to 
${\mathscr C}^{\infty}(a,b) = \frak S^{1}A\mathscr B^{\infty}[-1](a,b)$.
\end{lem} 
\begin{proof}
We begin with the following sublemma.
\begin{sublem}\label{sublem7878}
Let $W$ be a finitely generated subcomplex of 
$
\frak S^{m}\overset{\circ}A\mathscr B^{\infty}(a,b).
$
Then for any sufficiently large $n$ 
there exists $W^n \subset 
\frak S^{m}\overset{\circ}AB\mathscr C^{n}(a,b)_{\Lambda}$
such that
\begin{enumerate}
\item
$W^n$ is a subcomplex.
\item
The restriction of  $A\widehat{\Phi}^{\infty,n}$
defines an isomorphism $W^n \to W$.
\item
For $w \in W^n$ we have
\begin{equation}\label{shikinormno}
\vert \frak v(w)- \frak v(A\widehat{\Phi}^{\infty,n}(w))
\vert \le \epsilon_n
\end{equation}
with $\lim_{m\to \infty}\epsilon_n = 0$.
\item
$$
A\widehat{\Phi}^{\infty,n}\left(
\frak S^{0}\overset{\circ}AB\mathscr C^{n}[-1](a,b) \cap W_n  
\right)
=
\frak S^{0}\overset{\circ}A\mathscr B^{\infty}[-1](a,b)\cap W.
$$
\end{enumerate}
\end{sublem}
\begin{proof}
Let $y_1,\dots,y_k$ be a basis of $W_{\Lambda}$
such that $\frak v(y_i) = 0$. 
(Since $W$ is zero energy generated there exists such a basis.)
Then taking $n_0$ large there exists 
$\hat y_i \in   \frak S^{m}\overset{\circ}AB\mathscr C^{n_0}(a,b)_{\Lambda}$
such that $A\widehat{\Phi}^{\infty,n_0}(\hat y_i) = y_i$
and 
$
\vert \frak v(\hat y_i)- \frak v(y_i)
\vert \le \epsilon_{n_0}.
$
Since $W$ is a subcomplex there exists $a_{i,j} \in \Lambda_0$
such that
$
\frak M_1(y_i) = \sum_j  a_{i,j}y_j.
$
Therefore for sufficiently large $n$, 
$y'_i:= A\widehat{\Phi}^{n,n_0}(\hat y_i)$ satisfies
the equation
$
\frak M_1(y'_i) = \sum_j  a_{i,j}y'_j.
$
The subcomplex $W^n = \frak F^0(\sum \Lambda y'_i)$
has the properties (1)(2)(3).
Since total number filtration is preserved by 
$A\widehat{\Phi}^{\infty,n_0}$ we may 
choose $\hat y_i$ so that (4) is satisfied.
\end{proof}
Now we prove Lemma \ref{lem8888}.
Let $W$ be a finitely generated subcomplex of 
$
\frak S^{m}\overset{\circ}A\mathscr B^{\infty}(a,b).
$
Let $W^n$ be as in Sublemma \ref{sublem7878}.
We put
\begin{equation}\label{form810}
G := \frac{T^{(k+2)\epsilon}}{m-1} A\widehat{\Phi}^{\infty,n} \circ S \circ 
A\widehat{\Phi}^{n,\infty}_W.
\end{equation}
Here $S$ in the right hand side is as in Definition \ref{defnS}
and $A\widehat{\Phi}^{n,\infty}_W$ 
is the inverse of $A\widehat{\Phi}^{\infty,n}: W_n \to W$.
(\ref{new7576}), (\ref{new7676}), (\ref{shikinormno}), (\ref{filter666}) imply that
$$ 
T^{{(k+2)\epsilon}} - (G\circ \frak M_1 + \frak M_1 \circ G)
$$
sends $W^+ \cap \frak S^{m}\overset{\circ}A\mathscr B^{\infty}[-1](a,b)$
to $\frak S^{m-1}\overset{\circ}A\mathscr B^{\infty}[-1](a,b)$.
Note 
$S$ in Definition \ref{defnS} is zero 
on $\frak S^0AB\mathscr C^n(a,b)$. Therefore $G$ in (\ref{form810}) is 
zero on $\frak S^{0}\overset{\circ}A\mathscr B^{\infty}[-1](a,b)\cap W$
by Sublemma \ref{sublem7878} (4).
The fact that $G$ preserves energy filtration follows from 
(\ref{shikinormno}).
The proof of Lemma \ref{lem8888} is complete.
\end{proof}
The proof of Theorem \ref{theorem63} is now complete.
\end{proof}

\section{Reduced Bar-Co-bar resolution.}
\label{sec;reduced}

As we will see later (See Lemma \ref{Anounitality}.) it is hard to study 
unitality of $A\mathscr B^{\infty}$ and $AB\mathscr C$.
By this reason we study its reduced version.
We first consider $AB\mathscr C$.
Supose $\mathscr C$ has strict unit ${\bf e}_c$.
We consider a sub $\Lambda_0$ module $\overset{\circ}{\mathscr I}(\mathscr C)$ of $\overset{\circ}AB\mathscr C$
which is generated by
\begin{equation}\label{defnofcirc2}
{\bf x} = x_1 \odot_1 x_2 \odot_2 \dots \odot_{m-1} x_{m-1} \odot x_m
\end{equation}
such that one of $x_2$,\dots,$x_{m-1}$
is ${\bf e}_c$.
Let ${\mathscr I}(\mathscr C)$ be the closure of $\overset{\circ}{\mathscr I}(\mathscr C)$
in $AB\mathscr C$.
\begin{lem}\label{lemreducu}
$\mathscr I(\mathscr C)$ is an ideal of DG-category $AB\mathscr C$.
$\mathscr I(\mathscr C)$ is preserved by the operators $\hat \Delta$, 
$\partial_k$ and $S$.
\end{lem}
We remark that we say $\mathscr I(\mathscr C)$ is an ideal 
if ${\bf x} \in \mathscr I(\mathscr C)(a,b)$
 and ${\bf y} \in AB\mathscr C(b,c)$,
${\bf z} \in AB\mathscr C(d,a)$
then $\frak M_2({\bf x},{\bf y}) \in \mathscr I(\mathscr C)(a,c)$, 
$\frak M_2({\bf z},{\bf x}) \in \mathscr I(\mathscr C)(d,c)$.
The lemma is easy to prove.
\par
\begin{defn}
We denote by $A^{\rm red}B\mathscr C$
the quotient of $AB\mathscr C$ by the ideal $\mathscr I(\mathscr C)$.
\par
Lemma \ref{lemreducu} implies that $A^{\rm red}B\mathscr C$ 
is a completed DG-category with finite part
$
\overset{\circ}AB\mathscr C:= 
{\overset{\circ}AB\mathscr C}/{\overset{\circ}{\mathscr I}(\mathscr C)}
$.
The total number filtration $\frak S^kAB\mathscr C$ of $AB\mathscr C$ 
induces $\frak S^kA^{\rm red}B\mathscr C$ for $A^{\rm red}B\mathscr C$.
The energy filtration $\frak F^{\lambda}A^{\rm red}B\mathscr C$ of 
$A^{\rm red}B\mathscr C$ is also induced from the 
energy filtration of $AB\mathscr C$.
\end{defn}
\begin{lem}
There exists a homotopy equivalence:
$\frak I : \mathscr C \to A^{\rm red}B\mathscr C$.
\end{lem}
Using the fact that $S$ preserves $\frak I$, 
the proof is the same as the 
proof of Proposition \ref{Theorem71}.
\par
Let $(\{\mathscr C^n\},\{\Phi^n\})$ be a unital 
inductive system of filtered $A_{\infty}$ categories.
We observe 
$
A\widehat{\Phi}^n(\mathscr I(\mathscr C^n)) \subseteq 
\mathscr I(\mathscr C^{n+1}).
$
This is a consequence of strict unitality of 
$\widehat{\Psi}^n$. Therefore
we define
$
\overset{\circ}{\mathscr I^{\infty}} \subseteq \overset{\circ}A\mathscr B^{\infty}
$
by
$
\bigcup_{k} \bigcup_n A\widehat{\Phi}^{\infty,n}(\frak S^k\overset{\circ}AB\mathcal 
\mathscr C^n)
\cap  \mathscr I(\mathscr C^n))
$
and denote its closure by $\mathscr I^{\infty}$.
\begin{defn}
We denote by $A^{\rm red}\mathscr B^{\infty}$
the quotient of $A\mathscr B^{\infty}$ by the ideal $\mathscr I^{\infty}$.
\par
$A^{\rm red}\mathscr B^{\infty}$ 
is a completed DG-category with finite part
$
\overset{\circ}A\mathscr B^{\infty}:= 
{\overset{\circ}A\mathscr B^{\infty}}/{\overset{\circ}{\frak I}}
$.
The total number filtration 
$\frak S^kA^{\rm red}\mathscr B^{\infty}$ 
and the energy filtration $\frak F^{\lambda}A^{\rm red}\mathscr B^{\infty}$
of $A^{\rm red}\mathscr B^{\infty}$
are induced by those on $A^{\rm red}B\mathscr C^n$.
\end{defn}
\begin{lem}
The completed injection $I: ({\mathscr C}^{\infty}(a,b),\frak m_1) \to (A^{\rm red}\mathscr B^{\infty}(a,b),\frak M_1)
$
is an almost cochain homotopy equivalence.
\end{lem}

Using Lemma \ref{lemreducu} the proof is the same as the proof
of Theorem \ref{theorem63}.

\section{The triangle inequality of Gromov-Hausdorff distance.}
\label{sec;triangle}

In this section we prove Theorems \ref{triangleGH} and \ref{triangleharder}.
We prove Theorem \ref{triangleGH} first.
\par
Let $\mathscr C_i$ be a gapped and unital filtered $A_{\infty}$ category for $i=1,2,3$.
We assume $d_{\rm GH}(\mathscr C_i,\mathscr C_{i+1})
< \epsilon_i$ for $i=1,2$.
We will prove
$d_{\rm GH}(\mathscr C_1,\mathscr C_{3}) < \epsilon_1+\epsilon_2$.
By definition there exist gapped and unital 
filtered $A_{\infty}$ categories $\mathscr C_{12},  \mathscr C_{23}$
and homotopy equivalences to the images:
$\frak I_{12,1}:  \mathscr C_1 \to \mathscr C_{12}$, 
$\frak I_{12,2}:  \mathscr C_2 \to \mathscr C_{12}$,
$\frak I_{23,2}:  \mathscr C_2 \to \mathscr C_{23}$, 
$\frak I_{23,3}:  \mathscr C_3 \to \mathscr C_{23}$,
such that
$
d_{\rm H}(\frak{OB}(\mathscr C_1),\frak{OB}(\mathscr C_2))
< \epsilon_1$,
$
d_{\rm H}(\frak{OB}(\mathscr C_2),\frak{OB}(\mathscr C_3))
< \epsilon_2.
$
We first replace all the filtered $A_{\infty}$ categories 
with the canonical model.
Namely we use:
\begin{thm}
For any gapped and unital filtered $A_{\infty}$ category $\mathscr C$ 
there exists a gapped filtered and unital $A_{\infty}$ category $\mathscr H$
and a gapped and unital homotopy equivalence $\frak F: \mathscr C \to \mathscr H$
such that the operation $\overline{\frak m}_1$ is zero on $\mathscr H$.
\end{thm}
This is \cite[Theorem 5.4.1]{fooobook} in the case of 
filtered $A_{\infty}$ algebra. The proof given there can be 
used to prove the category case without any change.

We may assume that $\mathscr C_i$, $i=1,2,3$ and 
$\mathscr C_{12}$, $\mathscr C_{23}$ are canonical, 
that is, $\overline{\frak m}_1$ is zero.
We may also assume that $\frak I_{12,1}$, $\frak I_{12,2}$, 
$\frak I_{23,2}$, $\frak I_{23,3}$ are bijections on object.
\par
We note that a homotopy equivalence $\frak F: \mathscr C \to \mathscr C'$ between canonical and unital filtered $A_{\infty}$
categories which is a bijection on objects has a (strict) inverse $\frak G$.
Namely the compositions $\frak G\circ \frak F$ and $\frak F\circ \frak G$
are {\it equal to} the identity functors.
\par
We next replace filtered $A_{\infty}$ categories by DG-categories.
We can do so by using the composition of Bar and Co-Bar resolutions.
Namely
we consider $AB\mathscr C_i$, $i=1,2,3$ and 
$AB\mathscr C_{12}$, $AB\mathscr C_{23}$.
Then $\frak I_{12,1}$, $\frak I_{12,2}$, 
$\frak I_{23,2}$, $\frak I_{23,3}$ induce 
DG-functors among them. Note that a DG-functor $\Psi$ assigns 
{\it linear} maps $AB\mathscr C_{ij}(a,b) \to AB\mathscr C_{ij}(\Psi_{\rm ob}(a),\Psi_{\rm ob}(b))$
which are compatible with compositions and co-boundary operators.
They induce cochain homotopy equivalences between 
morphism complexes. (However they are {\it not} isomorphisms.)
Therefore to prove  Theorem \ref{triangleGH} it suffices to prove 
the next proposition.
\qed

\begin{prop}\label{prop104}
Let $\mathscr C_i$ ($i=1,2,3$) and $\mathscr C_{12},\mathscr C_{23}$
be gapped DG-categories.
Let 
$\frak I_{12,1}:  \mathscr C_1 \to \mathscr C_{12}$,
$\frak I_{12,2}:  \mathscr C_2 \to \mathscr C_{12}$,
$\frak I_{23,2}:  \mathscr C_2 \to \mathscr C_{23}$,
$\frak I_{23,3}:  \mathscr C_3 \to \mathscr C_{23}$,
be DG-functors such that:
\begin{enumerate}
\item Those are injective on objects.
\item
For any $a,b \in \frak{OB}(\mathscr C_1)$, the functor 
$\frak I_{12,1}$ induces an injective cochain homomorphism
$\mathscr C_1(a,b) \to \mathscr C_{21}((\frak I_{12,1})_{\rm ob}(a),(\frak I_{12,1})_{\rm ob}(b))$.
The same holds for $\frak I_{12,2}$, $\frak I_{23,2}$, $\frak I_{23,3}$.
\item The cochain homomorphisms in (2) induces an isomorphisms on $\overline{\frak m}_1$ homology.
\end{enumerate}
Then
$$
d_{\rm GH}(\mathscr C_1,\mathscr C_3)
\le
d_{\rm H}(\frak{OB}(\mathscr C_1),\frak{OB}(\mathscr C_2))
+ 
d_{\rm H}(\frak{OB}(\mathscr C_2),\frak{OB}(\mathscr C_3)).
$$
\end{prop}
\begin{proof}
We may and will assume
$
\frak{OB}(\mathscr C_{ij}) = \frak{OB}(\mathscr C_{i}) \sqcup 
\frak{OB}(\mathscr C_{j})
$
for $(ij) = (12),(23)$.
We will define a DG-category $\mathscr C_{13}$ below.
We first put
$$
\frak{OB}(\mathscr C_{13}) = \frak{OB}(\mathscr C_{1}) \sqcup 
\frak{OB}(\mathscr C_{2})
\sqcup 
\frak{OB}(\mathscr C_{3}).
$$
We will define morphism complexes.
Roughly speaking we do so by adding `composed morphisms' formally 
when it is not defined already.
The detail follows.
Let $a,b \in \frak{OB}(\mathscr C_{2})$.
We have three cochain complexes
$
\mathscr C_{12}(a,b)$, 
$
\mathscr C_{23}(a,b)$, 
$
\mathscr C_{2}(a,b).
$
There are cochain homotopy equivalences
$
\mathscr C_{2}(a,b) \to \mathscr C_{12}(a,b)$, 
$
\mathscr C_{2}(a,b) \to \mathscr C_{23}(a,b)$.
We remark that they are induced by the injective 
map of $R$ vector spaces.
Therefore they are spit injective.
We first define
$$
\Breve{\mathscr C}_{13}(a,b) = \bigoplus_{\text{$\alpha,\beta=(12)$ or $(23)$}} \Breve{\mathscr C}^{\alpha\beta}_{13}(a,b)
$$
where
\begin{equation}\label{form8383}
\aligned
\Breve{\mathscr C}^{(12)(12)}_{13}(a,b)
= \widehat{\bigoplus}\, \mathscr C_{12}(a_0,b_0) \widehat\otimes_{\Lambda_0}
& \mathscr C_{23}(b_0,a_1) 
\\ 
&\widehat\otimes_{\Lambda_0}\dots\widehat\otimes_{\Lambda_0} \mathscr C_{12}(a_k,b_k).
\endaligned
\end{equation}
Here the direct sum is taken over all $k=0,1,2,\dots$ 
and $a_i,b_i \in \frak{OB}(\mathscr C_2)$
such that $a_0 = a$, $b_k =b$ and $\widehat{\bigoplus}$ is the completion of the 
direct sum.
The definition of $\Breve{\mathscr C}^{ij}_{13}(a,b)$ for $i,j=(12)$,or $(21)$, or $(22)$ is similar.
See Figure \ref{Fig222}.
\par
We define its submodule as follows.
Let 
$$
{\bf x} = x_0 \otimes y_1 \otimes  \dots \otimes y_{k} \otimes x_k
$$
be an element of $\Breve{\mathscr C}^{(12)(12)}_{13}(a,b)$.
Suppose $y_i \in \mathscr C_{2}(b_{i-1},a_i)$.
Then we put
$$
{\bf x}'_i 
= x_0 \otimes  \cdots \otimes (x_{i-1}\circ y_i \circ x_{i+1}) \otimes  \dots y_{k} \otimes x_k
\in \Breve{\mathscr C}^{(12)(12)}_{13}(a,b).
$$
where $\circ$ is the composition in $\mathscr C_{12}$.
\par
Suppose $x_i \in \mathscr C_{2}(a_i,b_i)$ for $i=1,\dots,k-1$ then
we put
$$
{\bf x}''_i 
= x_0 \otimes \dots \otimes (y_{i-1}\circ x_i \circ y_{i}) \otimes  \dots y_{k} \otimes x_k\in \Breve{\mathscr C}^{(12)(12)}_{13}(a,b).
$$
where $\circ$ is the composition in $\mathscr C_{21}$.
If $x_0 \in \mathscr C_{2}(a_0,b_0)$
then we put
$$
{\bf x}''_0 
= (x_0\circ y_1) \otimes  \dots \otimes y_{k} \otimes x_k\in \Breve{\mathscr C}^{(21)(12)}_{13}(a,b).
$$
If $x_k \in \mathscr C_{2}(b_k,a_k)$
then we put
$$
{\bf x}''_k 
= x_0 \otimes y_1 \otimes  \dots \otimes 
(y_{k} \circ x_k)\in \Breve{\mathscr C}^{(12)(21)}_{13}(a,b).
$$
For ${\bf x} \in \Breve{\mathscr C}^{\alpha\beta}_{13}(a,b)$
for $(\alpha\beta) \ne ((12),(12))$ we define 
${\bf x}'_i$, ${\bf x}''_i$ in a similar way.
We divide $\Breve{\mathscr C}_{13}(a,b)$
by the closure of the submodules generated 
all of ${\bf x} -{\bf x}'_i$, ${\bf x} -{\bf x}''_i$
and denote the quotient by 
${\mathscr C}_{13}(a,b)$.
\par
The co-boundary operators of $\mathscr C_{12}$ and of $\mathscr C_{21}$
induce co-boundary operators on $\Breve{\mathscr C}_{13}(a,b)$.
It is easy to see that it induces a co-boundary operator 
on ${\mathscr C}_{13}(a,b)$.
\par
For $a \in \frak{OB}(\mathscr C_1)$, $b \in \frak{OB}(\mathscr C_2)$
we define:
$$
\aligned
\Breve{\mathscr C}_{13}(a,b) &= \mathscr C_{12}(a,b) \oplus \widehat{\bigoplus_{c\in \mathscr C_2}}
\mathscr C_{12}(a,c) \widehat\otimes_{\Lambda_0}  
(\Breve{\mathscr C}^{(23)(12)}_{13}(c,b) \oplus \Breve{\mathscr C}^{(23)(23)}_{13}(c,b)), 
\\
\Breve{\mathscr C}_{13}(b,a) &= \mathscr C_{12}(b,a) \oplus \widehat{\bigoplus_{c\in \mathscr C_2}}
 (\Breve{\mathscr C}^{(23)(23)}_{13}(b,c) \oplus \Breve{\mathscr C}^{(12)(23)}_{13}(b,c)) \widehat\otimes_{\Lambda_0}  
\mathscr C_{12}(c,a).
\endaligned
$$
We take a quotient in a similar way to obtain ${\mathscr C}_{13}(a,b)$, 
${\mathscr C}_{13}(b,a)$.
\par
For $a,b \in \frak{OB}(\mathscr C_1)$, we define:
(See Figure \ref{Fig2223}.)
$$
\Breve{\mathscr C}_{13}(a,b)
= \mathscr C_{12}(a,b)
\oplus
\widehat{\bigoplus_{c,c'\in \mathscr C_2}}
\mathscr C_{12}(a,c)\widehat\otimes_{\Lambda_0}
\Breve{\mathscr C}^{(23)(23)}_{13}(c,c') \widehat\otimes_{\Lambda_0} \mathscr C_{12}(c',b).
$$
We take a quotient in a similar way to obtain ${\mathscr C}_{13}(a,b)$.
\par
For $a \in \frak{OB}(\mathscr C_1)$, $b \in \frak{OB}(\mathscr C_3)$
we define: (See Figure \ref{Fig3323}.)
$$
\aligned
\Breve{\mathscr C}_{13}(a,b)
= &\widehat{\bigoplus_{c\in \mathscr C_2}}\mathscr C_{12}(a,c) \widehat\otimes 
\mathscr C_{23}(c,b) \\
&\oplus
\widehat{\bigoplus_{c,c'\in \mathscr C_2}}
\mathscr C_{12}(a,c)\widehat\otimes_{\Lambda_0}
\Breve{\mathscr C}^{(23)(12)}_{13}(c,c') \widehat\otimes_{\Lambda_0} \mathscr C_{23}(c',b).
\endaligned
$$
We take a quotient in a similar way to obtain ${\mathscr C}_{13}(a,b)$.
\par
The other cases are similar.
\begin{center}
\begin{figure}[h]
 \begin{tabular}{cc}
 \begin{minipage}[t]{0.30\hsize}
\centering
\includegraphics[scale=0.3]
{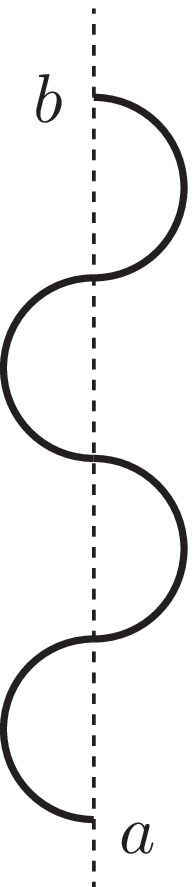}
\caption{$\mathscr C_{13}^{(12)(21)}$}
\label{Fig222}
\end{minipage} &
 \begin{minipage}[t]{0.35\hsize}
\centering
\includegraphics[scale=0.3]
{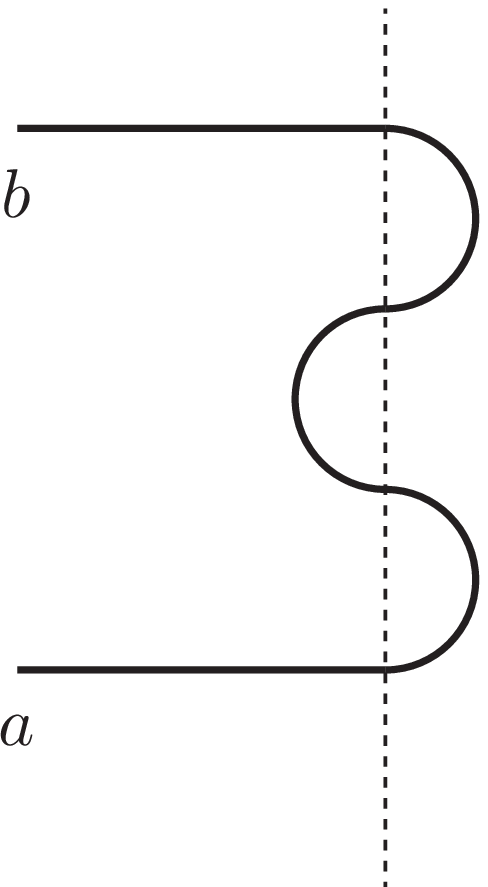}
\caption{From $\mathscr C_1$ to $\mathscr C_1$}
\label{Fig2223}
\end{minipage}
 \begin{minipage}[t]{0.35\hsize}
\centering
\includegraphics[scale=0.3]
{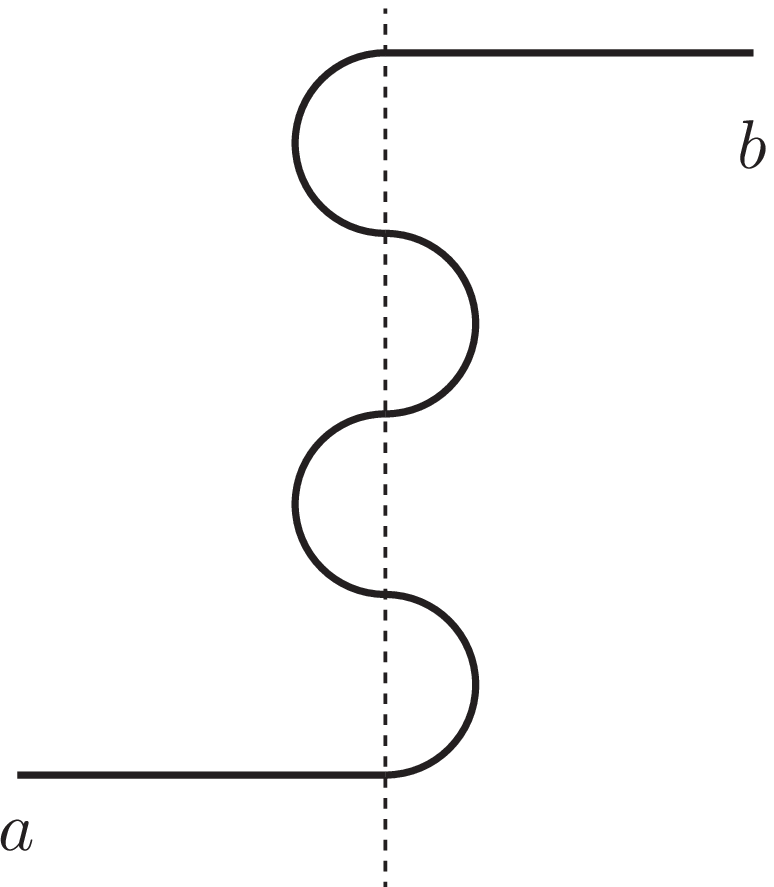}
\caption{From $\mathscr C_1$ to $\mathscr C_3$}
\label{Fig3323}
\end{minipage}
\end{tabular}
\end{figure}
\end{center}
\par
We define compositions in $\mathscr C_{13}$ by first taking 
'formal composition' and then contract if composition is defined 
in ${\mathscr C}_{12}$ or in ${\mathscr C}_{23}$.
For example if
$$
x_1 \otimes \dots \otimes x_k \in \Breve{\mathscr C}^{(12)(12)}_{13}(a,b),
\quad
y_1 \otimes \dots \otimes y_{\ell} \in \Breve{\mathscr C}^{(12)(12)}_{13}(b,c)
$$
with $a,b,c \in \frak{OB}(\mathscr C_2)$ then
$$
(x_1 \otimes \dots \otimes x_k) \circ
(y_1 \otimes \dots \otimes y_{\ell})
:= 
x_1 \otimes \dots \otimes (x_k\circ y_1)
\otimes \dots \otimes y_{\ell},
$$
where $\circ$ in the right hand side is the composition in $\mathscr C_{12}$.
It is easy to see that it induces a composition on $\mathscr C_{13}$.
We thus obtained a filtered DG-category $\mathscr C_{13}$.
\par 
There is an obvious DG-functor $\frak I_1 : \mathscr C_{1} \to \mathscr C_{13}$.
\begin{lem}\label{lem85}
For $a,b \in \frak{OB}(\mathscr C_{1})$ 
the DG-functor $\frak I_1$ induces an isomorphism 
$
H({\mathscr C}_{1}(a,b),\overline{\frak m}_1) \to 
H({\mathscr C}_{13}(a,b),\overline{\frak m}_1)
$
of $\overline{\frak m}_1$ cohomologies.
\end{lem}
\begin{proof}
The isomorphism 
$
H({\mathscr C}_{1}(a,b),\overline{\frak m}_1) \cong
H({\mathscr C}_{12}(a,b),\overline{\frak m}_1)
$
is the assumption.
We define a filtration on $\frak N_m\mathscr C_{13}(a,b)$ 
for $a,b \in \frak{OB}(\mathscr C_2)$ as follows.
We first filter $\Breve{\mathscr C}_{13}(a,b)$
by the number of tensors.
More precisely elements of $\frak N_m\Breve{\mathscr C}_{13}(a,b)$ 
are those written as at most $2m$ tensor products.
Then it induces a  filtration $\frak N_m\mathscr C_{13}(a,b)$ on $\mathscr C_{13}(a,b)$.
\par
Since $\mathscr C_2(a,b) \to \mathscr C_{12}(a,b)$ induces an 
isomorphism in $\overline{\frak m}_1$ cohomology
for $a,b \in \frak{OB}(\mathscr C_2)$,
there exists 
$G: \mathscr C_{12}(a,b) \to \mathscr C_{12}(a,b)$
of degree $+1$ such that 
\begin{enumerate}
\item 
The image of 
$dG + Gd - {\rm id}$ is in $\mathscr C_{2}(a,b)$.
\item
$G = 0$ in $\mathscr C_{2}(a,b)$.
\end{enumerate}
In fact $\overline{\frak m}_1$ is obtained 
from an $R$-linear map on $R$ vector spaces by taking $\hat\otimes_{R}\Lambda_0$
and an injective cochain map over $R$ which induces an isomorphism 
on cohomology has such $G$.
We choose a splitting
$$
\mathscr C_{12}(a,b) = \mathscr C_{2}(a,b) \oplus \mathscr C^*_{12}(a,b),
\quad 
\mathscr C_{23}(a,b) = \mathscr C_{2}(a,b) \oplus \mathscr C^*_{23}(a,b),
$$
as zero energy generated $\Lambda_0$ modules.
Each element of $\frak N_k\mathscr C_{13}(a,b)$
has a representative which is a $T$-adic convergent sum of the elements of the form:
\begin{equation}\nonumber
{\bf x} = u \otimes x_0 \otimes y_1 \otimes \dots \otimes y_k \otimes x_k \otimes v
\end{equation}
where $x_i \in \mathscr C^*_{23}(a_i,b_{i+1})$, 
$y_i \in \mathscr C^*_{23}(b_i,a_{i})$
\par
We define
\begin{equation}\label{defshrink}
\hat G({\bf x}) = \frac{1}{m}\sum_{i=0}^k (-1)^* 
u \otimes x_0 \otimes y_1 \otimes \dots \otimes G(x_i) \otimes \dots \otimes y_k \otimes x_k \otimes v
\end{equation}
on $\mathfrak N_m{\mathscr C}_{13}(a,b)$.
The sign is by Kozule rule.
\par
We claim that
the image of $\hat G \circ d + d \circ \hat G - {\rm id}$
is in $\mathfrak N_{m-1}{\mathscr C}_{13}(a,b)$.
Moreover $\hat G$ is zero on $\mathfrak N_{0}{\mathscr C}_{13}(a,b)$.
\par
Note that $\mathscr C^*_{12}(a,b)$, $\mathscr C^*_{23}(a,b)$
are not invariant by the co-boundary operator.
Nevertheless we can prove this claim as follows.
We need to show that the next two operations cancel 
modulo elements of $\mathfrak N_{m-1}{\mathscr C}_{13}(a,b)$.
\begin{enumerate}
\item[(a)]
First apply co-boundary operator to the factor $y_i$
and then apply $G$ to the factor $x_i$.
\item[(b)]
First  apply $G$ to the factor $x_i$ then 
apply co-boundary operator to a factor $y_i$.
\end{enumerate}
We put $dy_i = z^*_i + z_i$ where $z_i \in \mathscr C_{2}(b_i,a_i)$,
$z^*_i \in \mathscr C^*_{12}(b_i,a_i)$,
$Gx_i = v^*_i + v_i$ where 
$v_i \in \mathscr C_{2}(a_i,b_{i+1})$,
$v^*_i \in \mathscr C^*_{21}(a_i,b_{i+1})$.
We observe both (a) and (b) give
$
\cdots v^*_i \otimes z^*_i \cdots
$
up to sign, modulo elements of $\mathfrak N_{m-1}{\mathscr C}_{13}(a,b)$, 
by definition. Then sign is opposite.
In this way we can find that the terms which appears by the fact 
that $\mathscr C^*_{12}(a,b)$, $\mathscr C^*_{23}(a,b)$
are not preserved by $d$ lies in $\mathfrak N_{m-1}{\mathscr C}_{13}(a,b)$.
We have proved the claim.
\par
It is easy to see that the claim implies the lemma.
\end{proof}
Lemma \ref{lem85} and Whitehead's theorem imply that 
$\frak I_1$ is a homotopy equivalence to the full subcategory.
It is easy to see that
$
d_{\rm H}(\frak{OB}(\mathscr C_1),\frak{OB}(\mathscr C_3))
\le \epsilon_1+\epsilon_2.
$
The proof of Theorem \ref{triangleGH} is now complete.
\end{proof}
We next prove Theorem \ref{triangleharder}.
Actually most of the proof is the same as the proof of 
Proposition \ref{prop104}.
In the situation of Proposition \ref{prop104}
we remove the gapped-ness assumption of DG-categories
but assume they are completed DG-categories.
In (3) we assume the cochain homomorphisms are 
almost homotopy equivalences.
\par
We construct $\mathscr C_{13}$
in a similar way. Its finite part 
is defined  replacing $\widehat{\bigoplus}$
$\widehat{\otimes}$ 
by ${\bigoplus}$
${\otimes}$ 
and $\mathscr C_{23}(b_1,a_1)$ etc. by its finite part 
in Formula 
(\ref{form8383}) etc.
\par
Then we prove the following variant of Lemma \ref{lem85}.
\begin{lem}\label{lem942}
For $a,b \in \frak{OB}(\mathscr C_1)$ the completed injection
$
({\mathscr C}_{1}(a,b),{\frak m}_1) \to 
({\mathscr C}_{13}(a,b),{\frak m}_1)
$
is an almost cochain homotopy equivalence.
\end{lem}
\begin{proof}
Let $a,b \in \frak{OB}(\mathscr C_1)$.
The filtration $\frak N_m \overset{\circ}{\mathscr C}_{13}(a,b)$ is defined in 
the same way as in the proof of Lemma \ref{lem85}.
Using Lemma \ref{lemma8dayto}, it suffices to prove that
$\frak N_{m-1} {\mathscr C}_{13}(a,b) \to \frak N_m {\mathscr C}_{13}(a,b)$
is an almost cochain homotopy equivalence
relative to ${\mathscr C}_{12}(a,b)$.
Let $W$ be a finitely generated subcomplex of 
$\frak N_m \overset{\circ}{\mathscr C}_{13}(a,b)$.
We observe the following.
There exists a finite subset $\frak O$ of $\frak{OB}(\mathscr  C_2)$
and a finitely generated subcomplex $U_k(c,c')$ 
of $\mathscr C_{21}(c,c')$, $U'_k(c,c')$ of $\mathscr C_{12}(c,c')$ for 
$c,c' \in \frak O$, a finitely generated subcomlex 
$U_k(a,c)$ (resp. $U_k(c,b)$) of $\mathscr C_{12}(a,c)$
(resp. $\mathscr C_{12}(c,b)$)
for $k=1,\dots,2m$ 
with the following properties.
\begin{enumerate}
\item
We require $U_k(c,c') \subseteq U_{k+1}(c,c')$, 
$U'_k(c,c') \subseteq U'_{k+1}(c,c')$, 
$U_k(a,c) \subset U_{k+1}(a,c)$,
and 
$U_k(c,b) \subseteq U_{k+1}(c,b)$.
\item They are preserved by the composition
in the following sense:
$$
\aligned
&U_{k_1}(c_1,c_2) \circ (U'_{k_2}(c_2,c_3) \cap \mathscr C_{2}(c_2,c_3))
\circ U_{k_3}(c_3,c_4)
\subseteq 
U_{k_1+k_2+k_3}(c_1,c_4), \\
&U'_{k_1}(c_1,c_2) \circ (U_{k_2}(c_2,c_3) \cap \mathscr C_{2}(c_2,c_3))
\circ U'_{k_3}(c_3,c_4)
\subseteq 
U'_{k_1+k_2+k_3}(c_1,c_4),\\
&U_{k_1}(a,c_1) \circ (U'_{k_2}(c_1,c_2)\cap \mathscr C_{2}(c_1,c_2))
\subseteq 
U_{k_1+k_2}(a,c_2), \\
&(U'_{k_1}(c_1,c_2)\cap \mathscr C_{2}(c_1,c_2))
 \circ U_{k_2}(c_2,b)\subseteq 
U_{k_1+k_2}(c_1,b).
\endaligned
$$
\item
Moreover we require that every element of $W$ is represented by a finite 
sum of  elements of the form
\begin{equation}\label{form92}
{\bf x} = u \otimes x_0 \otimes y_1 \otimes \dots \otimes y_k \otimes x_k \otimes v
\end{equation}
where 
$u \in U_1(a,c)$,  $x_i \in U'_1(c,c')$, 
$y_i \in U_1(c,c')$, $v \in U_1(c,b)$ with $c,c' \in \frak O$, $k\le m$.
\end{enumerate}
Let $W^+$ be the submodule generated by elements of the form 
(\ref{form92}) 
such that
$u \in U_{n_1}(a,a_0)$, 
$x_i \in U_{k_i}(a_i,b_{i+1})$,
$y_i \in U_{\ell_i}(b_i,a_i)$,
$v \in U_{n_2}(a_k,b)$
with $a_i,b_i \in  \frak O$ and $n_1 + n_2 + \sum k_i + \sum \ell_i
\le 2m$.
\par
We take complement 
$U_k^*(c,c')$ such that $U_k(c,c') = (U_k^*(c,c')
\cap \mathscr C_{2}(c,c')) \oplus U_k^*(c,c')$ 
and $U_k^*(c,c') \subseteq U_{k+1}^*(c,c')$.
We take $U_k^{\prime *}(c,c')$, $U_k^*(a,c)$, $U_k^*(c,b)$
in the same way.
\par
Now we take $\epsilon/m$-$U'_m(c,c')$ shrinker $G$ for each $c,c' \in \frak O$.
For an element ${\bf x} \in W^+$ of the form  (\ref{form92}) 
we assume that $u \in U^*_{n_1}(a,a_0)$, $x_i \in U^*_{k_i}(a_i,b_{i+1})$,  
$y_i \in U^*_{\ell_i}(b_i,a_i)$,
$v \in U^*_{n_2}(a_k,b)$.
Now we can define $\epsilon$-$W^+$ shrinker 
relative to ${\mathscr C}_{1}(a,b)$ by the formula
(\ref{defshrink}) times $T^\epsilon$.
\end{proof}
The proof of Theorem \ref{triangleharder} is complete.
\qed
\begin{rem}
We made rather cumbersome choice of subcomplexes $U_k(c,c')$ etc.
during the proof of Lemma \ref{lem942}.
We assumed that $W$ is invariant by the co-boundary operator.
However when we write an element of $W$ as a sum of 
elements of the form (\ref{form92})  each term 
may not be an element of $W$. 
In such a case the calculation appearing in the proof 
of Lemma \ref{lem85} cannot be carried out.
After modifying $W$ to $W^+$ this becomes possible.
\end{rem}

\section{Homotopical unitality of filtered $A_{\infty}$ functors: Review.}
\label{sec;alhomoequi}

We next discuss the unitality of the inductive limit.
Since the definition of the Gromov-Hausdorff distance uses 
the {\it strict} unit, the unitality of the inductive limit is essential 
to discuss well-defined-ness of the limit in the sense of 
(weak) equivalence. 
We first recall the notion of homotopical unitality,
which is discussed in detail in \cite[Section 3.2]{fooobook}.
There the case of filtered $A_{\infty}$ algebra is discussed. 
However we can adapt the argument to our category case easily.

\begin{defn}
Let $\mathscr C$ be a filtered $A_{\infty}$ category.
Suppose we are given an element ${\bf e}_c$ of 
$\mathscr C(c,c)$ with $\deg {\bf e}_c = 0$, 
$\frak v({\bf e}_c) = 0$  for each $c \in \frak{OB}(\mathscr C)$.
We say that $\mathscr C$ is {\it homotopically unital} and 
$\{{\bf e}_c \mid c \in \frak{OB}(\mathscr C)\}$
is its {\it homotopy unit}, if the following follows.
\par
We use symbols ${\bf e}^+_c$, ${\bf f}_c$ and put
\begin{equation}\label{formula91}
\left\{
\aligned
&\mathscr C^+(c,c):= \mathscr C(c,c) \oplus \Lambda_0 {\bf e}^+_c 
\oplus \Lambda_0 {\bf f}_c, 
\\
&\mathscr C^+(c,c'):= \mathscr C(c,c')
\qquad \text{if $c \ne c'$.}
\endaligned
\right.
\end{equation}
We define $\deg {\bf e}_c^+ = 0$, $\deg {\bf f}_c = -1$,\footnote{These are 
degrees {\it before} shifted. The shifted degrees of ${\bf e}_c^+$ 
and ${\bf f}_c$ are $-1$ and $-2$, respectively.}
and $\frak v({\bf e}_c^+) = \frak v({\bf f}_c) = 0$.
We remark $B\mathscr C(a,b) 
\subseteq B\mathscr C^+(a,b)$.
We require that there exists a filtered $A_{\infty}$ category 
such that:
\begin{enumerate}
\item
$\frak{OB}(\mathscr C^+) = \frak{OB}(\mathscr C)$.
The module of morphisms of $\mathscr C^+$ is as in (\ref{formula91}).
\item
The structure operations of $\mathscr C^+$ coincide with the 
structure operations of $\mathscr C$ on $\mathscr C(a,b)$.
\item
The elements ${\bf e}^+_c$ is a stric unit of the  filtered $A_{\infty}$ category 
$\mathscr C^+$.
\item
$
\frak m^+_1({\bf f}_c) - ({\bf e}^+_c - {\bf e}_c)
\in 
\Lambda_+ \mathscr C(c,c).
$
\item
The image of $\frak m^+$ is in $\mathscr C$ except the 
case described in items (3)(4).
\end{enumerate}
We call $\mathscr C^+$ a {\it unitalization} of $\mathscr C$.
\end{defn}
\begin{rem}
Note that for a given $\mathscr C$ and its homotopy unit ${\bf e}_c$,
the operation $\mathfrak m^+$ which satisfies (1)-(4) above may not 
be unique. When we say homotopy unit we include the 
operator $\mathfrak m^+$
as a part of the structure given. The same remark 
applies to a homotopically unital functor.
\end{rem}

If $\mathscr C$ has a strict unit $\{{\bf e}_c\}$
then it is also a homotopy unit.
In fact, we define
$
\frak m^+_1({\bf f}_c) = {\bf e}^+_c - {\bf e}_c.
$
The operation to the other elements including  ${\bf e}^+_c$,  ${\bf f}_c$
are $0$ except those determined by the strict unitality of ${\bf e}^+_c$.
It is easy to check that this gives a structure of unital 
filtered $A_{\infty}$ category on $\mathscr C^+$.
From now on when $\mathscr C$ is unital we always 
take this unital filtered $A_{\infty}$ structure on $\mathscr C^+$.

We next discuss homotopically unital filtered $A_{\infty}$ functors.

\begin{defn}
Let $\mathscr C_1$, $\mathscr C_2$ be homotopically unital filtered $A_{\infty}$ 
categories. 
A filtered $A_{\infty}$ functor $\Phi: \mathscr C_1\to\mathscr C_2$ with energy loss $\rho$ is said 
to be {\it homotopically unital} if it 
extends to a unital filtered $A_{\infty}$ functor $\Phi^+: \mathscr C^+_1\to\mathscr C^+_2$
with energy loss $\rho$.
We call $\Phi^+$ the {\it unitalization} of $\Phi$
\par
If $\mathscr C_2$ is unital and $\mathscr C_1$ is homotopically unital 
we say a filtered $A_{\infty}$ functor with energy loss $\Phi: \mathscr C_1\to\mathscr C_2$ 
is {\it $\#$-homotopically 
unital} if $\Phi$ extends to a unital filtered $A_{\infty}$ functor with energy loss
$\Phi^{\#}: \mathscr C^+_1\to\mathscr C_2$.
We call $\Phi^{\#}$ the $\#$-{\it unitalization} of $\Phi$.
\end{defn}
We use the next construction in Section \ref{sec;indhomoequi}.
\begin{defnlem}\label{deflem93}
If $\mathscr C$ is a unital filtered $A_{\infty}$ category
then there exists a unital filtered $A_{\infty}$ functor
$
\frak I : \mathscr C \to \mathscr C^+.
$
\end{defnlem}
\begin{proof}
Since $\mathscr C(c,c)$ is the completion of a free $\Lambda_0$ module 
we take a splitting (as a graded $\Lambda_0$ modules) :
\begin{equation}\label{unitsplitting}
\mathscr C(c,c) = \Lambda_0{\bf e}_c \oplus \mathscr C^*(c,c).
\end{equation}
We then split
$
\frak m_k: B_k\mathscr C[1](c,c) \to \mathscr C[1](c,c)
$
as
$
\frak m_k({\bf x}) = \frak m^0_k({\bf x}) {\bf e}_c + \frak m^*_k({\bf x}).
$
Here $\frak m^0_k({\bf x}) \in \Lambda_0$ and 
$\frak m^*_k({\bf x}) \in \mathscr C[1](c,c)$.
We put
$
\frak I_k({\bf x}) = \frak m^0_k({\bf x}) {\bf f}_c
$
for $k\ge 2$ and
$
\frak I_1(a {\bf e}_c + x)
= a {\bf e}^+_c + x,
$
for $x \in \mathscr C^*(c,c)$.
It is easy to check that $\frak I$ is a unital filtered $A_{\infty}$ functor 
(with energy loss $0$.)
\end{proof}
\begin{defn}\label{defn9696666}
When $\Phi^{\#} : \mathscr C^+_1 \to \mathscr C_2$ is a unital filtered $A_{\infty}$
functor with energy loss $\rho$, we define:
$
\Phi^+ = \frak I \circ \Phi^{\#}.
$
\end{defn}
Note that 
$\Phi^+$ and $\Phi$ are different on $B\mathscr C_1$.
We remark however that 
$H(\mathscr C_1(c);\frak m_1)$ is canonically isomorphic
to $H(\mathscr C^+_1(c);\frak m_1)$ and $\Phi^+$ and 
$\Phi$ induces the same map on this cohomology groups.

\section{Homotopical unitality of the inductive limit.}
\label{sec;indhomoequi}

In this section, we prove Theorem \ref{mainalgtheorem2} (4).
Namely we prove Proposition \ref{prop101} below.
Let
$(\{\mathscr C^n\},\{\Phi^n\})$ be a 
unital inductive system of filtered $A_{\infty}$ categories.
We consider $\varinjlim\,{\mathscr C}^{n}: = A\mathscr B^{\infty}$
and its reduced version 
$(\varinjlim\,{\mathscr C}^{n})^{\rm red}: = A^{\rm red}\mathscr B^{\infty}$.

\begin{prop}\label{prop101}
The inductive limits $\varinjlim\,{\mathscr C}^{n}$
and $(\varinjlim\,{\mathscr C}^{n})^{\rm red}$ are homotopically 
unital.
\end{prop}
\begin{proof}
Let $c_n \in \frak{OB}(\mathscr C^n)$. We denote its unit by ${\bf e}_{c_n}$.
By definition of unitality $\Phi^n_1({\bf e}_{c_n}) = {\bf e}_{\Phi^n_{\rm ob}(c_n)}$.
Suppose $c_{\infty} \in \frak{OB}(\mathscr C^{\infty}) = \varinjlim\,\frak{OB}(\mathscr C^{n})$.
We take its representative $(c_n,\dots)$.
Then by the above remark $\Phi^{\infty,n}({\bf e}_{c_n})$ is independent of $n$.
We denote this element by ${\bf e}_{c_{\infty}}$.
Note that ${\bf e}_{c_{\infty}} \in \mathscr C^{\infty}(c_{\infty},c_{\infty}) 
=(\varinjlim\,{\mathscr C}^{n})(c_{\infty},c_{\infty})
\subset A\mathscr B^{\infty}(c_{\infty},c_{\infty})$.
It is contained in the finite part.
We will prove that this element is a homotopy unit.
Note that ${\bf e}_{c_{\infty}}$ is {\it not} a strict unit. 
In fact for $x \in \mathscr C^{\infty}(c_{\infty},c_{\infty})$ we have
$
\frak M_2({\bf e}_{c_{\infty}},x) = {\bf e}_{c_{\infty}} \boxtimes  x \ne x.
$
Here we use $+1$ in the sign of (\ref{frakM2}).
Let ${\bf x} \in \mathscr B^{\infty}(c_{\infty},c_{\infty}')$.
We take $c_n,c'_n \in \frak{OB}(\mathscr C^n)$ and 
${\bf x}_n \in B\mathscr C^n(c_{n},c_{n}')$
such that $\lim_{n\to\infty} c_n = c_{\infty}$, 
$\lim_{n\to\infty} c'_n = c'_{\infty}$,
$\lim_{n\to\infty} {\bf x}_n = {\bf x}_{\infty}$.
By  strict unitality of ${\Phi^n}$ and $\widehat{\Phi^n}({\bf x}_n) = {\bf x}_{n+1}$ we find that
$$
\widehat{\Phi}^n({\bf e}_{c_n} \otimes {\bf x}_n) = {\bf e}_{c_{n+1}} \otimes {\bf x}_{n+1},
\qquad
\widehat{\Phi}^n({\bf x}_n \otimes {\bf e}_{c_n}) =  {\bf x}_{n+1} \otimes {\bf e}_{c_{n+1}}.
$$
Therefore $\Phi^{\infty,n}({\bf e}_{c_n} \otimes {\bf x}_n)$
and $\Phi^{\infty,n}({\bf x}_n \otimes {\bf e}_{c_n})$
are independent of sufficiently large $n$.
We denote them by ${\bf e}_{c_{\infty}} \otimes {\bf x}_{\infty}$
and ${\bf x}_{\infty} \otimes {\bf e}_{c_{\infty}}$, respectively.
Note that $\frak v({\bf x}_{\infty})=\frak v({\bf e}_{c_{\infty}} \otimes {\bf x}_{\infty})
=\frak v({\bf x}_{\infty} \otimes {\bf e}_{c_{\infty}})$.
Now we define:
\begin{equation}\label{form1010101}
\left\{
\aligned
\frak M_2({\bf f}_{c_{\infty}},{\bf x}_{\infty})
&= -{\bf e}_{c_{\infty}} \otimes {\bf x}_{\infty},
\\
\frak M_2({\bf x}_{\infty},{\bf f}_{c_{\infty}})
&= -(-1)^{\deg'{\bf x}_{\infty}} {\bf x}_{\infty} \otimes {\bf e}_{c_{\infty}},\\
\frak M_2({\bf f}_{c_{\infty}},{\bf f}_{c_{\infty}}) &= 0,
\\
\frak M_1({\bf f}_{c_{\infty}}) &= {\bf e}^+_{c_{\infty}} - {\bf e}_{c_{\infty}}.
\endaligned
\right.
\end{equation}
In the other cases the operations $\frak M_1$, $\frak M_2$ are 
determined automatically from the requirement that ${\bf e}^+_{c_{\infty}}$ 
is a strict unit 
and the operations coincide with given one on 
$A\mathscr B^{\infty}$. 
It is easy to check that they define a structure of unital DG-category on 
$(A\mathscr B^{\infty})^+$. For example
$$
\aligned
&\frak M_1(\frak M_2({\bf f},x))
= - \frak M_1({\bf e} \otimes x)
= -x + {\bf e} \boxtimes x + {\bf e} \otimes \frak m_1(x),
\\
&\frak M_2(\widehat{\frak M}_1({\bf f},x))
= \frak M_2({\bf e}^+ - {\bf e},x) + \frak M_2({\bf f},\frak m_1(x))
= x - {\bf e} \boxtimes x - {\bf e} \otimes \frak m_1(x).
\endaligned
$$
The proof of the case of $A^{\rm red}\mathscr B^{\infty}$
is the same.
\end{proof}
\begin{defn}
For a unital filtered $A_{\infty}$ category $\mathscr C$
we define homotopy unit of $AB\mathscr C$ 
and of $A^{\rm red}B\mathscr C$ by the same formula  (\ref{form1010101}).
\par
We then obtain a unital filtered $A_{\infty}$ functor with energy loss
$
A\Phi^{\infty,n}: (AB\mathscr C^n)^+ \to (A\mathscr B^{\infty})^+
$,
$
A^{\rm red}\Phi^{\infty,n}: (A^{\rm red}B\mathscr C^n)^+ \to (A^{\rm red}\mathscr B^{\infty})^+.
$
\end{defn}

The situation which actually appears in the proof of Theorem \ref{mainalgtheorem2} (4) is slightly different 
and is as follows.
\begin{defn}
Let  $(\{\mathscr C^n\},\{\Phi^n\})$ be an 
inductive system of filtered $A_{\infty}$ categories.
We call it to be  {\it $\#$-homotopically unital} if
$\mathscr C^n$ is unital and $\Phi^n$ is $\#$-homotopically unital for $n=1,2,\dots$.
\par
If $(\{\mathscr C^n\},\{\Phi^n\})$ is $\#$-homotopically unital
then by Definition \ref{defn9696666} we obtain a unital 
filtered $A_{\infty}$ functor $(\Phi^n)^+ : (\mathscr C^n)^+ \to (\mathscr C^{n+1})^+$.
Then by Proposition \ref{prop101}, we obtain a unital completed
DG-category
$(\varinjlim\,(\mathscr C^n)^{\rm red})^+$.
\par
We call $(A\varinjlim\,(\mathscr C^n)^{\rm red})^+$ the {\it inductive limit of 
the $\#$-homotopically unital inductive system} 
$(\{\mathscr C^n\},\{\Phi^n\})$.
We denote it by
$$
\varinjlim\, (\mathscr C^n)^+:= ((A\varinjlim\,(\mathscr C^n)^+)^{\rm red})^+.
$$
\end{defn}
The reason we use reduced version here is the following.
\begin{lem}\label{Anounitality}
If $\mathscr C$ is a gapped unital filtered $A_{\infty}$
category, the  homotopy equivalence 
$\frak I: \mathscr C \to A^{\rm red}B\mathscr C$ is homotopically unital.
\end{lem}
\begin{proof}
We define 
$I_1^+({\bf e}^+_c)= {\bf e}^+_c$, $I_1^+({\bf f}^+_c)= {\bf f}^+_c$
and put all the other operations including ${\bf e}^+_c$,${\bf f}^+_c$   to be zero.
We can then check the fact that it defines a unital 
filtered $A_{\infty}$ functor (with energy loss $0$) easily.
For example
$$
\aligned
(I_*^+\circ \hat d)({\bf f}_c,x_1,\dots,x_k)
&= -{\bf e}_c \otimes x_1 \otimes \dots \otimes x_k,
\\
(\frak M_*\circ \widehat I^+)({\bf f}_c,x_1,\dots,x_k)
&= \frak M_2({\bf f}_c,x_1 \otimes \dots \otimes x_k)
\\
&= - {\bf e}_c \otimes x_1 \otimes \dots \otimes x_k.
\endaligned
$$
\end{proof}
Since the obvious quotient map  $(AB\mathscr C)^+ \to (A^{\rm red}B\mathscr C)^+$
which is unital induces an isomorphism on $\overline{\frak m}_1$ cohomology
it has a unital homotopy inverse ($A_{\infty}$ functor).
So $\mathscr C^+$ is unitally homotopy equivalent to $AB\mathscr C$.
However the author does not know 
an explicit formula of the homotopy inverse to 
the functor $(AB\mathscr C)^+ \to (A^{\rm red}B\mathscr C)^+$.
The author does not know whether  
$\frak I: \mathscr C \to AB\mathscr C$ is homotopically unital or not.
If we define $\frak I^+$ by the same formula as above then
$
(I_*^+\circ \hat d)(x,{\bf f}_c,y)
= (-1)^{\deg x} x \otimes {\bf e}_c \otimes y,
$
which is different from $(\frak M_*\circ \widehat I^+)(x,{\bf f}_c,y) = 0$.
\begin{lem}\label{lem1130}
Let $(\{\mathscr C^n\},\{\Phi^n\})$
be a unital inductive system of filtered $A_{\infty}$ categories.
Let $a_{\infty},b_{\infty} \in \frak{OB}(\varinjlim\, \mathscr C^n)$.
Take $a_n,b_n \in \frak{OB}(\mathscr C^n)$ such that $\lim_{n\to\infty} a_n = a_{\infty}$,
$\lim_{n\to\infty} b_n = b_{\infty}$. Then
\begin{equation}\label{ineq10222}
d_{\rm Hof}(a_{\infty},b_{\infty}) 
\le \liminf_{n\to \infty}d_{\rm Hof}(a_{n},b_{n}).
\end{equation}
The same holds for a $\#$-homotopically unital inductive system.
\end{lem}
\begin{proof}
Let $\rho = \lim_{n\to \infty}d_{\rm Hof}(a_{n},b_{n})$
and $\epsilon$ be a positive number. 
For a sufficiently large $n$, we have $d_{\rm Hof}(a_{n},b_{n})< \rho +\epsilon$.
We take a homotopy equivalence
$t_{a,b;n} \in \mathscr C^n_{\Lambda}(a^n,b^n)$,
$t_{b,a;n} \in \mathscr C_{\Lambda}(b^n,a^n)$,
$s_{a;n} \in \mathscr C_{\Lambda}(a^n,a^n)$,
$s_{b;n} \in \mathscr C_{\Lambda}(b^n,b^n)$
that has energy loss 
$\rho + \epsilon$.
\par
We replace 
$s_{a,n}$ by  $s_{a,n} + {\bf f}_a$
and $s_{b,n} + {\bf f}_b$.
Then they define homlotopy equivalence
 in the unitalization $(\mathscr C^n)^+$.
By Lemma \ref{Anounitality}
it induces 4-tuple of
elements 
$
\hat t_{a,b;n} \in A^{\rm red}B\mathscr C^n_{\Lambda}(a^n,b^n)$,
$
\hat t_{b,a;n} \in A^{\rm red}B\mathscr C^n_{\Lambda}(b^n,a^n)$,
$
\hat s_{a;n} \in A^{\rm red}B\mathscr C^n_{\Lambda}(a^n,a^n)$,
$
\hat s_{b;n} \in A^{\rm red}B\mathscr C^n_{\Lambda}(b^n,b^n)$
which is a homotopy equivalence
in $(A^{\rm red}B\mathscr C^n)^+$.
We put
$$
\left\{
\aligned
&t_{a,b;\infty} = A^{\rm red}\Phi_1^{\infty,n}(\hat t_{a,b;n}), 
\qquad t_{b,a;\infty} = A^{\rm red}\Phi_1^{\infty,n}(\hat t_{b,a;n}),
\\
&s_{a:\infty} = A^{\rm red}\Phi_1^{\infty,n}(\hat s_{a:n})
+ A^{\rm red}\Phi_2^{\infty,n}(\hat t_{a,b;n},\hat t_{b,a;n}), \\
& s_{b;\infty} = A^{\rm red}\Phi_1^{\infty,n}(\hat s_{b;n})
+ A^{\rm red}\Phi_2^{\infty,n}(\hat t_{b,a:n},\hat t_{a,b;n}).
\endaligned
\right.
$$
They 
become homotopy equivalence
in $(A^{\rm red}\mathscr B^{\infty})^+$.
\par
Since the energy loss of $\widehat{\Phi}^{\infty,n}$ is arbitrary small
on $\mathfrak S^2A^{\rm red}/B\mathscr C^n$, 
by taking $n$ large, $d_{\rm Hof}(a_{\infty},b_{\infty}) < \rho +2\epsilon$.
Since $\epsilon$ is arbitrary small we have proved (\ref{ineq10222}).
The proof of the other case is similar and so is omitted.
\end{proof}
We will use the next lemma to show the uniqueness of the limit up to 
weak equivalence.
\begin{lem}\label{1044lm}
Let $(\{\mathscr C^n_{(i)}\},\{\Phi^n_{(i)}\})$
be a $\#$-homotopically unital inductive system of filtered $A_{\infty}$ categories, for $i = 0,1,2$.
\par
Suppose there exist fully faithful embeddings of unital filtered $A_{\infty}$ categories
$
\Psi_{i}^n : \mathscr C_{(i)}^n \to \mathscr C_{(0)}^n
$
(with energy loss $0$)
such that:
\begin{equation}\label{equality103}
\Phi_{(0)}^n\circ  \Psi_{i}^n = \Psi_{i}^{n+1} \circ \Phi_{(i)}^n
\end{equation}
for $i=1,2$.
Then we have: 
\begin{equation}\label{form1010103333344}
d_{\rm GH}(\varinjlim\, (\mathscr C_{(1)}^n)^+, 
\varinjlim\, (\mathscr C_{(2)}^n)^+)
\le
\liminf_{n\to \infty}d_{\rm H}(\frak{OB}(\mathscr C_{(1)}^n),\frak{OB}(\mathscr C_{(2)}^n)).
\end{equation}
Here $d_H$ in the right hand side is the Hausdorff distance of subsets 
in $\frak{OB}(\mathscr C_{(0)}^n)$ with respect to the Hofer distance.
If the right hand side of (\ref{form1010103333344}) 
is zero then $\varinjlim\, (\mathscr C_{(1)}^n)^+$ 
is weakly equivalent to $\varinjlim\, (\mathscr C_{(2)}^n)^+$.
\end{lem}
\begin{proof}
We consider three inductive limits $\varinjlim\, (\mathscr C_{(1)}^n)^+$, 
$\varinjlim\, (\mathscr C_{(2)}^n)^+$, $\varinjlim\, (\mathscr C_{(0)}^n)^+$.
Since (\ref{equality103}) is the exact equality it is obvious from the construction 
that the sequence $\Psi_{i}^n$, $n=1,2,3,\dots$ induces a unital filtered $A_{\infty}$
functors
\begin{equation}
\Psi_{i}^{\infty} : \varinjlim\, (\mathscr C_{(i)}^n)^+ \to \varinjlim\, (\mathscr C_{(0)}^n)^+,
\end{equation}
for $i=1,2$.
\begin{lem}\label{lem12555}
$\Psi_{i}^{\infty}$ is an almost homotopy equivalent injection. 
The same holds for reduced version.
\end{lem}
We prove Lemma \ref{lem12555} later.
Using Lemma \ref{lem12555}  we obtain:
$$
d_{\rm GH}(\varinjlim\, (\mathscr C_{(1)}^n)^+, 
\varinjlim\, (\mathscr C_{(2)}^n)^+)
\le
d_{\rm H}(\frak{OB}(\varinjlim\, \mathscr C_{(1)}^n)^+), \frak{OB}(\varinjlim\, \mathscr C_{(2)}^n)^+))
$$
by the definition of Gromov-Hausdorff distance.
Lemma \ref{lem1130} implies that
$$
d_{\rm H}(\frak{OB}(\varinjlim\, (\mathscr C_{(1)}^n)^+), \frak{OB}(\varinjlim\, \mathscr C_{(2)}^n)^+))
\le 
\lim_{n\to \infty}d_{\rm H}(\frak{OB}(\mathscr C_{(1)}^n),\frak{OB}(\mathscr C_{(2)}^n)).
$$
Thus to prove (\ref{form1010103333344}) it remains to prove Lemma \ref{lem12555}.
\end{proof}
\begin{proof}[Proof of Lemma \ref{lem12555}]
The proof is a relative version of 
the proof of Theorem \ref{theorem63}.
We first consider 
$I_n: AB\mathscr C^n_{(1)} \to AB\mathscr C^n_{(0)}$.
Let $a,b \in \frak{OB}(\mathscr C^n_{(1)})$.
An element of $AB\mathscr C^n_{(0)}(a,b)$ 
is a linear combination of elements of the form
\begin{equation}\label{defnofcirc2}
{\bf x} = x_1 \odot_1 x_2 \odot_2 \dots \odot_{m-1} x_{m-1} \odot x_m.
\end{equation}
Here the notations are the same as in (\ref{defnofcirc}).
$\odot_i$ is either $\otimes$ or $\boxtimes$.
We say that $\odot_i$ is an {\it inner tensor}
if $x_i,x_{i+1}$ are both in $\mathscr C^n_{(1)}$.
Otherwise it is an {\it outer tensor}.
We define relative total number filtration $\frak S_{\rm out}^k
(AB\mathscr C^n_{(0)})$
such that 
${\bf x} \in  \frak S_{\rm out}^k
(AB\mathscr C^n_{(0)})$
if and only if the number of outer tensors appearing in 
(\ref{defnofcirc2}) in not greater than $k$.
We remark that
$
\frak S^{\rm out}_0
(AB\mathscr C^n_{(0)})
=
AB\mathscr C^n_{(1)}.
$
\par
We write $\otimes^{\rm out}$ or $\boxtimes^{\rm out}$
in place of $\otimes$ or $\boxtimes$
if they are outer tensors.
We then define
the operator $S^{\rm out}$ as the sum of the following operations
\begin{equation}\label{defnSout}
\aligned
&x_1 \odot_1 \dots \odot_{k-1} x_k \boxtimes^{\rm out} x_{k+1} \odot_{k+2} \dots  \odot_m x_m \\
&\mapsto 
(-1)^*  x_1 \odot_1 \dots \odot_{k-1} x_k \otimes^{\rm out} x_{k+1} \odot_{k+2} \dots \odot_m x_m.
\endaligned
\end{equation}
Namely we replace one of $\boxtimes^{\rm out}$ to $\otimes^{\rm out}$.
The sign is $* = \deg'x_1 + \deg \odot_1 + \dots + \deg\odot_{i}$.
The following formula can be proved in the 
same way as (\ref{new7576}), (\ref{new7676})
\begin{equation}\label{new75762}
\left\{
\aligned
&(\hat\Delta \circ S^{\rm out} + S^{\rm out} \circ \hat\Delta)({\bf x}) = k, \\
&\partial_1 \circ S^{\rm out} + S^{\rm out} \circ \partial_1 = 0.
\endaligned
\right.
\end{equation}
Here $k$ is the number of outer tensors 
appearing in (\ref{defnofcirc2}). Moreover we have
\begin{equation}\label{new76762}
(\partial_k\circ S^{\rm out} + S^{\rm out}\circ \partial_k)(\frak S^{\rm out}_k AB\mathscr C)
\subseteq \frak S^{\rm out}_{k-1} AB\mathscr C
\end{equation}
for $k\ge 2$.
\par
We next observe that the relative total number filtration
is preserved by $A\widehat{\Phi}^{\infty,n}: AB\mathscr C^n_{(0)} \to AB\mathscr C^{n+1}_{(0)}$.
Therefore it induces a relative total number filtration
on $A\mathscr B^{\infty}_{(0)}[-1]$, which we denote by
$\frak S^{\rm out}_k
(A\mathscr B^{\infty}_{(0)}[-1])(a,b)$.
It is easy to see
$$
\frak S^{\rm out}_0
(A\mathscr B^{\infty}_{(0)}[-1])(a,b)
=
(A\mathscr B^{\infty}_{(1)}[-1])(a,b).
$$
Therefore by Lemma \ref{lemma8dayto}
it suffices to show that 
$$
\frak S^{\rm out}_{k-1}
(A\mathscr B^{\infty}_{(0)}[-1])(a,b)
\to \frak S^{\rm out}_k
(A\mathscr B^{\infty}_{(0)}[-1])(a,b)
$$
is an almost cochain homotopy equivalence relative 
to $AB\mathscr C^n_{(1)}$.
Using (\ref{new75762}), (\ref{new76762}), we can prove it 
in the same way as Lemma \ref{lem8888}.
The proof of Lemma \ref{lem12555} is complete.
\end{proof}

\begin{rem}
It seems difficult to obtain the same conclusion as 
Lemma \ref{lem113} when we require only 
homotopy commutativity of (\ref{equality103}).
We will discuss this point further in Remark \ref{infhikkuri}.
\end{rem}

\begin{proof}[Proof of Theorem \ref{mainalgtheorem2}]
Let $(\{\mathscr C^n\},\{\Phi^n\})$ be an 
inductive system of filtered $A_{\infty}$ categories.
We defined, in Definition \ref{defnindlim}, the inductive limit
$\varinjlim\, \mathscr C^n$.
Theorem \ref{mainalgtheorem2} (1) is immediate from the construction.
Theorem \ref{mainalgtheorem2} (3) is a consequence of Theorem \ref{theorem63}.
The filtered $A_{\infty}$ functor 
$\Upsilon^{\infty,n}: (\mathscr C^n)^+ \to (\varinjlim\, \mathscr C^n)^+$ 
is obtained 
as the composition of
$$
\frak I^+ : (\mathscr C^n)^+ \to (A^{\rm red}B\mathscr C^n)^+,
\qquad
A^{\rm red}\Phi^{n,\infty}: (A^{\rm red}B\mathscr C^n)^+ \to 
(A^{\rm red}\mathscr B^{\infty})^+
$$
(Note that $(\varinjlim\, \mathscr C^n)^+ = (A^{\rm red}\mathscr B^{\infty})^+$
by definition.)
Lemma \ref{Anounitality} implies that $\frak I^+$ is unital.
The unitality of the $A^{\rm red}\Phi^{n,\infty}$ is obvious.
Theorem \ref{mainalgtheorem2} (2) is proved.
\end{proof}
\par\medskip
Gromov \cite{greenbook} proved that 
the metric space 
(with respect to the Gromov-Hausdorff distance) 
of all isometry classes 
of compact metric spaces\footnote{
More precisely we take a Universe and consider 
only a metric space whose underlying 
set is an element of this Universe.} is complete.
See, for example, \cite[Theorem 1.5]{fhousdorff} for its proof.
\par
Theorem \ref{mainalgtheorem2} is its filtered $A_{\infty}$
category analogue in a certain sense. 
However it does not contain a statement that 
$\mathscr C^n$ converges to the inductive limit 
with respect to the Gromov-Hausdorff distance.
Theorem \ref{thm131111} is a certain 
completeness result.

\section{Hofer infinite distance and homotopical unitality}
\label{Hofinfhomoto}

To apply the results of Section \ref{sec;indhomoequi} 
 to the 
inductive system obtained from a sequence of 
subsets of the given filtered $A_{\infty}$ category by 
Theorem \ref{GrokaraInd},
a difficulty lies in the fact that the filtered $A_{\infty}$ functor with energy loss 
obtained by Theorem \ref{GrokaraInd} is not unital. It seems that it is not homotopically unital 
either, in general.
\begin{rem}
In several references such as \cite{Lef,Se}, the unit in the cohomology level, 
that is, an element of $H(\mathscr C(c,c))$ which behaves as the identity map 
in the cohomology groups, is used instead of a homotopy unit or a strict unit.
It is known (see \cite{Lef,Se}) that the existence of such a homology unit is 
equivalent to homotopical unitality, in the case of $A_{\infty}$ category 
over a field. A method to prove it is to reduce the structure to the 
$\frak m_1$ cohomology group. In the case of a filtered $A_{\infty}$ category we cannot reduce the structure to 
the $\frak m_1$ cohomology group, since it has a torsion.
\par
Note that the functor in Theorem \ref{GrokaraInd} preserves the unit 
on homology level. However it does not seem to imply 
homotopical unitality. 
\end{rem}
We strengthen Hofer distance to Hofer infinite distance to overcome this difficulty.
Let $\mathscr C$ be a unital filtered $A_{\infty}$ category.
For $x \in \mathscr C(c_1,c_2)$ we write 
${\rm soc}(x) = c_1$, ${\rm tar}(x) = c_2$.
\begin{defn}\label{infhomoequiv}
Let $\mathscr C$ be a filtered $A_{\infty}$ category and
$c_1,c_2 \in \frak{OB}(\mathscr C)$. 
An {\it infinite homotopy equivalence} between $c_1$ and $c_2$ 
is a 4-tuple of sequences $(\{t^k_{12}\},\{t^k_{21}\},\{s^k_{1}\},\{s^k_{2}\})$,
$k=1,2,\dots$
with the following properties.
\begin{enumerate}
\item $t^k_{12} \in \mathscr C_{\Lambda}(c_1,c_2)$, $t^k_{21} \in \mathscr C_{\Lambda}(c_2,c_1)$, 
$s^k_{1} \in \mathscr C_{\Lambda}(c_1,c_1)$, $s^k_{2} \in \mathscr C_{\Lambda}(c_2,c_2)$.
\item $\deg' t^k_{12}, \deg' t^k_{12} = 1-2k$, $\deg' s^k_{1}, \deg' s^k_{2} = -2k$.
\item
They satisfy   Maurer-Cartan  equation (\ref{MC11}).
We consider a sequence $\vec u =(u(1),\dots,u({\ell}))$ which 
enjoyes the following properties:
\par\smallskip
\begin{enumerate}
\item The symbol $u(i)$ is one of the elements $t^k_{12}$,  $t^k_{21}$, $s^k_{1}$, $s^k_{2}$ where $k=1,2,\dots$.
\item We require ${\rm tar}(u(i))= {\rm soc}(u(i+1))$.
\end{enumerate}
\par\smallskip
We put $\deg' \vec u = \sum \deg' u(i) = \sum (\deg u(i)-1)$ and ${\rm soc}(\vec u) = {\rm soc}(u(1))$, 
${\rm tar}(\vec u) = {\rm soc}(u(\ell))$.
We also put
$$
\frak m(\vec u) = \frak m_{\ell}(u(1),\dots,u({\ell})).
$$
For $k =1,2,\dots$ let $\frak u(1,2;k)$ be the set of all such $\vec u$ 
with ${\rm soc}(\vec u) = 1$, ${\rm tar}(\vec u) = 2$ and $\deg' \vec u = 1-2k$. We define 
$\frak u(2,1;k)$ in the same way.
For $k =1,2,\dots$ let $\frak u(1;k)$ be the set of all such $\vec u$ 
with ${\rm soc}(\vec u) = 1$, ${\rm tar}(\vec u) = 1$ and $\deg' \vec u = -2k$. 
We define $\frak u(2;k)$ in the same way.\footnote{Even in the case 
when our filtered $A_{\infty}$ category is not $\Z$ graded, we define the 
set $\frak u(1,2;k)$ etc. in exactly the same way as the $\Z$ graded case.}
\par
Now we require the following:
\begin{equation}\label{MC11}
\left\{
\aligned
&\sum_{\vec u \in \frak u(1,2;k)} \frak m(\vec u) = 0, 
\quad
\sum_{\vec u \in \frak u(2,1;k)} \frak m(\vec u) = 0, \\
&\sum_{\vec u \in \frak u(1;1)} \frak m(\vec u) = {\bf e}_{c_1}, 
\quad
\sum_{\vec u \in \frak u(2;1)} \frak m(\vec u) = {\bf e}_{c_2},
\\
&\sum_{\vec u \in \frak u(1;k)} \frak m(\vec u) = 0, 
\quad
\sum_{\vec u \in \frak u(2;k)} \frak m(\vec u) = 0, \qquad \text{$k\ge 2$}.
\endaligned
\right.
\end{equation}
\end{enumerate}
\end{defn}
\begin{defn}
We define the {\it Hofer infinite distande} $d_{\rm Hof,\infty}^{\rm bi}(c_1,c_2)$
as follows.
The inequality $d_{\rm Hof,\infty}(c_1,c_2) < \epsilon$
holds
if and only if there exists an infinite homotopy equivalence 
$(\{t^k_{12}\},
\{t^k_{21}\},\{s^k_{1}\},\{s^k_{2}\})$ such that
\begin{equation}\label{equation1122222}
\aligned
\frak v(s_{12}^k) \ge -\epsilon_1 - (k-1)\epsilon,
\qquad
&\frak v(s_{21}^k) \ge -\epsilon_2 - (k-1)\epsilon,
\\
\frak v(s_{1}^k) \ge -k\epsilon,
\qquad
&\frak v(s_{2}^k) \ge -k\epsilon,
\endaligned
\end{equation}
and $\epsilon = \epsilon_1+\epsilon_2$.
We define the {\it energy loss of an infinite homotopy equivalence} $(\{t^k_{12}\},
\{t^k_{21}\},\{s^k_{1}\},\{s^k_{2}\})$
is the pair of positive numbers $(\epsilon_1,\epsilon_2)$ such that (\ref{equation1122222})
is satisfied.
\par
In the same way as Remark \ref{Rem23232}, we can modify the infinite homotopy 
as $t^k_{12} \mapsto T^{\delta}t^k_{12}$, 
$t^k_{21} \mapsto T^{-\delta}t^k_{21}$ so that 
$\epsilon_1 = \epsilon_2 = \epsilon/2$. In such a case we say 
energy loss is $\epsilon$.
\end{defn}

Let us demonstrate a few  cases of the equation (\ref{MC11}).
The set $\frak u(1,2;1)$ consists of a single element $t^{1}_{12}$.
Therefore the first equation is
\begin{equation}\label{form112}
\frak m_1(t^{1}_{12}) = 0.
\end{equation}
The set $\frak u(1;1)$ consists of 2 elements $(t^{1}_{12},t^{1}_{21})$
and $s^1_{1}$. Therefore the equation is
\begin{equation}\label{form113}
\frak m_2(t^{1}_{12},t^{1}_{21}) + \frak m_1(s^{1}_{1}) = {\bf e}_{c_{1}}.
\end{equation}
Formulas (\ref{form112}) and (\ref{form113}) coincide with 
Definition \ref{defhoferdist} (1).
We remark also that the inequality (\ref{equation1122222}) coincides with 
Definition \ref{defhoferdist} (2).
\par
The set $\frak u(1,2;2)$ consists of four elements,
$(t^{1}_{12},s^{1}_{2})$, $(s^{1}_{1},t^{1}_{12})$, $(t^{2}_{12})$,
and $(t^{1}_{12},t^{1}_{21},t^{1}_{12})$.
Therefore the equation is:
$$
\frak m_2(t^{1}_{12},s^{1}_{2}) + \frak m_2(s^{1}_{1},t^{1}_{12}) +
\frak m_1(t^{2}_{12}) + \frak m_3(t^{1}_{12},t^{1}_{21},t^{1}_{12}) = 0.
$$
If $\mathscr C$ is a DG-category then the fourth term drops.
It then means that $t^1$ and $s^1$  `homotopicaly commutes'.
(Here we put homotopically commutes in the quote since $s^1$ is not a cocycle.)
\begin{rem}
Let us consider the case of a DG-category.
The cohomology class of the cocycle 
$\frak m_2(t^{1}_{12},s^{1}_{2}) + \frak m_2(s^{1}_{1},t^{1}_{12})$ 
could be an obstruction to the existence of infinite 
homotopy equivalence.
Note that we can change $s^{1}_{2}$ and $s^{1}_{1}$ 
by adding cocycles so that this cohomology class vanishes.
However then the cohomology class of 
$\frak m_2(t^{1}_{21},s^{1}_{1}) + \frak m_2(s^{1}_{2},t^{1}_{21})$ 
changes also. Therefore it seems impossible to eliminate 
this obstruction by changing $s^{1}_{2}$ and $s^{1}_{1}$ by cocycles.
\end{rem}
\begin{lem}\label{lem113}
If $\Phi : \mathscr C_1 \to \mathscr C_2$ is a unital filtered $A_{\infty}$ functor 
with energy loss $\rho$, then 
$$
d_{{\rm Hof},\infty}(\Phi_{\rm ob}(c_1),\Phi_{\rm ob}(c_2)) 
\le d_{{\rm Hof},\infty}(c_1,c_2) + 2\rho.
$$
\end{lem}
\begin{proof}
Let $c'_i = \Phi_{\rm ob}(c_i)$ for $i=1,2$ 
and $(\{t^k_{12}\},
\{t^k_{21}\},\{s^k_{1}\},\{s^k_{2}\})$ the infinite homotopy 
equivalence between them with energy loss $
\epsilon$. Using the notation in Definition \ref{infhomoequiv} (3)
we put
$$
(t^{k}_{12})' = \sum_{\vec u \in \frak u(12;k)} \Phi_*(\vec u)
\in \mathscr C_{\Lambda}(c'_1,c'_2), \qquad
(s^{k}_{1})' = \sum_{\vec u \in \frak u(1;k)} \Phi_*(\vec u)
\in \mathscr C_{\Lambda}(c'_1,c'_1).
$$
We define $(t^{k}_{21})'$ and $(s^{k}_{2})'$ in the same way.
It is easy to check that they become an infinite homotopy equivalence between $c'_1$ and 
$c'_2$
with  energy loss $\epsilon+2\rho$.
\end{proof}
Lemma \ref{lem113} immediately implies the following:
\begin{cor}\label{invhoferinfdist}
If $\Phi$ is a unital homotopy equivalence between two unital  filtered $A_{\infty}$ categories 
$\mathscr C_1$, $\mathscr C_2$, then
$\Phi_{\rm ob}$ preserves Hofer infinite distance.
\end{cor}
\begin{thm}\label{cong106}
$
d_{\rm Hof,\infty}(c_1,c_3) \le
d_{\rm Hof,\infty}(c_1,c_2) 
+ d_{\rm Hof,\infty}(c_2,c_3).
$
\end{thm}
We will prove Theorem \ref{cong106} later in this section.
\par
We can prove that 
Hofer infinite metric is equal to Hofer metric
in certain cases.
\begin{prop}\label{thm129}
We assume that ${\mathscr C}$ is gapped.
If degree $-1$ cohomology $H^{-1}(\overline{\mathscr C}(c,c),\overline{\frak m}_1)$  
of the $R$ reduction of $\mathscr C(c,c)$ vanishes for an arbitrary object $c \in \frak{OB}(\mathscr C)$ then 
Hofer infinite distance coincides with  Hofer distance.
\end{prop}
\begin{proof}
We can find a unital and gapped filtered $A_{\infty}$ category 
$\mathscr C'$
which is unitally homotopy equivalent to 
$\mathscr C$ and whose morphism spaces $\mathscr C'(a,b)$ are 
\begin{equation}
\mathscr C'(a,b) \cong 
H(\overline{\mathscr C}(a,b),\overline{\frak m}_1) \otimes_R \Lambda_0,
\end{equation}
as $\Lambda_0$ modules.
In the case when we replace $\Lambda_0$ by $R$ this is a classical result 
going back to Kadai\v sibili \cite{kadeisivili}.
In the filtered case this fact is proved in \cite[Theorem 5.4.2]{fooobook} for a 
filtered $A_{\infty}$
algebra. We can prove the category case in exactly the same way.
\par
By Corollary \ref{invhoferinfdist} it suffices to prove the case of 
$\mathscr C'$.
Let $c_1,c_2 \in \frak{OB}(\mathscr C')$ and $(t_{12},t_{21},s_{1},s_{2})$ is 
an infinite homotopy 
equivalence between them.
We observe that the degree of $s_{1}$, $s_{2}$ are $-1$ and so $s_{1}$, $s_{2}$ are $0$ by 
assumption.
Therefore by putting $t^1_{12} = t_{12}$,  $t^1_{21} = t_{21}$ and 
all other $t^k_{12}$,  $t^k_{21}$, $s^k_{1}$, $s^k_{2}$ to be $0$
we obtain an infinite homotopy equivalence.
It implies that the Hofer infinite distance is not greater than 
Hofer distance. The inequality of the other direction is obvious.
\end{proof}

We define a Hausdorff distance between subsets of $(\frak{OB}(\mathscr C),d_{\rm Hof.\infty})$ using Hofer infinite distance.
We denote it by $d_{\rm H,\infty}$.

\begin{defn}\label{GHinftygapped}
Let $\mathscr C_1$, $\mathscr C_2$ be unital and gapped filtered $A_{\infty}$ categories.
We define the {\it Gromov-Hausdorff infinite distance} $d_{\rm GH,\infty}(\mathscr C_1,\mathscr C_2)$
as follows. The inequality $d_{\rm GH,\infty}(\mathscr C_1,\mathscr C_2) < \epsilon$
holds if and only if  
there exists a unital and gapped filtered $A_{\infty}$ category $\mathscr C$ such that
(1)(2) of Definition \ref{GromovHausdorff} hold and
\begin{enumerate}
\item[(3)']  
$d_{\rm H,\infty}(\frak{OB}(\mathscr C_1),\frak{OB}(\mathscr C_2) < \epsilon.$
\end{enumerate}
Two unital and gapped filtered $A_{\infty}$ categories $\mathscr C_1$, $\mathscr C_2$ 
are said to be {\it  equivalent} 
if there exists
a unital and gapped filtered $A_{\infty}$ category $\mathscr C$
such that (1)(2) of Definition \ref{GromovHausdorff} hold 
and 
\begin{enumerate}
\item[(3)'']  $d_{\rm H,\infty}(\frak{OB}(\mathscr C_1),\frak{OB}(\mathscr C_2) = 0$.
\end{enumerate}
\end{defn}
In the situation when we do not assume gapped-ness we define:
\begin{defn}\label{inftywaekGH}
Let $\mathscr C_1$, $\mathscr C_2$ be completed unital 
 DG-categories.
We define the {\it Gromov-Hausdorff infinite distance} $d_{\rm GH,\infty}(\mathscr C_1,\mathscr C_2)$
between them as follows. The inequality $d_{\rm GH}(\mathscr C_1,\mathscr C_2) < \epsilon$
holds if and only if there exists a unital completed DG-category 
$\mathscr C$ such that
(1) of Definition \ref{GromovHausdorff} 
(2)' of Definition \ref{weakGH} and (3)' of Definition \ref{GHinftygapped}
 hold.
Two unital 
and completed DG categories $\mathscr C_1$, $\mathscr C_2$
are said to be {\it  equivalent} if 
there exists $\mathscr C_0$ such that (1) of Definition \ref{GromovHausdorff} 
and (2)' above hold and the Hausdorff distance
$
d_{\rm H,\infty}(\frak{OB}(\mathscr C_1),\frak{OB}(\mathscr C_2))$
is $0$.
\end{defn}
\begin{thm}
Gromov-Hausdorff infinite distance in Definition \ref{GHinftygapped}
satisfies the triangle inequality. 
Equivalence in Definition \ref{GHinftygapped}
is an equivalence relation.
\end{thm}

Using Theorem \ref{cong106} the proof is the same as the proofs of Theorems 
\ref{triangleGH} and \ref{triangleharder}.

\begin{thm}\label{idometric}
If $\mathscr C \to \mathscr C'$ 
is an almost homotopy equivalent injection 
of completed DG-categories, then $\frak{OB}(\mathscr C) \to \frak{OB}(\mathscr C')$
is an isometric embedding with respect to Hofer infinite distance.
\end{thm}
Theorem \ref{idometric} is an enhancement of Lemma \ref{isoembHFH}.
We will prove it at the end of this section.
Using Theorem \ref{idometric}, we will prove the next result 
in Section \ref{sec;spedim}.
\begin{thm}\label{them123333}
Suppose a unital complete DG-category $\mathscr C_1$ 
is equivalent to $\mathscr C_2$. 
We assume $\frak{OB}(\mathscr C_i)$ is complete and Haudorff for $i=1,2$.
Then 
there exists an isometry $I: \frak{OB}(\mathscr C_1)
\to \frak{OB}(\mathscr C_2)$ such that for $a,b \in \frak{OB}(\mathscr C_1)$
the cohomology group $H(\mathscr C_1(a,b),\frak m_1)$
is almost isomorphic to $H(\mathscr C_2(I(a),I(b)),\frak m_1)$.
\end{thm}
We can find $\mathscr C'$ which is
equivalent to $\mathscr C$ such that 
$\frak{OB}(\mathscr C')$ is complete and Hausdorff
(Proposition \ref{completion}).

We can improve Theorem \ref{estimateHofLag} as follows.
\begin{thm}\label{estimateHofLaginf2}
In the situation of Theorem \ref{estimateHofLag} we have:
$$
d_{\rm Hof,\infty}((L,b),(\varphi(L),\varphi_*(b)) \le E({H}).
$$
\end{thm}
We will prove Theorem \ref{estimateHofLaginf2} in Section \ref{sec;hoferinf1}.
The next corollary then follows.
\begin{cor}\label{corghchikachik2a}
Let $\varphi_i: X \to X$ be Hamiltonian diffeomorphisms 
such that $d_{\rm H}(\varphi_i,{\rm id}) < \epsilon$ for $i=1,\dots,k$
and $L_1,\dots,L_k$ be relatively spin Lagrangian submanifolds.
Then
$$
d_{\rm GH,\infty}(\frak F(L_1,\dots,L_k), \frak F(\varphi_1(L_1),\dots,
\varphi_k(L_k))
< \epsilon.
$$
\end{cor}
Now we have the following improvement of Theorem \ref{GrokaraInd}:

\begin{thm}\label{GrokaraIndunit}
Suppose $\mathscr C_1$ and $\mathscr C_2$ are 
gapped and unital filtered $A_{\infty}$ categories with 
$d_{\rm GH,\infty}(\mathscr C_1,\mathscr C_2) < \epsilon$.
Then there exists a gapped and $\#$-homotopy unital filtered $A_{\infty}$
functor $\Phi: \mathscr C_1 \to \mathscr C_2$ 
with energy loss $\epsilon$.
\par
Its linear part $\Phi_1 : \mathscr C_1(a,b) \to \mathscr C_2(\Phi_{\rm ob}(a),\Phi_{\rm ob}(b))$
is a cochain homotopy equivalence with energy loss $2\epsilon$.
\par
The same holds for completed DG-category (which is not necessary gapped).
\end{thm}
\begin{proof}
In the same way as the proof of Theorem \ref{GrokaraInd} it suffices to prove the next lemma.
\begin{lem}\label{Lemma3434unit}
Let $\mathscr C$ be a unital filtered $A_{\infty}$
category and $\frak C_{\mathcal I} = \{c_i \mid i \in \mathcal I\}$, $
\frak C'_{\mathcal I} = \{c'_i \mid i \in \mathcal I\}$ 
are sets of objects. 
Here $\mathcal I$ is a certain index set.
We assume that  
$d_{\rm Hof}(c_i,c'_i) < \epsilon$ for $i \in \mathcal I$.
Then there exists a $\#$-homotopy unital filtered $A_{\infty}$ functor 
with energy loss $\epsilon$, 
$
\Phi: {\mathscr C}(\frak C_{\mathcal I})
\to 
{\mathscr C}(\frak C_{\mathcal I}')
$
between full subcategories.
\par
Its linear part is a cochain homotopy equivalence with energy loss $2\epsilon$.
\par
The same holds for completed DG-category (which is not necessary gapped).
\end{lem}
\begin{proof}
It suffices to show that the filtered $A_{\infty}$ functor $\Psi$ with 
energy loss $\epsilon$
given by (\ref{nhomotoyp}) extends to a 
unital filtered $A_{\infty}$ functor 
$\Psi^{\#}: \mathscr C(\frak C_{\mathcal K})^+
\to {\mathscr C}(\frak C_{\mathcal K}')$ with 
energy loss $\epsilon$.
We put
\begin{equation}\label{form126}
\aligned
&\Psi^{\#}_{k+\sum \ell_i}({\bf f}_{c_0}^{\otimes\ell_0},x_1,{\bf f}^{\otimes\ell_1}_{c_1},\dots,
{\bf f}_{c_{k-1}}^{\otimes\ell_{k-1}},x_k,{\bf f}_{c_{k}}^{\otimes\ell_{k}})
\\
&:= t^{\ell_0+1}_{c'_0,c_0} \circ x_1 \circ s_{c_1}^{\ell_1+1} \circ x_2 \circ \dots \circ x_{k-1}\circ  s_{c_{k-1}}^{\ell_{k-1}+1}  
\circ x_{k} \circ t_{c_k,c'_k}^{\ell_k+1}
\endaligned
\end{equation}
where $\ell_i = 0,1,\dots$
and
$
\Psi^{\#}_k({\bf f}_c^{\otimes k}) := s_c^{k},
$
for $k =1,2,\dots$.
The explaination of the notations in (\ref{form126})
is in order.
We assume $x_i \in \mathscr C(c_{i-1},c_i)$ for $i=1,\dots,k$.
We remark that 
$
\mathscr C^+(c_{i},c_i) = \mathscr C^+(c_{i},c_i) \oplus
\Lambda_0 {\bf f}_{c_i} \oplus \Lambda_0 {\bf e}^+_{c_i}.
$
The element ${\bf f}_{c_i}^{\otimes\ell_i}$
appearing in (\ref{form126}) is defined by
$
{\bf f}_{c_i}^{\otimes\ell_i}:=
\underbrace{{\bf f}_{c_i} \otimes \dots \otimes {\bf f}_{c_i}}_{\ell_i}.
$
By the assumption, there exists an
infinite homotopy equivalence $(\{t^k_{c_i,c'_i}\}, 
\{t^k_{c'_i,c_i}\}\{s^k_{c_i}\},\{s^k_{c'_i}\})$
of energy loss $\epsilon$.
They appear in (\ref{form126}).
Note that (\ref{nhomotoyp}) is the case of DG-category.
The composition $\circ$ is obtained from $\frak m_2$ by 
putting the sign as in (\ref{signchangetocirc}).
\par
In the case of a DG-category, the Maurer-Cartan equation 
(\ref{MC11}) becomes the following equations.
\begin{equation}\label{form128}
\left\{
\aligned
&d(t^k_{c,c'}) - 
\sum_{\ell=1}^{k-1} s^{\ell}_{c} \circ t^{k-\ell}_{c,c'}
+ \sum_{\ell=1}^{k-1} t^{k-\ell}_{c,c'} \circ  s^{\ell}_{c'} = 0 \\
&d(t^k_{c',c}) - 
\sum_{\ell=1}^{k-1} s^{\ell}_{c'} \circ t^{k-\ell}_{c',c}
+ \sum_{\ell=1}^{k-1}  t^{k-\ell}_{c',c} \circ s^{\ell}_{c} = 0
\endaligned
\right.
\end{equation}
\begin{equation}\label{form129}
\left\{
\aligned
&d(s^k_{c})  -
\sum_{\ell=1}^{k-1} s^{\ell}_{c} \circ s^{k-\ell}_{c}
+ \sum_{\ell=1}^{k} t^{1+k-\ell}_{c,c'} \circ t^{\ell}_{c',c} = 
\begin{cases} 0 &\text{$k\ne 1$}\\
{\bf e}_{c} &\text{$k =  1$}
\end{cases}
\\
&d(s^k_{c'}) - 
\sum_{\ell=1}^{k-1} s^{\ell}_{c'} \circ s^{k-\ell}_{c'}
+ \sum_{\ell=1}^{k} t^{1+k-\ell}_{c',c} \circ t^{\ell}_{c,c'} = 
\begin{cases} 0 &\text{$k\ne 1$}\\
{\bf e}_{c'} &\text{$k = 1$}
\end{cases}
\endaligned
\right.
\end{equation}
Using (\ref{form128}) and (\ref{form129}) it is straightforward 
to check that $\Psi^{\#}$ is a unital filtered $A_{\infty}$ functor with 
energy loss $\epsilon$ as follows.
Let $\pi_1 : B\mathscr C[1](c_0,c_k) \to \mathscr C[1](c_0,c_k)$
be the projection to the direct summand.
We calculate, for $k>0$:
\begin{equation}\label{hajimenohhoo6eisi}
\aligned
&(\pi_1 \circ \hat d\circ \widehat{\Phi^{\#}})({\bf f}_{c'_0}^{\otimes\ell_0},x_1,{\bf f}^{\otimes\ell_1}_{c'_1},\dots,
{\bf f}_{c'_{k-1}}^{\otimes\ell_{k-1}},x_k,{\bf f}_{c'_{k}}^{\otimes\ell_{k}}) 
\\
&= (\pi_1 \circ \hat d_1)(t^{\ell_0+1}_{c'_0,c_0} \circ x_1 \circ s_{c_1}^{\ell_1+1} \circ \dots \circ   s_{c_{k-1}}^{\ell_{k-1}+1}  
\circ x_{k} \circ t_{c_k,c'_k}^{\ell_k+1}) \\
&\qquad+
\sum_{i=1}^{k-1}\sum_{j=1}^{\ell_i} (-1)^{\deg x_1 + \dots + \deg'x_i}
(t^{\ell_0+1}_{c'_0,c_0} \circ x_1 \circ s_{c_1}^{\ell_1+1} \circ \dots 
\\
&\qquad\qquad\qquad\qquad\qquad\qquad\circ  s_{c_{i-1}}^{\ell_{j-1}+1}  
\circ x_{i} \circ t_{c_i,c'_i}^{j+1})
\\
&\qquad\qquad\circ
(t^{\ell_i-j+1}_{c'_i,c_i} \circ x_{i+1} \circ \dots \circ  s_{c_{k-1}}^{\ell_{k-1}+1}  
\circ x_{k} \circ t_{c_k,c'_k}^{\ell_k+1})\\
&\qquad- \sum_{j=1}^{\ell_0}
s^{j}_{c'_0,c'_0} \circ
(t^{\ell_0-j+1}_{c'_0,c_0} \circ x_1 \circ s_{c_1}^{\ell_1+1} \circ \dots \circ  s_{c_{k-1}}^{\ell_{k-1}+1}  
\circ x_{k} \circ t_{c_k,c'_k}^{\ell_k+1})\\
&\qquad+ \sum_{j=1}^{\ell_k}(-1)^{\deg x_1 + \dots + \deg'x_k+1}
(t^{\ell_0+1}_{c'_0,c_0} \circ x_1 
\\
&\qquad\qquad\qquad\qquad\circ s_{c_1}^{\ell_1+1} \dots \circ  s_{c_{k-1}}^{\ell_{k-1}+1}  
\circ x_{k} \circ t_{c_k,c'_k}^{j+1})\circ s_{c'_k,c'_k}^{\ell_k - j+1},
\endaligned
\end{equation}
Here $\hat d_1$ is co-derivation on the Bar-complex which is induced from the coboundary operator $d$ 
of $\mathscr C(c_i,c_{i+1})$.
We next calculate:
\begin{equation}\label{owarinohoutei}
\aligned
&(\pi_1 \circ \widehat{\Phi^{\#}} \circ \hat d )({\bf f}_{c'_0}^{\otimes\ell_0},x_1,{\bf f}^{\otimes\ell_1}_{c'_1},\dots,
{\bf f}_{c'_{k-1}}^{\otimes\ell_{k-1}},x_k,{\bf f}_{c'_{k}}^{\otimes\ell_{k}}) \\
&= \sum_{i=1}^k (-1)^{\deg x_1 + \dots + \deg'x_{i-1}+1}
t^{\ell_0+1}_{c'_0,c_0} \circ x_1 \circ s_{c_1}^{\ell_1+1} \circ  \dots 
\\
&\qquad\qquad\qquad\qquad\qquad\qquad\circ dx_i \circ \dots \circ  s_{c_{k-1}}^{\ell_{k-1}+1}  
\circ x_{k} \circ t_{c_k,c'_k}^{\ell_k+1}\\
&\quad+
\sum_{i=1}^{k-1}\sum_{j=1}^{\ell_i}(-1)^{\deg x_1 + \dots + \deg'x_i}
t^{\ell_0+1}_{c'_0,c_0} \circ x_1  \circ \dots 
\\
&\qquad\qquad\qquad\qquad\qquad\qquad \circ x_{i} \circ s^{j}_{c_i} \circ {\bf e}_{c_i} \circ {s}^{\ell_i-j-1}_{c_i} \circ  \dots  
\circ x_{k} \circ t_{c_k,c'_k}^{\ell_k+1}\\
&\quad-\sum_{j=1}^{\ell_0}t^{j+1}_{c'_0,c'_0} \circ {\bf e}_{c_0} \circ s^{\ell_0 - j +1}_{c'_0,c_0}\circ
x_1 \circ s_{c_1}^{\ell_1+1} \circ  \dots 
 \circ  s_{c_{k-1}}^{\ell_{k-1}+1}  
\circ x_{k} \circ t_{c_k,c'_k}^{\ell_k+1}
\\
&\quad+\sum_{j=1}^{\ell_k}(-1)^{\deg x_1 + \dots + \deg'x_k}
t^{\ell_0 +1}_{c'_0,c_0}\circ
x_1 \circ s_{c_1}^{\ell_1+1} \circ  \dots 
\\
&\qquad\qquad\qquad\qquad\qquad\qquad\circ  s_{c_{k-1}}^{\ell_{k-1}+1}  
\circ x_{k} \circ s_{c_k,c'_k}^{j+1} \circ {\bf e}_{c'_k} \circ t_{c'_k,c'_k}^{\ell_k-j+1}
\\
&\quad+\sum_{i; \ell_i = 0}(-1)^{\deg x_1 + \dots + \deg'x_{i-1}}
t^{\ell_0+1}_{c'_0,c_0} \circ x_1 \circ s_{c_1}^{\ell_1+1} \circ  \dots 
\\
&\qquad\qquad\qquad\qquad\qquad\qquad \circ x_i \circ x_{i+1} \circ \dots \circ  s_{c_{k-1}}^{\ell_{k-1}+1}  
\circ x_{k} \circ t_{c_k,c'_k}^{\ell_k+1}
\endaligned
\end{equation}
(\ref{form128}) and (\ref{form129}) implies (\ref{hajimenohhoo6eisi})
= (\ref{owarinohoutei}).
The case when $k=0$ is similar.
The proof of Lemma \ref{Lemma3434unit} is complete.
\end{proof}
The proof of Theorem \ref{GrokaraIndunit} is now complete.
\end{proof}
\begin{proof}[Proof of Theorem \ref{cong106}]
By Corollary \ref{invhoferinfdist} it suffices to consider 
the case of DG-category.
Let $(\{s^{k,12}_{c_1}\},\{s^{k,12}_{c_2}\},\{t^k_{c_1,c_2}\},\{t^k_{c_2,c_1}\})$
be an infinite homotopy equivalence between $c_1$ and $c_2$ 
and 
$(\{s^{k,23}_{c_2}\},\{s^{k,23}_{c_3}\},\{t^k_{c_2,c_3}\},\{t^k_{c_3,c_2}\})$
an infinite homotopy equivalence between $c_2$ and $c_3$.
\par
We will define an 
infinite homotopy equivalence
$(\{s^{k,13}_{c_1}\},\{s^{k,13}_{c_3}\},\{t^k_{c_1,c_3}\},\{t^k_{c_3,c_1}\})$
between $c_1$ and $c_3$.
We define
\begin{equation}
s^{k,13}_{c_1}
= s^{k,12}_{c_1}
+ \sum_{\ell}\sum_{m_0,\dots,m_{\ell+1}} 
t^{m_0}_{c_1,c_2} \circ s^{m_1,21}_{c_2}
\circ s^{m_2,12}_{c_2}
\circ \dots \circ s^{m_\ell,21}_{c_2}
\circ t^{m_{\ell+1}}_{c_2,c_1}.
\end{equation}
Here the sum in the second term is taken over all $\ell = 1,3,5,\dots$
and $m_i$ such that 
$$
-1 - 2k = -2m_0 + \sum_{i=1}^{\ell} (-1-2m_i) - 2m_{\ell+1}.
$$  
In other words, we take the sum for all the choices which give elements 
of correct degree.
The definition of $s^{k,13}_{c_3}$ is similar.
\par
We define
\begin{equation}
t^{k}_{c_1,c_3}
= 
\sum_{\ell}\sum_{m_0,\dots,m_{\ell+1}}
t^{m_0}_{c_1,c_2}  
\circ s^{m_1,21}_{c_2}
\circ \dots \circ s^{m_\ell,12}_{c_2}
\circ t^{m_{\ell+1}}_{c_2,c_3}.
\end{equation}
Here the sum is taken over all $\ell = 0,2,4,\dots$
and $m_i$ such that 
$$
- 2k = -2m_0 + \sum_{i=1}^{\ell} (-1-2m_i) - 2m_{\ell+1}.
$$
The definition of $t^{k}_{c_3,c_1}$ is similar.
We can check that they satisfy Maurer-Cartan equation 
by a straightforward calculation.
Let us explain it by figures.
\begin{figure}[h]
\centering
\includegraphics[scale=0.4]{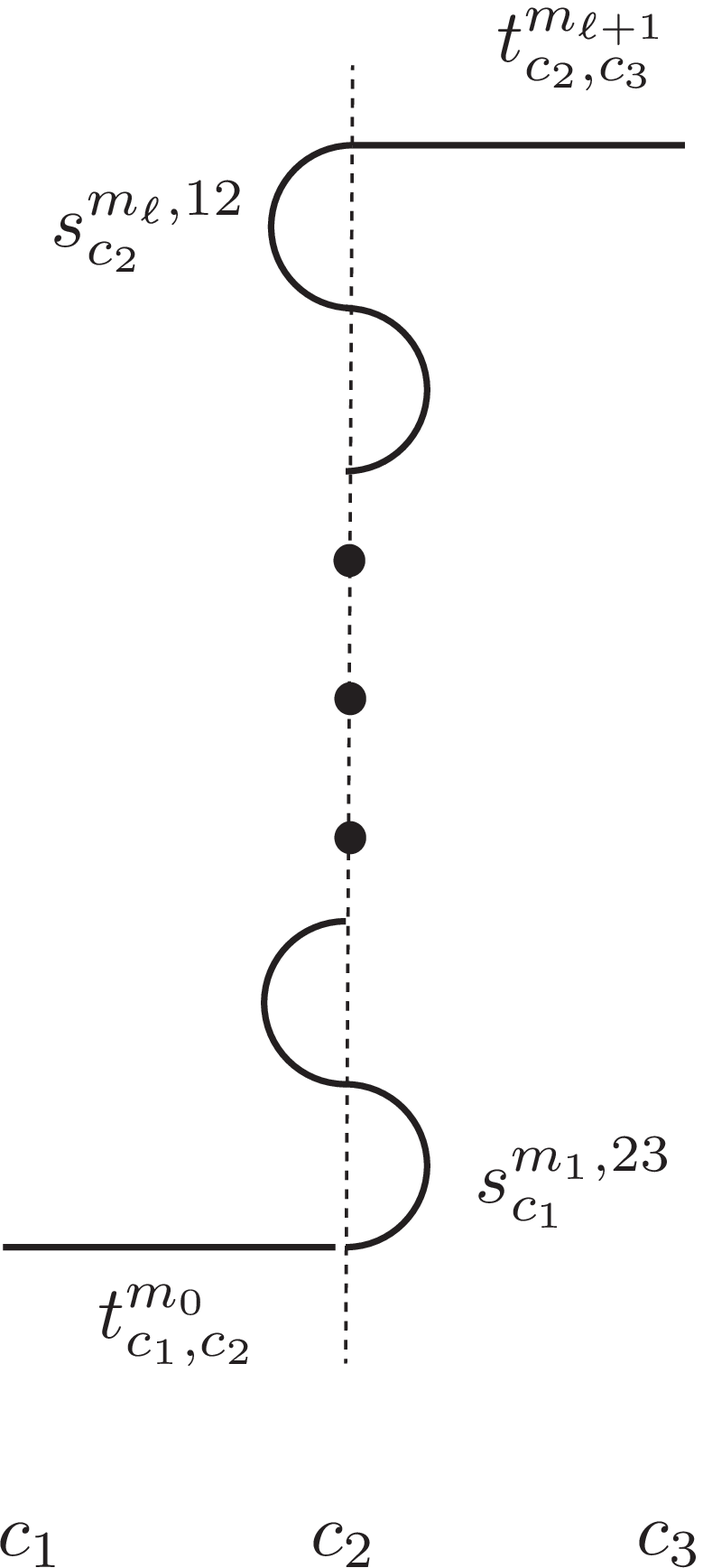}
\caption{$t^{k}_{c_1,c_3}$}
\label{tnoeee}
\end{figure}
Figure \ref{tnoeee} depicts $t^{k}_{c_1,c_3}$.
Its co-boundary $dt^{k}_{c_1,c_3}$ is obtained by applying 
the co-boundary operator $d$ to various factors.
When $d$ is applied to $t^{m_0}_{c_1,c_2}$ it gives
the terms depicted by (a)(b) of Figure \ref{hitthett}.
\begin{figure}[h]
\centering
\includegraphics[scale=0.4]{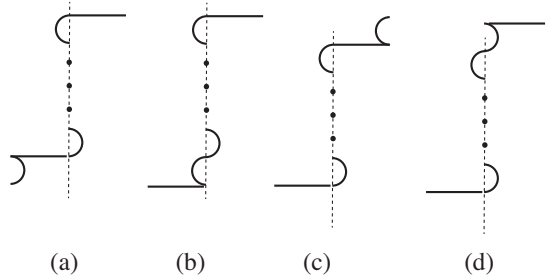}
\caption{$d t^{k}_{c_1,c_2}$ and $d t^{k}_{c_2,c_3}$}
\label{hitthett}
\end{figure}
When $d$ is applied to $t^{m_{\ell+1}}_{c_2,c_3}$ it gives
the terms depicted by (c)(d) of Figure \ref{hitthett}.
When $d$ is applied to $s^{m_1,23}_{c_2}$ it gives
the terms depicted by (e)(f)(g) of Figure \ref{hitthes23}.
Note that (g) appears only when $m_1 =1$.
\begin{figure}[h]
\centering
\includegraphics[scale=0.4]{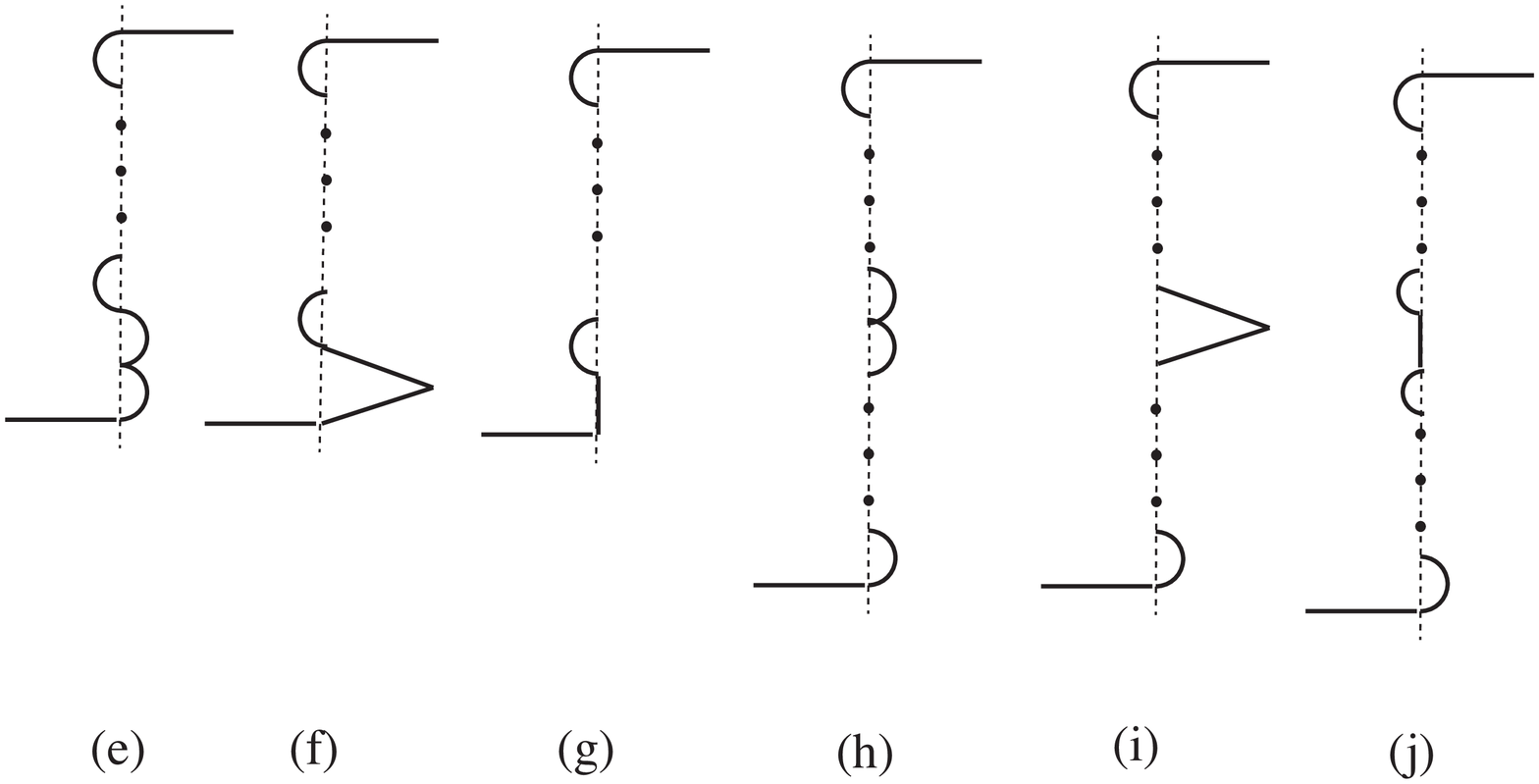}
\caption{$d s^{k,23}_{c_2}$}
\label{hitthes23}
\end{figure}
When $d$ is applied to $s^{m_i,23}_{c_2}$,
$i > 1$, it gives
the terms depicted by (h)(i)(j) of Figure \ref{hitthes23}.
Note that (j) appears only when $m_i =1$.
When $d$ is applied to $s^{m_\ell,12}_{c_2}$ it gives
the terms depicted by (k)(l)(m) of Figure \ref{hitthes12}.
Note that (m) appears only when $m_{\ell} =1$.
\begin{figure}[h]
\centering
\includegraphics[scale=0.4]{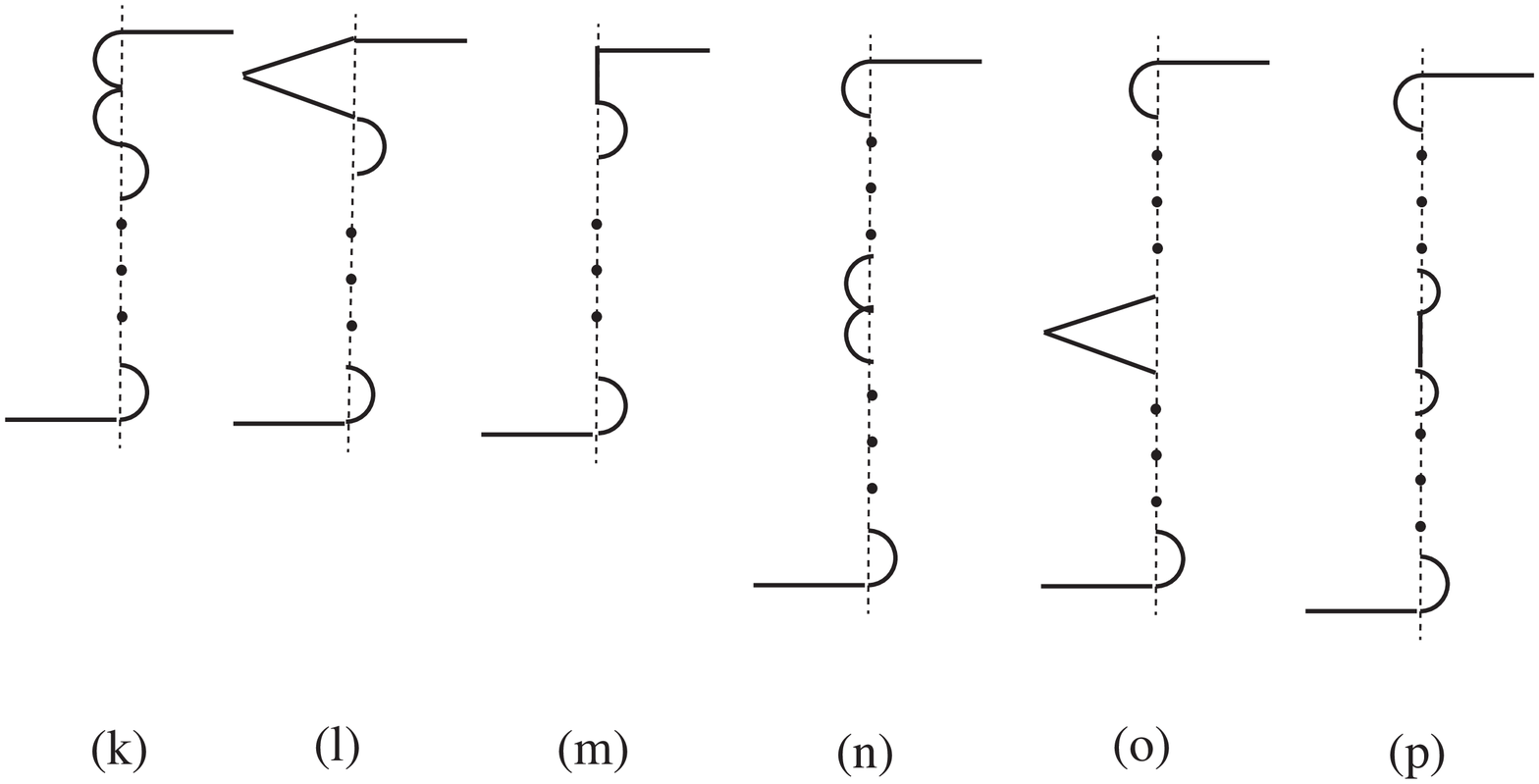}
\caption{$d s^{k,12}_{c_2}$}
\label{hitthes12}
\end{figure}
When $d$ is applied to $s^{m_i,12}_{c_2,c_2}$,
$i < \ell$, 
it gives the terms depicted by (n)(o)(p) of Figure \ref{hitthes23}.
Note that (p) appears only when $m_i =1$.
\par
Now we observe that (b) and (g), (d) and (m), 
(e)+(h) and (p), (k)+(n) and (j) 
cancel each other.
\par
(a) + (l) + (o) becomes 
$
\sum_{k_1+k_2 = k} s^{k_1,13}_{c_1}\circ t^{k_2,13}_{c_1,c_3}
$
and 
(c) + (f) + (i) becomes 
$
\sum_{k_1+k_2 = k} t^{k_1,13}_{c_1,c_3} \circ s^{k_2,13}_{c_3}.
$
We can check the sign easily. We thus verified 
(\ref{form128}).
We can verify (\ref{form129}) in a similar way.
\par
It is easy to calculate the energy loss of 
$(\{t^{k,13}_{c_1,c_3}\},\{t^{k,13}_{c_3,c_1},\{s^{k,13}_{c_1}\},
\{s^{k,13}_{c_3}\})$ and complete the proof of Theorem \ref{cong106}.
\end{proof}
\begin{proof}[Proof of Theorem \ref{idometric}.]
The proof is based on the obstruction theory 
and uses a construction of \cite{fooobook}.
Let $\mathscr C$ be a unital completed DG-category.
\begin{defn}\label{defn122888}
A {\it model of $\mathscr C \times[0,1]$} is a 
unital completed DG-category $\mathfrak{C}$ with the following properties.
(See \cite[Definition 4.2.1]{fooobook}.)
\begin{enumerate}
\item $\frak{OB}(\mathscr C) = \frak{OB}(\mathfrak C)$.
\item
There exist linear and unital morphisms of 
completed DG-categories
${\rm Eval}_{s=0} : \mathfrak C \to \mathscr C$,
${\rm Eval}_{s=1} : \mathfrak C \to \mathscr C$
and 
${\rm Incl} : \mathscr C \to  \mathfrak C$.
They are the identity maps for objects and preserve the finite parts.
\item
${\rm Eval}_{s=0}\circ {\rm Incl} = {\rm Eval}_{s=0}\circ {\rm Incl} 
= {\rm Identity}$.
\item
The linear parts of ${\rm Eval}_{s=0}$ and ${\rm Eval}_{s=1}$
induce cochain  homotopy equivalences 
$
\mathfrak C(c,c') \to \mathscr C(c,c')
$
for $c,c' \in \frak{OB}(\mathscr C)$. It preserves the finite parts.
\item
The map 
$
({\rm Eval}_{s=0})_1 \oplus
({\rm Eval}_{s=1})_1
: \mathfrak(a,b)
\to \mathscr C(c,c') \oplus \mathscr C(c,c')
$
is surjective. The restriction to the finite part is also surjective.
\end{enumerate}
\end{defn}
(Actually we do not use ${\rm Incl}$ in this paper.)
A model of $\mathscr C \times[0,1]$ always exists.
This is proved in the case of filtered $A_{\infty}$
algebra in \cite[Subsection 4.2.1]{fooobook}.
The construction there can be adapted to our case 
of unital completed filtered DG-category 
easily.
\par
We next modify Maurer-Cartan equation a bit.
Let $s^k_{c}$, $s^k_{c'}$, $t^k_{c,c'}$, $t^k_{c',c}$,  be a solution of Maurer-Cartan 
equation (\ref{form128}),  (\ref{form129}) 
with energy loss $(\epsilon_1,\epsilon_2)$.
We take 
$$
S^k_{c} = T^{k\epsilon}s^k_c,
\quad S^k_{c'} = T^{k\epsilon}s^k_{c'}
\qquad
T^k_{c,c'} = T^{\epsilon_1 + (k-1)\epsilon}t^k_{c,c'},
\qquad
T^k_{c',c} = T^{\epsilon_2 + (k-1)\epsilon}t^k_{c',c}.
$$
They satisfy:
\begin{equation}\label{form1282}
\left\{
\aligned
&d(T^k_{c,c'}) - 
\sum_{\ell=1}^{k-1} S^{\ell}_{c} \circ T^{k-\ell}_{c,c'}
+ \sum_{\ell=1}^{k-1} T^{k-\ell}_{c,c'} \circ  S^{\ell}_{c'} = 0 \\
&d(T^k_{ba}) - 
\sum_{\ell=1}^{k-1} S^{\ell}_{c'} \circ T^{k-\ell}_{c',c}
+ \sum_{\ell=1}^{k-1}  T^{k-\ell}_{c',c} \circ S^{\ell}_{c} = 0
\endaligned
\right.
\end{equation}
\begin{equation}\label{form1292}
\left\{
\aligned
&d(S^k_{c})  -
\sum_{\ell=1}^{k-1} S^{\ell}_{c} \circ S^{k-\ell}_{c}
+ \sum_{\ell=1}^{k} T^{1+k-\ell}_{c,c'} \circ T^{\ell}_{c',c} = 
\begin{cases} 0 &\text{$k\ne 1$}\\
T^{\epsilon}{\bf e}_{c} &\text{$k =  1$}
\end{cases}
\\
&d(S^k_{c'}) - 
\sum_{\ell=1}^{k-1} S^{\ell}_{c'} \circ S^{k-\ell}_{c'}
+ \sum_{\ell=1}^{k} T^{1+k-\ell}_{c',c} \circ T^{\ell}_{c,c'} = 
\begin{cases} 0 &\text{$k\ne 1$}\\
T^{\epsilon}{\bf e}_{c'} &\text{$k = 1$}
\end{cases}
\endaligned
\right.
\end{equation}
Moreover $S^k_{c},S^k_{c'}, T^k_{c,c'},T^k_{c',c}$
are contained in $\mathscr C(c,c)$, 
$\mathscr C(c',c')$, 
$\mathscr C(c,c')$, 
$\mathscr C(c',c)$,
respectively.
\par
When $(S^k_{c},S^k_{c'}, T^k_{c,c'},T^k_{c',c})$
is a solution of (\ref{form1282}),(\ref{form1292}) 
then, for $\delta > 0$, 
$(T^{2k\delta}S^k_{c}$, $T^{2k\delta}S^k_{c'}$, $T^{(2k-1)\delta}T^k_{c,c'}$,
$T^{(2k-1)\delta}T^k_{c',c})$
is a solution with $\epsilon$ replaced by $\epsilon + 2\delta$.
In case we can take $\delta$ as small as we want 
we say `raising the energy level arbitrary small amount'
to describe this process.
\par
Now we start the proof.
Let $c,c' \in \frak{OB}(\mathscr C)$.
The inequality 
$
d_{\rm Hof,\infty}^{\mathscr C_1}(c,c') \ge d_{\rm Hof,\infty}^{\mathscr C}(c,c')
$
is obvious. Here $d_{\rm Hof,\infty}^{\mathscr C_1}$,
$d_{\rm Hof,\infty}^{\mathscr C}$ 
(resp. $d_{\rm Hof,\infty}^{\mathscr C}$) is the Hofer infinite distance of 
 $\mathscr C_1$ (resp. $\mathscr C$).
We will prove the opposite inequality.
Suppose there exists an infinite homotopy equivalence between $c$ and $c'$ with 
energy loss
$\epsilon$.
Then we have a solution $S^k_{c},S^k_{c'}, T^k_{c,c'},T^k_{c',c}$
of (\ref{form1282}),(\ref{form1292}).
\par
We will construct a solution 
$\hat S^k_{c},\hat S^k_{c'}, \hat T^k_{c,c'},\hat T^k_{c',c}$ 
in $\mathscr C$ of 
(\ref{form1282}),(\ref{form1292}) with $\epsilon$ replaced by 
$\epsilon + 2\delta_K$  for $k=1,\dots,K$
by induction on $K$. 

\begin{enumerate}
\item[(a)] We require
${\rm Eval}_{s=1}(\hat S^k_{c}) =  T^{2k\delta_K}S^k_{c}$, 
${\rm Eval}_{s=1}(\hat S^k_{c'}) =  T^{2k\delta_K}S^k_{c'}$,
${\rm Eval}_{s=1}(\hat T^k_{c,c'}) =  T^{2(k-1)\delta_K}T^k_{c,c'}$,
and
${\rm Eval}_{s=1}(\hat T^k_{c',c}) =  T^{2(k-1)\delta_K}T^k_{c',c}$.
\item[(b)] We require
${\rm Eval}_{s=0}(\hat S^k_{c}) \in \mathscr C_1(c,c)$, 
${\rm Eval}_{s=0}(\hat S^k_{c'}) \in \mathscr C_1(c',c')$,
${\rm Eval}_{s=0}(\hat T^k_{c,c'}) \in \mathscr C_1(c,c')$,
and
${\rm Ev}_{s=0}(\hat T^k_{c',c}) \in \mathscr C_1(c',c)$.
\end{enumerate}
The construction is by induction on $K$.
There are two inductive steps.
The first is: 
given $\hat S^k_{c},\hat S^k_{c'}, \hat T^k_{c,c'},\hat T^k_{c',c}$ 
for $k<K$ 
we construct  $\hat T^K_{c,c'},\hat T^K_{c',c}$
by raising the energy level arbitrary small amount.
The second is: 
given $\hat S^k_{c},\hat S^k_{c'}$ for $k < K$ 
and $\hat T^k_{c,c'},\hat T^k_{c',c}$  for $k\le K$
we construct $\hat S^K_{c},\hat S^K_{c'}$
by raising the energy level arbitrary small amount.
\par
We describe the first inductive step.
The second is entirely similar.
\par
We put $\frak s^k_c: = {\rm Eval}_{s=0}(\hat S^k_{c})$
etc.
We consider
$$
OB = \sum_{\ell=1}^{K-1} \hat S^{\ell}_{c} \circ \hat T^{K-\ell}_{c,c'}
- \sum_{\ell=1}^{K-1} \hat T^{K-\ell}_{c,c'} \circ  \hat S^{\ell}_{c'}
$$
This is a cocycle and 
${\rm Eval}_{s=1}(OB) = \frak M_1(T^K_{c,c'})$.
Since ${\rm Eval}_{s=1}$ is surjective and 
cochain homotopy equivalence, we can show that 
there exists $\hat T^{K,\prime}_{c,c'} \in \frak C(c,c')$
such that  
$
\frak M_1(\hat T^{K,\prime}_{c,c'}) = OB$,
$
{\rm Eval}_{s=1}(\hat T^{K,\prime}_{c,c'}) =  T^{K,\prime}_{c,c'}.
$
We have
\begin{equation}\label{form1216}
\frak M_1({\rm Eval}_{s=0}(\hat T^{K,\prime}_{c,c'}))
= 
\sum_{\ell=1}^{K-1} \frak s^{\ell}_{c} \circ \frak t^{K-\ell}_{c,c'}
- \sum_{\ell=1}^{K-1} \frak t^{K-\ell}_{c,c'} \circ  \frak s^{\ell}_{c'}.
\end{equation}
We denote the right hand side by $\frak{ob}$.
It is a cycle in $\mathscr C_1(c,c')$. 
(\ref{form1216}) implies that $[\frak{ob}]$ is zero 
in $H(\mathscr C(c,c'))$.
Therefore by Lemma \ref{lemma217}
(the almost injectivity)
we may assume that $[\frak{ob}]$ is zero 
in $H(\mathscr C_1(c,c'))$ by raising the energy level arbitrary small amount.
Therefore we have
$\frak t^{K,\prime}_{c,c'} \in \mathscr C_1(c,c')$
such that 
$
\frak M_1(\frak t^{K,\prime}_{c,c'}) = \frak{ob}.
$
We consider 
$
{\rm Eval}_{s=0}(\hat T^{K,\prime}_{c,c'})
- \frak t^{K,\prime}_{c,c'}.
$
This is a cycle and represents an element 
of $H(\mathscr C(c,c'))$.
Therefore by Lemma \ref{lemma217}
(the almost surjectivity) we may 
change $\frak t^{K,\prime}_{c,c'}$ 
by a cycle in $\mathscr C_1(c,c')$
to $\frak t^{K}_{c,c'}$
such that
$
{\rm Eval}_{s=0}(\hat T^{K,\prime}_{c,c'})
- \frak t^{K}_{c,c'}
$
is a coboundary $\frak M_1({\rm C})$,
by raising the energy level arbitrary small amount.
By Definition \ref{defn122888} (5), 
there exists $\hat{\rm C} \in \mathfrak C_1(c,c')$
such that
$
{\rm Eval}_{s=0}(\hat{\rm C}) = C,
$
$
{\rm Eval}_{s=1}(\hat{\rm C}) = 0.
$
Now we change $\hat T^{K,\prime}_{c,c'}$ 
to
$
\hat T^{K}_{ab} = \hat T^{K,\prime}_{ab} - \frak M_1(\hat{\rm C}). 
$
Then
$\hat T^{K}_{c,c'}$ has the required properties.
\par
Thus by induction we obtain 
a solution $\frak s^k_{c},\frak s^k_{c'}, \frak t^k_{c,c'},\frak t^k_{c',c}$
of (\ref{form1282}),(\ref{form1292}) 
on $\frak C_1$ by raising the energy level arbitrary small amount.
The proof of Theorem  \ref{idometric} is complete.
\end{proof}

\section{Equivalence of inductive limit.}
\label{sec;equilim}

Let $\mathcal K$ be an index set.
We consider two $\mathcal K$ parametrized 
families of sequences 
$\frak c_{k,(i)} = (c^1_{k,(i)},c^2_{k,(i)},\dots)$ of objects of ${\mathscr C}$ for
$k \in \mathcal K$, $i=1,2$
such that
\begin{equation}\label{newform22rev}
d_{\rm Hof,\infty}(c^n_{k,(i)},c^{n+1}_{k,(i)})   <   \epsilon_n,
\qquad
\sum_{n=1}^{\infty} \epsilon_n   <   \infty.
\end{equation}
We put $\frak C_{\mathcal K,(i)} = \{\frak c_{k,(i)} \mid k \in \mathcal K\}$
and $\frak C^n_{\mathcal K,(i)} = \{c^n_{k,(i)} \mid k \in \mathcal K\}$.
By Theorem \ref{GrokaraIndunit} we obtain 
a $\#$-unital inductive system of 
unital filtered $A_{\infty}$ categories
$(\{\mathscr C(\frak C^n_{\mathcal K,(i)})\},\{\Phi^{n}_{(i)}\})$.
By Proposition \ref{prop101} and Definition \ref{defn9696666}, 
we obtain a homotopically unital
completed DG-categories 
as the inductive limits of these inductive systems.
We denote them by ${\mathscr C}^{\infty}(\frak C_{\mathcal K,(i)}):
= \varinjlim\, \mathscr C(\frak C^n_{\mathcal K,(i)})$
and denote the unitalization of 
${\mathscr C}^{\infty}(\frak C_{\mathcal K,(i)})$
by
${\mathscr C}^{\infty}(\frak C_{\mathcal K,(i)})^+ =
(\varinjlim\, \mathscr C(\frak C^n_{\mathcal K,(i)}))^+$.
We use reduced version of Bar-Co-Bar resolution.

\begin{thm}\label{thm131111}
In the above situation we have the following two inequalities.
\begin{equation}
\left\{
\aligned
d_{{\rm Hof},\infty}({\mathscr C}^{\infty}(\frak C_{\mathcal K,(1)})^+ ,
{\mathscr C}^{\infty}(\frak C_{\mathcal K,(2)})^+)
&\le
\liminf_{n\to \infty}
d_{{\rm H},\infty}(\frak C^n_{\mathcal K,(1)},\frak C^n_{\mathcal K,(2)}),
\\
d_{{\rm Hof}}({\mathscr C}^{\infty}(\frak C_{\mathcal K,(1)})^+ ,
{\mathscr C}^{\infty}(\frak C_{\mathcal K,(2)})^+) 
&\le
\liminf_{n\to \infty}
d_{\rm H}(\frak C^n_{\mathcal K,(1)},\frak C^n_{\mathcal K,(2)})).
\endaligned
\right.
\end{equation}
\end{thm}
\begin{proof}
We first construct another unital inductive system 
of filtered $A_{\infty}$ categories 
$(\{\mathscr C^n_{(0)})\},\{\Phi^{n}_{(0)}\})$.
Note that to define completed unital 
DG categories ${\mathscr C}^{\infty}(\frak C_{\mathcal K,(1)})^+$,  
${\mathscr C}^{\infty}(\frak C_{\mathcal K,(2)})^+$, 
we defined and used 
$(\{\mathscr C(\frak C^n_{\mathcal K,(i)})\},\{\Phi^{n}_{(i)}\})$
for $i=1,2$.
We made two choices to define it.
\begin{enumerate}
\item[$(*)$]
We take an infinite homotopy equivalence between $c_{k,(i)}^n$ and $c_{k,(i)}^{n+1}$
for each $n$ and $k$. (We used it to obtain $\Phi_{(i)}^n$ by Lemma \ref{Lemma3434unit}.)
\item[$(**)$]
We also take a splitting (\ref{unitsplitting}) to 
obtain a unital filtered $A_{\infty}$ functor from 
a $\#$-homotopical unital filtered $A_{\infty}$ functor.
\end{enumerate}
We fix these two choices for $i=1,2$.
We put $\mathscr C_{(0)}^n = {\mathscr C}^{(n)}(\frak{OB}(\mathscr C_{(1)}^n)\sqcup 
\frak{OB}(\mathscr C_{(2)}^n))$.
Note that we take {\it disjoint} union $\sqcup$ and not the union.
In other words in the case when 
$\frak{OB}(\mathscr C_{(1)}^n) \cap \frak{OB}(\mathscr C_{(2)}^n) \ne \emptyset$
we still regard the elements of the intersection  as two {\it different} elements.
\par
Now the choices $(*)(**)$ made for  $i=1,2$ induce the choices $(*)(**)$
for $i=0$.
Using them we obtain a unital inductive system of filtered $A_{\infty}$
categories 
$(\{\mathscr C^n_{(0)})\},\{\Phi^{n}_{(0)}\})$.
\par
Because of our choice 
we obtain a fully faithful embeddings of unital filtered $A_{\infty}$ categories
$
\Psi_{i}^n : \mathscr C(\frak C^n_{\mathcal K,(i)}) \to \mathscr C_{(0)}^n
$
with energy loss $0$ for $i = 1,2$, such that 
\begin{equation}\label{equality103rev}
\Phi_{(0)}^n\circ  \Psi_{i}^n = \Psi_{i}^{n+1} \circ \Phi_{(i)}^n.
\end{equation}
for $i=1,2$. We emphasis that this is an exact equality.
Therefore the second inequality is a consequence of Lemma \ref{1044lm}.
\par
The proof of the first inequality is mostly the same.
During the proof of Lemma \ref{1044lm} we obtained 
${\mathscr C}^{\infty}(\frak C_{\mathcal K,(0)})^+: = \varinjlim\,\mathscr C_{(0)}^n$
and an almost homotopy equivalent injection
to this completed DG-category from ${\mathscr C}^{\infty}(\frak C_{\mathcal K,(i)})^+$, 
for $i=1,2$ (Lemma \ref{lem12555}).
It suffices to estimate 
$d_{{\rm Hof},\infty}(\frak c_{k,(1)},\frak c_{k,(2)})$.
Here $\frak c_{k,(1)},\frak c_{k,(2)}$ are regarded as 
objects of ${\mathscr C}^{\infty}(\frak C_{\mathcal K,(0)})^+$.
\par
By construction there exists a unital filtered 
$A_{\infty}$ functor $(\Upsilon^{\infty,n})^+ :\mathscr C_{(0)}^n
\to {\mathscr C}^{\infty}(\frak C_{\mathcal K,(0)})^+$
of energy loss $\epsilon'_n$ such that
$\lim_{n\to \infty} \epsilon'_n = 0$.
Let $\rho_n = d_{{\rm H},\infty}(\frak C^n_{\mathcal K,(1)},\frak C^n_{\mathcal K,(2)})$
Note $\Upsilon^{\infty,n}_{\rm ob}(c^n_{\mathcal K,(i)}) = \frak c_{k,(i)}$.
Therefore by Lemma \ref{lem113} we have
$
d_{{\rm Hof},\infty}(\frak c_{k,(1)},\frak c_{k,(2)})
\le
\rho_n + \epsilon'_n.
$
Since $\epsilon'_n$ is arbitrary small the first inequality holds.
\end{proof}

We remark that the above proof implies the next:
\begin{prop}\label{proposi142}
The limit completed DG-category 
${\mathscr C}^{\infty}(\frak C_{\mathcal K,(i)})^+$
is independent of the choices $(*)$, $(**)$ up to equivalence.
\end{prop}

\begin{cor}\label{cor143434}
We consider a $\mathcal K$ parametrized 
family of sequences 
$\frak c_{k} = (c^1_{k},c^2_{k},\dots)$ of objects of ${\mathscr C}$ for
$k \in \mathcal K$ such that
$
d_{\rm Hof,\infty}(c^n_{k},c^{n+1}_{k})   <   \epsilon_n,
$
$
\sum_{n=1}^{\infty} \epsilon_n   <   \infty.
$
We put $\frak C^n_{\mathcal K} = \{c^n_{k} \mid k \in \mathcal K\}$.
Then the sequence $\mathscr C(\frak C^n_{\mathcal K})^+$, $n=1,2,\dots$ converges to 
$\varinjlim\,\mathscr C(\frak C^n_{\mathcal K})^+$ in Gromov-Hausdorff infinite
distance.
\end{cor}
\begin{proof}
We put $c^n_{k,(2)} = c^m_{k,(2)}$ for $n>m$ 
and 
$
c^n_{k,(2)} = 
c^n_{k,(2)}$  for $n\le m$. 
We also put
$c^n_{k,(1)}  = c^n_{k}$.
Then Theorem \ref{thm131111} implies
\begin{equation}
d_{{\rm Hof}}(\varinjlim\,\mathscr C(\frak C^n_{\mathcal K})^+,
\mathscr C(\frak C^m_{\mathcal K})^+)
\le
\liminf_{n\to \infty}
 d_{H,\infty}(\frak C^n_{\mathcal K},\frak C^m_{\mathcal K}).
\end{equation}
The corollary follows from this inequality.
\end{proof}

\begin{rem}\label{infhikkuri}
Let $(\{\mathscr C^n_{(i)}\},\{\Phi_{(i)}^n\})$
be a unital inductive system of filtered $A_{\infty}$ categories, for $i=1,2$.
The author does not know whether
$$
d_{\rm GH,\infty}(\varinjlim\, \mathscr C^n_{(1)}, 
\varinjlim\, \mathscr C^n_{(2)})
\le
\liminf_{n\to \infty}d_{\rm GH,\infty}
(\mathscr C^n_{(1)},\mathscr C^n_{(2)})
$$
holds or not. 
If one tries to prove it then it seems necessary 
to generalize Lemma \ref{1044lm} 
so that the exact equality (\ref{equality103}) is relaxed to 
a certain `homotopical commutativity'.
\end{rem}

\section{Spectre dimension and Gromov-Hausdorff distance.}
\label{sec;spedim}

In this section we study the relation of spectre dimension 
of the cohomlology group and Gromov-Hausdorff distance.
Let $\mathscr C_1$ and $\mathscr C_2$ be filtered $A_{\infty}$
category (gapped case) or completed DG-category (the case of the limit).
We assume that $d_{\rm GH,\infty}(\mathscr C_1,\mathscr C_2) < \epsilon$
and take $\mathscr C$ as in Definition \ref{inftywaekGH}. 

\begin{thm}\label{thm161}
Let $a,b \in \frak{OB}(\mathscr C_1)$ and $a',b' \in \frak{OB}(\mathscr C_2)$
such that
$$
d_{\rm Hof,\infty}(a,a')  < \epsilon_, \qquad d_{\rm Hof,\infty}(b,b') <  \epsilon.
$$
Let $H_1(a,b) = H(\mathscr C_1(a,b),\frak m_1)$, $H_2(a',b') = H(\mathscr C_2(a',b'),\frak m_1)$.
Then for $2\epsilon \le \lambda$ 
$$
f(\lambda;H_1(a,b)) \le  f(\lambda-2\epsilon;H_2(a',b')),
\qquad
f(\lambda;H_2(a',b')) \le  f(\lambda-2\epsilon;H_1(a,b)).
$$
In particular if 
$
d_{\rm Hof,\infty}(a,a') = d_{\rm Hof,\infty}(b,b') =0
$
then $H(\mathscr C_1(a,b),\frak m_1)$ is almost isomorphic to $H(\mathscr C_2(a',b'),\frak m_1)$.
\end{thm}
Theorem \ref{thm161}, Proposition \ref{proposi142} and Corollary \ref{cor143434}
imply the next:
\begin{cor}\label{cor15252}
Suppose that we are in the situation of Proposition \ref{proposi142}. 
Then 
$H(\mathscr C^{\infty}(\frak c_{k,(1)},\frak c_{k',(1}),\frak m_1)$
is almost isomorphic to $H(\mathscr C^{\infty}(\frak c_{k,(2)},\frak c_{k',(2)}),\frak m_1)$.
\par
In the situation of Corollary \ref{cor143434}.
We put $H_{\infty} = H(\mathscr C^{\infty}(\frak c_{k},\frak c_{k'}),\frak m_1)$
and $H_n = H(\mathscr C^{n}(c_{k},c_{k'}),\frak m_1)$.
Then, there exists $\epsilon_n$ which converges to $0$ and such that 
$$
f(\lambda;H_{\infty}) \le  f(\lambda-\epsilon_n;H_n),
\qquad
f(\lambda;H_n) \le  f(\lambda-\epsilon_n;H_{\infty}).
$$
\end{cor}

The proof of Theorem \ref{thm161} uses Lemma \ref{lem244}, which we now prove.
\begin{proof}[Proof of Lemma \ref{lem244}]
Let $\varphi: V_1 \to V_2$ be an almost isomorphism 
between divisible filtered $\Lambda_0$ modules.
Let $W \in V_2$ be a finitely generated submodule of $F^{\lambda}V_2$.
Let $x_i \in V_2$, $i=1,\dots,m$ be its generator. 
Note that $T^{-\lambda}x_i \in V_2$ exists by divisibility.
For $\epsilon > 0$,
there exists $y'_i\in V_1$ such that $T^{\epsilon-\lambda}x_i =
\varphi(y'_i)$.
We put $y_i = T^{\lambda-\epsilon}y'_i$.
Let $W'$ be a submodule of $F^{\lambda-\epsilon}V_1$ generated 
by $y_i$, $i=1,\dots,m$.
Then $\varphi(W') = W$.
\par
By definition, $W'$ is generated by $m'$ elements $z_1,\dots,z_{m'}$
with $m' = f(\lambda-\epsilon;V_1)$.
Therefore $W$ is generated by  $\varphi(z_1),\dots,\varphi(z_{m'})$.
We have proved
$
f(\lambda;V_2) \le \lim_{\epsilon>0, \epsilon \to 0} f(\lambda-\epsilon;V_1).
$
\par
We next prove an inequality of opposite direction.
Since $f(\lambda;V_1)$ is an integer valued non-increasing function, 
there exists a discrete set $S$ such that if $\lambda \notin S$ then 
there exists $\lambda' > \lambda$ such that 
$f(\lambda;V_1) = f(\lambda';V_1)$.
We take such $\lambda$, $\lambda'$. We put $\epsilon = \lambda' - \lambda$.
Let $W \subset F^{\lambda'}V_1$ be a finitely generated submodule 
such that the minimal number of generators of $W$ is $m = f(\lambda';V_1)$.
By divisibility there exists $W_1 \subset F^{\lambda}V_1$ 
such that $T^{\epsilon}W_1 = W$ and $W_1$ is finitely generated.
Let $\{x_1,\dots,x_m\}$ be a generator of $W_1$.
Let $W' = \varphi(W_1)$. Then $W'$ is a finitely generated submodule 
of $F^{\lambda}V_2$. 
Therefore it is generated by $m' =  f(\lambda;V_2)$ elements.
Suppose it is generated by $\varphi(y_1),\dots,\varphi(y_{m'})$.
There exists $a_{ij} \in \Lambda_0$
such that
$
\varphi \left(x_i - \sum_j a_{ij}y_j\right) = 0
$
for $i=1,\dots,m$. Since $\varphi$ is almost injective, it implies
$
T^{\epsilon}x_i = \sum_j T^{\epsilon}a_{ij}y_j.
$
Therefore $W$ is generated by $m'$ elements.
We have proved 
$
f(\lambda;V_1) \le f(\lambda;V_2),
$
for $\lambda \notin S$.
The proof of Lemma \ref{lem244} is complete.
\end{proof}
\begin{lem}\label{spectre2lemma}
Let $(C_1,d_1),(C_2,d_2)$ be filtered 
cochain complexes over $\Lambda_0$ and $\varphi: (C_1,d_1) \to (C_2,d_2)$ a 
cochain homotopy equivalence with energy loss $\epsilon$ in the 
sense of Definition \ref{defn4646}.
Then
$$
f(\lambda;H(C_1,d_1)) \le  f(\lambda-\epsilon;H(C_2,d_2)),
\qquad
f(\lambda;H(C_2,d_2)) \le  f(\lambda-\epsilon;H(C_1,d_1)).
$$
for $\lambda \ge \epsilon$.
\end{lem}
\begin{proof}
By definition there exists $\epsilon_1$, $\epsilon_2$ with $\epsilon_1+\epsilon_2
= \epsilon$ and 
$\psi: (C_1,d_1) \to (C_2,d_2)$
such that $T^{\epsilon_1}\varphi$ and $T^{\epsilon_2}\psi$ preserves the
energy filtration
and $T^{\epsilon_2}\psi \circ T^{\epsilon_1}\varphi$ and 
$T^{\epsilon_1}\varphi \circ T^{\epsilon_2}\psi$ are cochain homotopic 
to $T^{\epsilon}$, (by cochain homotopy which preserves filtrations).
Thus there exist $h = (T^{\epsilon_1}\varphi)_* : H(C_1,d_1)
\to H(C_2,d_2)$ and $g = (T^{\epsilon_2}\psi)_* : H(C_2,d_2)
\to H(C_1,d_1)$ such that $h\circ g = T^{\epsilon}$ 
and $g\circ h = T^{\epsilon}$.
\par
Let $W \subset F^{\lambda}H(C_1,d_1)$ be a finitely generated 
submodule. By divisibility, there exists $W_0\subset F^{\lambda-\epsilon}H(C_1,d_1)$ 
such that $W = T^{\epsilon}W_0$. 
Since $h(W_0) \subset F^{\lambda-\epsilon}H(C_2,d_2)$
it is generated by $f(\lambda-\epsilon;H(C_2,d_2))$
elements. Therefore $W = gh(W_0)$ is generated by $f(\lambda-\epsilon;H(C_2,d_2))$
elements.
We have proved 
$f(\lambda;H(C_1,d_1)) \le  f(\lambda-\epsilon;H(C_2,d_2))$.
We can prove the other inequality by exchanging the role of 
$g$ and $h$.
\end{proof}
Theorem \ref{thm161} follows from Theorem \ref{GrokaraInd} 
and Lemmas \ref{lem244} and \ref{spectre2lemma}.

\begin{rem}
The proof of Lemma \ref{spectre2lemma} is mostly 
the same as the proof of \cite[Lemma 6.5.31]{fooobook}
which was used to prove (\ref{Lipshich}).
A similar argument is used in the study of persistence modules.
See for example \cite[Section 2.2]{PRSZ}.
\end{rem}
\begin{proof}[The proof of Theorem \ref{them123333}]
By Theorem \ref{idometric} 
$\frak{OB}(\mathscr C_1)$ and 
$\frak{OB}(\mathscr C_2)$ 
are subspaces of 
$\frak{OB}(\mathscr C)$.
The Hausdorff distance 
between them in $\frak{OB}(\mathscr C)$
is zero.
Using the assumption that 
$\frak{OB}(\mathscr C_i)$
is complete and Hausdorff 
it implies that 
$\frak{OB}(\mathscr C_1)$ is 
isometric to  
$\frak{OB}(\mathscr C_2)$. 
The second half then follows from Theorem \ref{thm161}.
\end{proof}

\section{The proof of Theorem \ref{Lagcatecomplete}.}
\label{sec;compgeo}

In this section we prove Theorem \ref{Lagcatecomplete}.
Let 
$\{\mathbb L^n\}\in \mathfrak{Lag}^K_{\rm trans}$, $n=1,2,\dots$ be a 
Cauchy sequence with respect to  Hofer distance.
We first consider the following situation.
\begin{defn}\label{defn1555}
We say that $\{\mathbb L^n\}$ is a strongly transversal sequence 
if $\mathbb L^n \sqcup \mathbb L^{n+1}
\in \mathfrak{Lag}^{2K}_{\rm trans}$ holds for each $n$.
\end{defn}

\noindent{\bf Step 1}.
Let $\{\mathbb L^n\}$ be a strongly transversal Cauchy sequence.
We first fix a compatible almost complex structure.
We also choose a system of data needed to define the filtered $A_{\infty}$
category $\frak F(\mathbb L^n)$, that is, 
 Kuranishi structures on the moduli spaces of pseudo-holomorphic 
disks, strips and polygons involved and  CF-perturbations to obtain virtual fundamental chains.
We denote this choice by $\Xi^n$.
We have thus obtained $\frak F(\mathbb L^n)$.
\par
By \cite{Fu5} we can take those Kuranishi structures and 
CF-perturbations so that they are compatible with the 
forgetful maps of the boundary marked points (other than $0$-th 
marked point). Then the differential 0-form $1$ of each $L \in \mathbb L^n$ 
becomes a strict unit.\footnote{We do not study cyclic symmetry here.}
We next define $\frak F(\mathbb L^n \sqcup \mathbb L^{n+1})$.
Note that the moduli spaces of pseudo-holomorphic disks, strips and polygons
needed to define this filtered $A_{\infty}$ category
have Kuranishi structures. 
We construct the Kuranishi structures by induction on energy 
and complexity of the source.
Its boundary is a union of the fiber products of 
moduli spaces of pseudo-holomorphic 
disks, strips and polygons
which appear in the definition of $\frak F(\mathbb L^n)$
and $\frak F(\mathbb L^{n+1})$ or those which appeared 
in the earlier steps of induction.
Therefore the Kuranishi structures and CF-perturbations 
at the boundary of the moduli spaces 
are already given by $\Xi^n$, $\Xi^{n+1}$ or 
by induction hypothesis.
By \cite{const1,fooo:const2}
we can obtain Kuranshi structures 
on those moduli spaces of pseudo-holomorphic 
disks, strips and polygons so that they are compatible each other 
at the boundary and corners, 
and also for the part of the moduli space of 
disks, strips and polygons appearing in the construction 
of  $\frak F(\mathbb L^n)$
and $\frak F(\mathbb L^{n+1})$
it coincides with $\Xi^n$ and $\Xi^{n+1}$.
Then by \cite[Chapter 17]{fooo:springerbook}, 
we obtain a system of CF-perturbations 
so that they are  compatible with each other 
and also with the CF-perturbations given in  
$\Xi^n$ and $\Xi^{n+1}$.
We denote this system of Kuranishi structures 
and CF-perturbations by $\Xi^{n,n+1}$.
\par
We thus obtained $\frak F(\mathbb L^n \sqcup \mathbb L^{n+1})$.
By construction 
$\frak F(\mathbb L^n)$ and $\frak F(\mathbb L^{n+1})$ are 
its full subcategories.
In fact, since we extend the given system of Kuranishi structures 
and CF-perturbations to the one we need to 
obtain this filtered $A_{\infty}$ category, 
the structure operations of $\frak F(\mathbb L^n\cup\mathbb L^{n+1})$
coincide with the structure operations of $\frak F(\mathbb L^n)$
or $\frak F(\mathbb L^{n+1})$ as far as the morphisms between the objects
 in the latter categories concern.
\par\medskip
We put $\epsilon_n = d_{\rm Hof}(\mathbb L^n,\mathbb L^{n+1})$.
Then by Corollary \ref{corghchikachik2a}
the Haudorff distance (with respect to the Hofer infinite 
distance) 
of $\frak{OB}(\frak F(\mathbb L^n))$
and  $\frak{OB}(\frak F(\mathbb L^{n+1}))$
in $\frak{OB}(\mathbb L^n \sqcup \mathbb L^{n+1})$
is not greater than $\epsilon_n$.
Therefore by Theorem \ref{GrokaraIndunit} we obtain a $\#$-unital
filtered $A_{\infty}$ functor 
$\Phi^n : \frak{OB}(\frak F(\mathbb L^n)) 
\to \frak{OB}(\frak F(\mathbb L^{n+1}))$
with energy loss $\epsilon_n$.
Thus we obtain a $\#$-unital 
inductive system of filtered $A_{\infty}$
categories $(\{\frak F(\mathbb L^n)\},\{\Phi^n\})$.
Therefore we obtain a unital completed DG-category $\varinjlim\,\frak F(\mathbb L^n)$.
This depends on the choices $\{\Xi^n\},\{\Xi^{n,n+1}\}$
and  $(*)$, $(**)$ appearing in the proof of Theorem \ref{thm131111}.
\begin{defn}
Let $\{\mathbb L^{n,i}\}$ be a strongly transversal sequence 
for $i=1,2$. We say a pair $(\{\mathbb L^{n,1}\},\{\mathbb L^{n,2}\})$
is strongly transversal
if for any $n$ the set of Lagrangian submanifolds
$\mathbb L^{n,1} \cup \mathbb L^{n+1,1}\cup \mathbb L^{n,2} \cup \mathbb L^{n+1,2}$ 
is an element of $\mathfrak{Lag}^{4K}_{\rm trans}$.
\end{defn}
\noindent{\bf Step 2}. Let $(\{\mathbb L^{n,1}\},\{\mathbb L^{n,2}\})$ be a 
strongly transversal pair of strongly transversal Cauchy sequences.
We fix choices of 
$\{\Xi^{n,i}\},\{\Xi^{n,n+1,i}\}$ for $i=1,2$ and 
$(*)$, $(**)$.
We then obtain $\varinjlim\,\frak F(\mathbb L^{n,i})$
for $i=1,2$.
\begin{lem}\label{lem147}
In this situation we have:
\begin{equation}
d_{\rm GH,\infty}(\varinjlim\,\frak F(\mathbb L^{n,1}),
\varinjlim\,\frak F(\mathbb L^{n,2}))
\le \liminf_{n\to \infty} d_{\rm Hof}(\mathbb L^{1},\mathbb L^{2}).
\end{equation}
In the case the right hand side is $0$,
the inductive limit $\varinjlim\,\frak F(\mathbb L^{n,1})$
is equivalent to $\varinjlim\,\frak F(\mathbb L^{n,2})$.
\end{lem}
\begin{proof}
We first construct a filtered $A_{\infty}$ category
$\frak F(\mathbb L^{n,1} \sqcup \mathbb L^{n,2})$.
The construction is the same as the case of 
$\frak F(\mathbb L^{n}\sqcup\mathbb L^{n+1})$.
Namely we can extend the given system of 
Kuranishi structures and CF-perturbations 
$\Xi^{n,1}$, $\Xi^{n,2}$ to $\Xi^{n,1,2}$.
Therefore $\frak F(\mathbb L^{n,1})$ and 
$\frak F(\mathbb L^{n,2})$ are full subcategory 
of $\frak F(\mathbb L^{n,1} \sqcup \mathbb L^{n,2})$.
We next perform the same construction 
to obtain a system of 
Kuranishi structures and CF-perturbations
$\Xi^{n,n+1,1,2}$ and define 
$\frak F(\mathbb L^{n,1} \sqcup \mathbb L^{n+1,1} 
\sqcup \mathbb L^{n,2} \sqcup \mathbb L^{n+1,2})$.
The four filtered $A_{\infty}$ categories 
$\frak F(\mathbb L^{n,1} \sqcup \mathbb L^{n,2})$,
$\frak F(\mathbb L^{n+1,1} \sqcup \mathbb L^{n+1,2})$,
$\frak F(\mathbb L^{n,1} \sqcup \mathbb L^{n+1,1})$,
$\frak F(\mathbb L^{n,2} \sqcup \mathbb L^{n+1,2})$,
are full subcategories of 
$\frak F(\mathbb L^{n,1} \sqcup \mathbb L^{n+1,1} 
\sqcup \mathbb L^{n,2} \sqcup \mathbb L^{n+1,2})$.
Using again Corollary \ref{corghchikachik2a} and 
Theorem \ref{GrokaraIndunit} we obtain a $\#$-unital 
filtered $A_{\infty}$ functor 
$$
\Phi^{n,0} : 
\frak F(\mathbb L^{n,1} \sqcup \mathbb L^{n,2})
\to 
\frak F(\mathbb L^{n+1,1} \sqcup \mathbb L^{n+1,2})
$$
with energy loss $\epsilon_n$.
Here $\epsilon_n = \max (d_{\rm Hof}(\mathbb L^{n,1},\mathbb L^{n+1,1}),
d_{\rm Hof}(\mathbb L^{n,2},\mathbb L^{n+1,2}))$.
\par
Let $\Psi_i^{n}: \frak F(\mathbb L^{n,i})
\to \frak F(\mathbb L^{n,1} \sqcup \mathbb L^{n,2})$
be canonical embedding as full subcategories, for $i=1,2$.
By construction we have 
\begin{equation}\label{equality103rev2}
\Phi^{n,0}\circ  \Psi_{i}^n = \Psi_{i}^{n+1} \circ \Phi^{n,i}.
\end{equation}
This is an exact equality.
The rest of the proof is the same as the proof of Lemma \ref{1044lm}
and Theorem \ref{thm131111}.
\end{proof}
\par
\noindent{\bf Step 3}. 
Let $\mathbb L$ be an element of the completion $\overline{\mathfrak{Lag}^K}$
of $\mathfrak{Lag}^K$.
We first observe that there exists a strongly 
transversal Cauchy sequence $\{\mathbb L^{n}\}$ which 
converges to $\mathbb L$.
We next observe that for 
two strongly transversal  Cauchy sequences $\{\mathbb L^{n,i}\}$,
$i=1,2$ converging to $\mathbb L$, there 
exists a third strongly transversal  Cauchy sequence $\{\mathbb L^{n,3}\}$
converging to $\mathbb L$,
such that the pairs $(\{\mathbb L^{n,1}\},\{\mathbb L^{n,3}\})$
and $(\{\mathbb L^{n,2}\},\{\mathbb L^{n,3}\})$ are both 
strongly transversal.
Using Lemma \ref{lem147} twice we conclude that
$\varinjlim\,\frak F(\mathbb L^{n,1})$ is equivalent to 
$\varinjlim\,\frak F(\mathbb L^{n,2})$.
Thus by using the limit we obtain 
$\frak F(\mathbb L)$ which is well-defined up to equivalence.
\par\smallskip
\noindent{\bf Step 4}.
Let $\mathbb L_m \in \overline{\mathfrak{Lag}^K}$ and $\mathbb L_m$, $m=1,2,\dots$ 
be a Cauchy sequence.
We will prove that $\frak F(\mathbb L_m)$ converges with respect to the 
Gromov-Hausdorff infinite distance.
We put $\epsilon'_m = \sum_{\ell\ge m} d_{\rm Hof}(\mathbb L_{\ell},\mathbb L_{\ell+1})$.
We observe that there exists a 
$n,m \in \Z_{+}$ parametrized family of objects $\mathbb L^n_m 
\in \mathfrak{Lag}^K_{\rm trans}$ such that:
\begin{enumerate}
\item
$d_{\rm Hof}(\mathbb L_m,\mathbb L^n_m) < 1/n$,
\item
For any positive integers $m,m',m''$ we have
$\mathbb L^{n}_{m} \cup \mathbb L^{n+1}_{m}\sqcup \mathbb L^{n}_{m'} \sqcup \mathbb L^{n+1}_{m'}
\sqcup \mathbb L^{n}_{m''} \sqcup \mathbb L^{n+1}_{m''}
\in \mathfrak{Lag}^{6K}_{\rm trans}$.
\end{enumerate}
The existence of such $\mathbb L^n_m$ is easy to prove.
\par
We consider the sequence $\{\mathbb L^n_n\}$.
It is a Cauchy sequence and is strongly transversal.
Therefore we obtain the limit $\varinjlim\,\frak F(\mathbb L^n_n)$.
We consider another sequence, 
$\{\mathbb L^n_m\}$ for a fixed $m$.
Then by item (2), 
the pair $(\{\mathbb L^n_m\},\{\mathbb L^n_n\})$
is transversal.
Therefore by Lemma \ref{lem147} we have
$
d_{\rm GH,\infty}(\varinjlim\,\frak F(\mathbb L^n_n),\mathbb L_m)
\le \epsilon'_m.
$
Therefore $\frak F(\mathbb L_m)$ converges to $\varinjlim\,\frak F(\mathbb L^n_n)$
by Gromov-Hausdorff infinite distance.
The uniqueness of the limit up to equivalence can
be proved in the same way as Step 3.
\par\smallskip
\noindent{\bf Step 5}.
We take a compatible almost complex structure $J$ at the beginning of the 
construction.
We will prove that $\frak F(\mathbb L)$ is independent of the choice of $J$ up to 
equivalence.
Let $J_0, J_1$ be two compatible almost complex structures.
We take a path $\mathcal J = \{J_t\}$ of compatible almost complex structures joining them.
We may assume that $J_t = J_0$ for sufficiently small $t$ 
and  $J_t = J_1$ if $1- t$ is sufficiently small.
\par
Let $\mathbb L^n \in \mathfrak{Lag}^K_{\rm trans}$.
We take choices $\Xi_{0}^n$ (resp. $\Xi_1^{n}$), that is, Kuranishi structures and CF-perturbations of the moduli spaces 
of pseudo holomorphic disks, strips and polygons, with respect 
to the almost complex structure $J_0$ (resp $J_1$)
and obtain $\frak F(\mathbb L^n;J_0)$, $\frak F(\mathbb L^n;J_1)$.
\par
We consider the disjoint union $(\mathbb L^n \times \{0\})
\sqcup (\mathbb L^n \times \{1\})$ of two copies 
of $\mathbb L^n$.
We put $\widehat{\mathbb L}^{n} = (\mathbb L^n \times \{0\}) 
\sqcup (\mathbb L^n \times \{1\})$.
For an element $L$ of $\mathbb L^n$ we write $L(0)$ (resp. $L(1)$)
when we 
regard it as an object of $\frak F(\mathbb L^n;J_0)$
(resp. $\frak F(\mathbb L^n;J_1)$).
\par
Using $J_t$ and extending the given data 
$\Xi^n$, $\Xi^{\prime n}$ we obtain a pseudo-isotopy between
$\frak F(\mathbb L^n;J_0)$ and $\frak F(\mathbb L^n;J_1)$,
which we write $\frak F(\mathbb L^n;\mathcal J)$.
Roughly speaking this is the following object.
For $L\ne L'$, $L,L' \in \mathbb L^n$,
the morphism complex $CF(L, L';J_j)$ (for $j=0,1$)
between them is a free $\Lambda_0$ module 
$$
CF(L,L';J_j) = \bigoplus_{p \in L\cap L'} \Lambda_0 [p]
$$
generated by the intersection points.
We change it to 
$$
CF(L,L';\mathcal J) = \bigoplus_{p \in L\cap L'}  \Omega([0,1]) [p]
\,\widehat\otimes\,\Lambda_0.
$$
Here $\Omega([0,1])$ is the de Rham complex of the interval.
\par
For $L = L' \in \mathbb L^n$, 
we put
$$
CF(L,L;\mathcal J) =  \Omega(L\times [0,1]) \widehat\otimes
\Lambda_0.
$$
\par
We use the union of the moduli spaces of $J_t$-disks, strips and polygons for $t\in [0,1]$.
On these spaces we can define Kuranishi structures and CF perturbations 
extending $\Xi_0^n$, $\Xi^{n}_1$, 
which we denote by $\Xi_{01}^n$. See \cite{const1,fooo:const2} for the detail of this construction.
\par
This filtered $A_{\infty}$ category is curved.
(The operator $\frak m_0$ may be non-zero.)
Using bounding cochains we can eliminate $\frak m_0$ and obtain a strict filtered $A_{\infty}$
category which we denote by $\frak F(\mathbb L^n;\mathcal J)$.
\par
The pullbacks of differential forms by the inclusions $\{0\} \to [0,1]$, 
$\{1\} \to [0,1]$ define filtered $A_{\infty}$ functors (of energy loss $0$)
$\frak J^n_0: \frak F(\mathbb L^n;\mathcal J) \to \frak F(\mathbb L^n;J_0)$,
$\frak J^n_1: \frak F(\mathbb L^n;\mathcal J) \to \frak F(\mathbb L^n;J_1)$
which are homotopy equivalences.
\par
To use the filtered $A_{\infty}$ category 
$\frak F(\mathbb L^n;\mathcal J)$ to prove that  $\frak F(\mathbb L^n;J_0)$ 
is equivalent to 
$\frak F(\mathbb L^n;J_1)$, we invert the arrows of $\frak J^n_j$ as follows.
We first review the structure operations $\frak M_k$ of $\frak F(\mathbb L^n;\mathcal J)$.
\par
Let $h_i$ be a differential form on $[0,1]\times (L\cap L')$ or $[0,1] \times L$.
Let $t$ be the standard coordinate of $[0,1]$.
We write $h_i$ as 
$
h_i = h_i^0 + dt\wedge h_i^1
$
where $h_i^0$, $h_i^1$ are 
smoothly $t$-dependent family of smooth differential forms on $L$ 
or on $L \cap L'$. (In case 
$L\ne L'$ the latter is a 
finite set.)
Then
\begin{equation}\label{formula153}
\aligned
&\frak M_k(h_1,\dots,h_k)\\
&= 
\frak m^{\sharp}_k(h^0_1,\dots,h^0_k) 
+ dt \wedge \frak c_k(h^0_1,\dots,h^0_k) \\
&\,\,\,\,\,+dt \wedge \sum_{i=1}^k (-1)^{1+ 
\deg' h^0_1 + \dots + \deg' h^0_{i-1}} \frak m^{\sharp}_k(h^0_1,\dots,h^1_i,\dots,h^0_k)
\endaligned
\end{equation}
for $k\ne 1$ and
\begin{equation}\label{formula154}
\aligned
\frak M_1(h)
= 
&\frak m^{\sharp}_1(h^0) - \frac{\partial h^0}{\partial t}
+ dt \wedge \frak c_1(h^0) -dt \wedge  \frak m^{\sharp}_1(h^1)
\endaligned
\end{equation}
where $\frak m^{\sharp}_k$ and $\frak c_k$ are 
smoothly $t$-depedent family of operators 
with smooth Schwartz kernels, 
except $\frak m_1^{\sharp}$ contains the de Rham differential $d$ on $L$
and $\frak m_2^{\sharp}$ contains the cup product.
(They do not have smooth Schwartz kernels but their $t$ derivatives are $0$.)
(See \cite[Section 22.3]{fooo:springerbook}.)
\par
A filtered $A_{\infty}$ functor $\Psi^{n,c}_0: \frak F_c(\mathbb L^n;J_0) \to \frak F_c(\mathbb L^n;\mathcal J)$
we look for is of the form
$$
(\Psi^{n,c}_0)_k(g_1,\dots,g_k)
= (\Psi^{n,c}_0)^0_k(g_1,\dots,g_k) + dt \wedge (\Psi^{n,c}_0)^1_k(g_1,\dots,g_k),
$$
where $g_i$ are differential forms on $L$ or on $L\cap L'$ and 
$(\Psi^{n,c}_0)^0_k$, $(\Psi^{n,c}_0)^1_k$ have smooth Schwartz kernels, 
except $(\Psi^{n,c}_0)^0_1$ contains the identity map.\footnote{The 
Schwartz kernel of the identity map is a distributional form which is 
supported at the diagonal.}
(Note that the right hand side is a differential form on $[0,1] \times L$
or $[0,1] \times (L\cap L')$.) 
\begin{lem}\label{lem158}
There exists uniquely such a filtered $A_{\infty}$ functor 
with the following additional properties:
\begin{enumerate}
\item The composition $\frak J^n_0 \circ \Psi^{n,c}$ is the identity functor.
\item $(\Psi^{n,c}_0)^1_k = 0$.
\end{enumerate}
\end{lem}
\begin{proof}
We consider the condition that $(\Psi^{n,c}_0)^0_k$ and $(\Psi^{n,c}_0)^1_k = 0$
define a filtered $A_{\infty}$ functor.
Hereafter during the proof of Lemma \ref{lem158}, we write $\Psi_k$
in place of $(\Psi^{n,c}_0)^0_k$.
The condition consists of two systems of equations: one concerns the part 
which contains $dt$, the other  concerns the part 
which does not contain $dt$. 
In view of Formulas (\ref{formula153}) and (\ref{formula154}), 
the first one is:
\begin{equation}\label{eq15555}
0 = -\frac{d\Psi_k}{dt}({\bf x}) +
\sum_{\ell\le k}
\sum_{c\in I({\bf x},\ell)}\frak c_{\ell}
\left(
\Psi_{*}({\bf x}_{c;1}),\dots,
\Psi_{*}({\bf x}_{c;\ell})\right).
\end{equation}
The explanation of the second term of the right hand side is in order.
We put 
$$
\Delta^{\ell}: = (\Delta \otimes 
\underbrace{{\rm id} \otimes \dots \otimes {\rm id}}_{\ell-1})
\circ \dots \circ \Delta.
$$
This is a $\Lambda_0$ homomorphism from $B\mathscr C(a,b)$ to 
$B_{\ell}(B\mathscr C)(a,b)$.
($\mathscr C = \frak F(\mathbb L^n;J_0)$ here.)
We then write
\begin{equation}\label{Swedler}
\Delta^{\ell}({\bf x})
= \sum_{c\in I({\bf x},\ell)}{\bf x}_{c;1} \otimes \dots \otimes {\bf x}_{c;\ell},
\end{equation}
where $I({\bf x},\ell)$ is a certain index set depending on ${\bf x}$ and $\ell$.
In case ${\bf x}_{c;i} \in B_{m_i}\mathscr C$
$\Psi_{*}({\bf x}_{c;i})$ is $\Psi_{m_i}({\bf x}_{c;i})$ by definition.
\par
Note that $\Psi_k$ can be regarded as a $t$-parametrized family of operations
$$
\Psi_k(t) : B_k\mathscr C(a,b) \to \mathscr C(a,b)
$$
which is smooth in $t$ variables. Thus the first term makes sense.
The condition (1) gives an initial condition
(\ref{eq15555}), that is,
\begin{equation}\label{inicond1}
\Psi_1(0) = {\rm identity}, \qquad \Psi_k(0) = 0, \quad k > 1.
\end{equation}
We observe that $\frak c_1 \equiv 0 \mod \Lambda_+$.
Note that we can regard (\ref{eq15555}) 
as an ordinary differential equation on the Schwartz kernels.
Using these facts we can solve (\ref{eq15555}) uniquely with 
initial condition (\ref{inicond1}) by induction 
on energy and number filtrations.
(We use gapped-ness here.)
\par
To complete the proof of Lemma \ref{lem158}
it suffices to check that $\{\Psi_k\}$ defines a filtered $A_{\infty}$
functor.
The part of this condition which contains $dt$ is (\ref{eq15555}).
The other part can be written as follows.
\begin{equation}\label{formula157}
\aligned
&\sum_{\ell}\sum_{c\in I({\bf x},\ell)}
\sum_{i=1}^{\ell}(-1)^{*}\Psi_{\ell} 
({\bf x}_{c;1} \dots \frak m_*({\bf x}_{c;i})\dots {\bf x}_{c;\ell}) \\
&=\sum_{\ell}\sum_{c\in I({\bf x},\ell)}
\frak m^{\sharp}_{\ell}
(\Psi_{*}({\bf x}_{c;1}), \dots \Psi_{*}({\bf x}_{c;\ell})),
\endaligned
\end{equation}
where the sign $*$ is $* = \deg'{\bf x}_{c;1} + \dots + \deg'{\bf x}_{c;i-1}$.
We will prove this equality by induction on energy and number 
filtrations.
We remark that (\ref{formula157}) is satisfied at 
$t=0$ by (\ref{inicond1}).
\par
The $t$ derivative of the left hand side of (\ref{formula157}) becomes:
\begin{equation}\label{form15888}
\frak c\circ \widehat{\Psi} \circ \hat d = 
\frak c\circ \hat d^{\sharp} \circ \widehat{\Psi},
\end{equation}
where $\widehat{\Psi}$ is the co-homomorphism 
induced from $\{\Psi_k\}$, $\hat d$ is the coderivation induced from $\{\frak m_k\}$
and $\hat d^{\sharp}$ is the coderivation induced from $\{\frak m^{\sharp}_k\}$.
Using the fact that $\frak c_k$ strictly increases energy 
filtration we can use induction hypothesis to prove the equality in 
(\ref{form15888}).\footnote{We use the gapped-ness of our filtered 
$A_{\infty}$ structure here.}
\par
The $t$-derivative of the right hand side of (\ref{formula157}) 
is
\begin{equation}\label{form159999}
\frac{d\frak m^{\sharp}_*}{dt}(\widehat{\Psi}({\bf x}))
+\sum_{c\in I({\bf x},3)}
\frak m^{\sharp}_*\left(\widehat{\Psi}({\bf x}_{c;1})
\otimes \frak c_*(\widehat{\Psi}({\bf x}_{c;2}))\otimes
\widehat{\Psi}({\bf x}_{c;3}) \right)
\end{equation}
Here we use (\ref{eq15555}).
\par
For ${\bf y} \in B\frak F_c(\mathbb L^n;\mathcal J)$,
the $A_{\infty}$ relations of the structure operations 
induce the following formula:
\begin{equation}\label{form15333}
- \frac{d\frak m^{\sharp}_*}{dt}({\bf y})
+ (\frak c_* \circ \widehat d^{\sharp})({\bf y})
- 
\sum_{c\in I({\bf y},3)}
\frak m^{\sharp}_*\left({\bf y}_{c;1}
\otimes \frak c_*({\bf y}_{c;2})\otimes
{\bf y}_{c;3} \right)
= 0.
\end{equation}
We apply (\ref{form15333}) to ${\bf y} = \widehat{\Psi}({\bf x})$
to obtain (\ref{form15888}) = (\ref{form159999}).
\par
Thus (\ref{formula157}) is proved by induction.
The proof of Lemma \ref{lem158} is complete.
\end{proof}
In the same way we obtain 
a filtered $A_{\infty}$ functor $\Psi^{n,c}_1: \frak F_c(\mathbb L^n;J_1) \to \frak F_c(\mathbb L^n;\mathcal J)$.
\par
Using the bounding cochains, 
the maps $\Psi^{n,c}_0$,$\Psi^{n,c}_1$ induce 
filtered $A_{\infty}$ functor
$
\frak F(\mathbb L^n;J_j) \to \frak F(\mathbb L^n;\mathcal J)
$
for $j=0,1$, which we denote by $\Psi^n_j$.
\par
Let $b(0)$ be a bounding cochain of $L \in \mathbb L^n$
with respect to the filtered $A_{\infty}$ structure 
induced by an almost complex structure $J_0$.
We then obtain 
$
b(1): = (\frak J^n_1 \circ \Psi^n_0)_*(b(0)).
$
Here $b(1)$ is a bounding cochain of $L \in \mathbb L$
with respect to the filtered $A_{\infty}$ structure 
induced by an almost complex structure $J_1$.
Inspecting the definition 
we have
$b(t)$ which solves
$
\frac{d b(t)}{dt} = \sum_{\ell}
\frak c_{\ell}(b(t),\dots,b(t)),
$
so that at $t=0,1$ it becomes $b(0)$ and $b(1)$.
Using this fact we can find
$
(\Psi^n_0)_{\rm ob}(L,b(0)) = (\Psi^n_1)_{\rm ob}(L,b(1)).
$
This fact and the fact that $\Psi^n_j$ is a homotopy equivalence for $j=0,1$ 
implies 
\begin{equation}\label{eq1512}
d_{\rm H,\infty}(\frak{OB}(\frak F(\mathbb L^n;J_0;\Xi^n_0)),
\frak{OB}(\frak F(\mathbb L^n;J_1;\Xi^n_1))) = 0.
\end{equation}
Here  Hausdorff distance is taken with respect to  
Hofer infinite distance.
\par
We can perform the same construction for 
a pair $n,n+1$. 
Namely we obtain filtered $A_{\infty}$ categories
$\frak F(\mathbb L^n\sqcup \mathbb L^{n+1};\mathcal J)$
and filtered $A_{\infty}$ functors
$$
\Psi^{n,n+1}_j : 
\frak F(\mathbb L^n\sqcup \mathbb L^{n+1};J_j) 
\to \frak F(\mathbb L^n\sqcup \mathbb L^{n+1};\mathcal J).
$$
$\frak F(\mathbb L^n;\mathcal J)$
and $\frak F(\mathbb L^{n+1};\mathcal J)$
are full subcategories of 
$\frak F(\mathbb L^n\sqcup \mathbb L^{n+1};\mathcal J)$.
The next diagram commutes.
\begin{equation}\label{diagram1512}
\begin{CD}
\frak F(\mathbb L^{n};\mathcal J)
@>>>
\frak F(\mathbb L^{n}\sqcup \mathbb L^{n+1};\mathcal J) @ <<<
\frak F(\mathbb L^{n+1};\mathcal J)\\
@ AA{\Psi^{n}_j}A @ AA{\Psi^{n,n+1}_j}A@ AA{\Psi^{n+1}_j}A \\
\frak F(\mathbb L^{n};J_j)
@>>>
\frak F(\mathbb L^n\sqcup \mathbb L^{n+1};J_j) @ <<< \frak F(\mathbb L^{n+1};J_j)
\end{CD}
\end{equation}
Using the first line of Diagram (\ref{diagram1512}) 
we obtain 
$\Phi^n_{\mathcal J}: \frak F(\mathbb L^n;\mathcal J)
\to \frak F(\mathbb L^{n+1};\mathcal J)$.
Thus we obtain an
inductive system $(\{\frak F(\mathbb L^n;\mathcal J),\{\Phi^n_{\mathcal J}\}\})$.
\par
Moreover for  $(\{\frak F(\mathbb L^n;J_0),\{\Phi^n_{J_0}\}\})$
and $(\{\frak F(\mathbb L^n;J_1),\{\Phi^n_{J_1}\}\})$
the commutativity of Diagram (\ref{diagram1512})
implies
$
\Psi^{n+1}_j \circ  \Phi^n_{j}  =  \Phi^n_{\mathcal J} \circ \Psi^{n}_j,
$
for $j=0,1$.
This is an exact equality.
The rest of the proof of independence of $J$ is the same as the proof of Lemma \ref{1044lm}
and Theorem \ref{mainalgtheorem2}.
We use $\varinjlim\,\frak F(\mathbb L^n;\mathcal J)$, 
and (\ref{eq1512})
to conclude that $\varinjlim\,\frak F(\mathbb L^n;J_0;\Xi^n_0)$
is equivalent to $\varinjlim\,\frak F(\mathbb L^n;J_1;\Xi^n_1))$.
\par\smallskip
\noindent{\bf Step 6}.
We finally mention the `homotopy inductive limit argument'.
When we construct a filtered $A_{\infty}$ category $\frak F(\mathbb L)$
for $\mathbb L \in \mathfrak{Lag}^k_{\rm trans}$, we use the moduli spaces of 
pseudo-holomorphic disks, strips, and polygons
and Kuranishi structures and CF-perturbations thereon.
Actually we need to restrict ourselves to finitely many moduli spaces for this purpose.
The reason is explained in detail in \cite[Subsection 7.2.3]{fooobook2}.
The following way to resolve this problem is established in \cite[Section 7]{fooobook2}
and is used repeatedly in various papers by FOOO.
We use the notion of a filtered $A_{m,K}$-structure.
Here the natural number $m$ indicates that the structure operations $\frak m_k$ 
are defined only for $k \le m$
and the positive real number $K$ indicates that the structure operations are 
defined of the form 
$
\sum_{i} T^{\lambda_i} \overline{\frak m}_{k,i}
$
with $\lambda_i \le K$ and $\overline{\frak m}_{k,i}$ 
being $\R$ linear. Moreover the $A_{\infty}$ relation holds only modulo 
$T^{K}$. 
(See \cite[Subsection 7.2.6]{fooobook2} for the detail of the definition of 
filtered $A_{m,K}$ structures.)
We need to study only finitely many
moduli spaces to obtain such a structure.
We thus obtain $\frak F(\mathbb L;n,K)$.
Note that we obtain $A_{m,K}$ structure for any but finite $m,K$.
We next show that for $m\le m'$ and $K \le K'$
the filtered $A_{m',K'}$ category $\frak F(\mathbb L;m',K')$
regarded as a filtered $A_{m,K}$ category 
is homotopy equivalent to $\frak F(\mathbb L;m,K)$.
Then by an obstruction theory the $A_{m,K}$ structure 
on $\frak F(\mathbb L;m,K)$ extends to an $A_{m',K'}$ structure
without changing its underlying $A_{m,K}$ structure,
(in the chain level).
Thus going to $m' \to \infty$, $K'\to \infty$ we 
obtain a filtered $A_{\infty}$ structure.
See \cite[Subsection 7.2]{fooobook2} or 
\cite[Section 22.2]{fooo:springerbook} for detail.
\par
Now we can easily combine this `homotopy inductive limit argument' to our 
discussion to obtain a limit.
Suppose we have a strongly transversal sequence 
$\{\mathbb L^n\}$. 
We take for each $m,K$ a data $\Xi^n_{m,K}$,
that is, a system of Kuranishi structures and CF-perturbations,
on a finitely many moduli spaces, to obtain an $A_{m,K}$
structure. Then we use them for all $m,K$ to 
promote $A_{m,K}$ structure to $A_{\infty}$ structure.
This is the way we obtain $\frak F(\mathbb L^n)$.
\par
Now for $n$ and $n+1$, we obtained 
$\frak F(\mathbb L^n;m,K)$ and $\frak F(\mathbb L^{n+1};m,K)$.
We can extend it to obtain 
$\frak F(\mathbb L^{n}\sqcup \mathbb L^{n+1};m,K)$.
Now we can perform the process of extending $A_{m,K}$ 
structure to $A_{m',K'}$ structure in a way consistent with the 
fully faithful embeddings 
$$
\frak F(\mathbb L^{n};m,K)
\to \frak F(\mathbb L^{n}\sqcup \mathbb L^{n+1};m,K),
\quad 
\frak F(\mathbb L^{n+1};m,K)
\to \frak F(\mathbb L^{n}\sqcup \mathbb L^{n+1};m,K).
$$
Thus we obtain 
$\frak F(\mathbb L^{n}\sqcup \mathbb L^{n+1})$ 
together with fully faithful embeddings:
$$
\frak F(\mathbb L^{n})
\to \frak F(\mathbb L^{n}\sqcup \mathbb L^{n+1}),
\quad 
\frak F(\mathbb L^{n+1})
\to \frak F(\mathbb L^{n}\sqcup \mathbb L^{n+1}).
$$
Thus the construction of Step 1 works.
The construction of the other steps can be combined with the 
`homotopy inductive limit argument' in the same way.
\par
The proof of Theorem \ref{Lagcatecomplete} is now complete.
\qed
\begin{proof}[Proof of Corollary \ref{maintheorem1}]
The completed DG-category $\frak F(\mathbb L)$ is obtained 
in Theorem \ref{Lagcatecomplete}.
(1) is immediate from construction.
(2) is the consequence of Corollary \ref{cor15252}.
\end{proof}

\section{Hofer infinite distance in Lagrangian Floer theory.}
\label{sec;hoferinf1}

In this section, we prove Theorem \ref{estimateHofLaginf2}.
This is an improvement of Theorem \ref{estimateHofLag} which is 
\cite[Theorem 15.5]{FuFu6}. The proof of Theorem  \ref{estimateHofLaginf2} is also 
an enhancement of the proof of \cite[Theorem 15.5]{FuFu6}.
We begin with a review of the proof of \cite[Theorem 15.5]{FuFu6}.
Suppose we have a finite set of relatively spin compact Lagrangian 
submanifolds $\mathbb L$ of $(X,\omega)$.
(We can include the case when elements of $\mathbb L$ is immersed, 
if they are self-clean.)
We then obtain a (strict) filtered $A_{\infty}$ category 
$\frak F(\mathbb L)$ whose object is a pair 
$(L,b)$ where $L \in \mathbb L$ and $b \in \Lambda_+CF^{\rm odd}(L,L)$ is a 
bounding cochain.
Note that $CF(L,L) = \Omega(\tilde L \times_X \tilde L) \widehat{\otimes} \Lambda_0$ and $b$ is assumed to satisfy  Maurer-Cartan 
equation
$
\sum_{k=0}^{\infty} \frak m_k(b,\dots,b) = 0.
$
Let $\varphi : X \to X$ be a Hamiltonian diffeomorphism.
We assume $\varphi(L) \in \frak F(\mathbb L)$.
Theorem \ref{estimateHofLag} = \cite[Theorem 15.5]{FuFu6}
claims
\begin{equation}
d_{\rm Hof}((L,b),(\varphi(L),\varphi_*(b)) \le d_{\rm Hof}({\rm id},\varphi).
\end{equation}
The proof uses Yoneda embedding $\frak{YON}: \frak F(\mathbb L) \to \frak{Rep}(\frak F(\mathbb L)^{\rm op},
\mathcal{CH})$.
The target of this functor can be identified with 
the full sub-category of the filtered DG-category of right filtered $A_{\infty}$ modules,
which is represented by an object of $\frak F(\mathbb L)$.
Using the fact that Yoneda embedding is a homotopy equivalece to the 
image (\cite[Theorem 9.1]{fu4}) Corollary \ref{invhoferinfdist} implies that 
$$
d_{\rm Hof,\infty}(\frak{YON}(L,b),\frak{YON}(\varphi(L),\varphi_*(b)))
=
d_{\rm Hof,\infty}((L,b),(\varphi(L),\varphi_*(b)).
$$
Therefore we estimate the left hand side.
We write $\frak Y(L,b)$
and $\frak Y(\varphi(L),\varphi_*(b))$ instead of 
$\frak{YON}(L,b)$ and $ \frak{YON}(\varphi(L),\varphi_*(b))$
from now on.
\par
We review the definition of those objects.
The description below is the same as \cite{fu4}  except the following 
two points:
We use the notion of a right module rather than a functor to the 
DG-category of chain complexes.  (They are equivalent.)
We work over Novikov ring while the discussion of \cite{fu4}  is 
the case of $A_{\infty}$ category over the ground ring $R$.
There is no difference between them since our 
filtered $A_{\infty}$ category $\frak F(\mathbb L)$ is strict.
\par
A right  filtered $A_{\infty}$ module over $\frak F(\mathbb L)$
assigns a cochain complex 
$D(a)$ to each object $a$ of $\frak F(\mathbb L)$ and structure maps
$$
\frak n_k: D(a) \otimes B_k \frak F(\mathbb L)[1](a,b) 
\to D(b), 
$$
to each $k$ and objects $a,b$.
Note  that $\frak n_0$ is the co-boundary operator of $D(a)$.
We require that they satisfy  $A_{\infty}$ relations
\begin{equation}\label{rightmoddef}
\aligned
\sum_{c\in I({\bf x},2)}  \frak n_*(\frak n_*(y;{\bf x}_{2;1});{\bf x}_{2;2})
&+ \frak n_0(\frak n_k(y;{\bf x}))
\\
&+ \frak n_k(\frak n_0(y);{\bf x})
+ (-1)^{\deg' y}\frak n_*(y;\hat d({\bf x}))
= 0.
\endaligned
\end{equation}
Here we use notation (\ref{Swedler}).
$\hat d$ is the co-derivation induced by  
$A_{\infty}$ operations.
Note that in this paper we do not include $B_0\mathscr C(a,b)$
in $B\mathscr C(a,b)$
so we include the second and the third terms. 
In the convention of several other papers 
they are spacial cases of the first term.
\par
If $c$ is an object of $\frak F(\mathbb L)$
the Yoneda embedding associates a
right $\frak F(\mathbb L)$ module 
$\frak Y(c)$ that is
\begin{equation}\label{yonedadef}
\left\{
\aligned
&(\frak Y(c))(a): = CF(c,a;\Lambda_0) \\
&\frak n_k(y;x_1,\dots,x_k)
:= \frak m_{k+1}(y,x_1,\dots,x_k)
\endaligned
\right.
\end{equation}
(We remark that the operations $\frak m_{k+1}$ in this formula already contains the correction 
by the bounding cochains.)
(\ref{rightmoddef}) is a consequence of the $A_{\infty}$ relations of $\frak m_k$.
\par
Now we assume that $(L,b)$ and $(\varphi(L),\varphi_*(b))$ 
are objects of $\frak F(\mathbb L)$.
We review the construction of a right module homomorphism
from $\frak Y(L,b)$ to $\frak Y(\varphi(L),\varphi_*(b))$
which was given in \cite[Section 15.5]{FuFu6} 
based on \cite[Subsection 5.3.1]{fooobook}.
\par
We first review the construction of the bounding cochain 
$\varphi_*(b)$.
(The discussion below is sketchy since it is a review 
of the construction of \cite[Section 15.5]{FuFu6} 
and \cite[Subsection 4.6.1]{fooobook}.
See those papers for  detail.)
Let $J$ be the almost complex structure we use 
to define our filtered $A_{\infty}$ category $\frak F(\mathbb L)$.
We take time dependent Hamiltonian $H : X \times [0,1] \to \R$.  
Let $\varphi^{t} : X \to X$ be the family 
of Hamiltonian diffeomorphisms such that
$\varphi^0 = {\rm id}$ and $\frac{d}{d{t}}\varphi^{t} = \frak X_{H_{t}}\circ \varphi^{t}$.
Here $\frak X_{H_{t}}$ is the Hamiltonian vector field associated to 
$H_{t}(x): = H(x,t)$.
We assume $\varphi^1 = \varphi$. We may assume that $H_{t}$ is a constant 
function if $t$ is in a small neighborhood of $0$ or of $1$.
\par
We consider a one parameter family of almost complex structures 
$J^{(\rho)}$, such that $J^{(0)} = J$ and $J^{(1)} = (\varphi^{-1})_*J$.
(Later we specify the choice.)
We observe that moduli spaces of $J^{(1)}$-holomorphic disks
bounding $L$ is canonically identified with  moduli spaces of $J$-holomorphic disks
bounding $\varphi(L)$.
\par
We consider the moduli spaces of 
pseudo-holomorphic maps $u$ bounding $L$ from the tree like union $\Sigma$
of disks with $k+1$ boundary marked points 
$(z_0,\dots,z_k)$
together with tree like unions of sphere components 
rooted at  interior points of disk components.  The map $u$ is required to be $J^{(\rho_{\alpha_i})}$ holomorphic 
on a disk component $D^2_{\alpha_i}$ of $\Sigma$ where $\rho_{\alpha_i}$ depends on the 
component. We require $\rho_{\alpha_i} \ge \rho_{\alpha_j}$ if $D^2_{\alpha_i}$ 
is `closer' to the $0$-th marked point $z_0$ than $D^2_{\alpha_j}$.
(See Figure \ref{fig19FOOO}.)
\begin{figure}[h]
\centering
\includegraphics[scale=0.8]{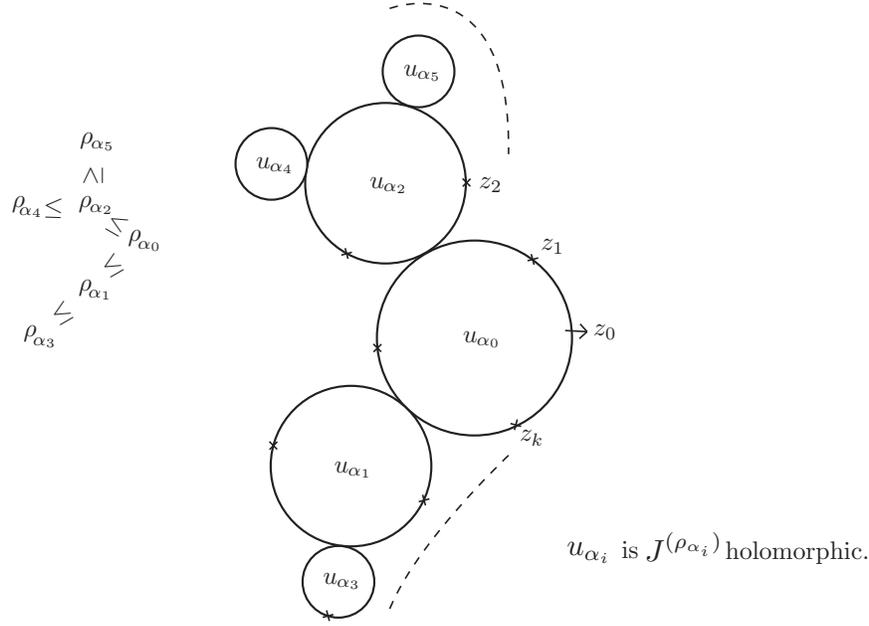}
\caption{Time ordered product moduli space.}
\label{fig19FOOO}
\end{figure}
When there is a tree of sphere components, we require that $u$ is 
$J^{(\rho_{\alpha_i})}$ holomorphic there, if the root of this tree 
of sphere components is at $D^2_{\alpha_i}$.
\par
As we mentioned already the moduli spaces of $J^{(1)}$ holomorphic disks
which bounds $L$ 
is canonically identified with   
the moduli spaces of $J$ holomorphic disks
which bounds $\varphi(L)$.
We use Kuranishi structures and CF-perturbations of the 
latter when we defined  $\frak F(\mathbb L)$.
We use exactly the same  Kuranishi structures and CF-perturbations
for the moduli spaces of $J^{(1)}$ holomorphic disks
which bound $L$. 
The Kuranishi structures and CF-perturbations 
of $J^{(0)}$ (= $J$) holomorphic disks
which bounds $L$ is also given while we defined  $\frak F(\mathbb L)$.
We extend them to  Kuranishi structures and CF-perturbations 
of the moduli space of maps $u$ as in Figure \ref{fig19FOOO}.
It defines a map
$
\frak g_k: B_k CF(L,L) \to CF(\varphi(L),\varphi(L)).
$
By studying its boundary we can find that it is a filtered $A_{\infty}$
homomorphism between two (curved) filtered $A_{\infty}$ algebras 
$CF(L,L)$ and $CF(\varphi(L),\varphi(L))$.
(Their curved filtered $A_{\infty}$ structures are defined 
during the construction of $\frak F(\mathbb L)$ using the almost 
complex structure $J$.)
Now we put
$
\varphi_*(b) := \sum_{k=0}^{\infty} \frak g_k(b,\dots,b).
$
\par
We next review the construction of 
the right filtered $A_{\infty}$ module homomorphism
$t^1_{c,c'}: \frak Y(L,b)\to \frak Y(\varphi(L),\varphi_*(b))$.
(Here $c = \frak Y(L,b)$ and $c' =\frak Y(\varphi(L),\varphi_*(b))$.)
Let $M_0,\dots,M_k \in \mathbb L$ and $b_i$ be bounding cochain of $M_i$.
We will define:
\begin{equation}
\phi_k : CF(L,M_0) \otimes CF(M_0,M_1) \otimes \dots 
\otimes CF(M_{k-1},M_k) \to CF(\varphi(L),M_k).
\end{equation}
Here and hereafter we write $M_i$ instead of $(M_i,b_i)$.
\par
To define $\phi_k$ we use moduli spaces 
which we describe below. See \cite[Definition 15.26]{FuFu6}
for detail.
Let $\chi: \R \to [0,1]$ be a non-decreasing smooth function 
such that $\chi(\tau) = 0$ for sufficiently small $\tau$ 
and $\chi(\tau) = 1$ for sufficiently large $\tau$.
We take a $(\tau,t) \in \R \times [0,1]$ dependent family 
of almost complex structures $J_{\tau,t}$ with the following 
properties. 
\begin{enumerate}
\item
There exists $A>0$ such that $J_{\tau,t} =
J$
if $\tau< -A$.
\item
\begin{equation}
J_{\tau,t} = (\varphi^{t})^{-1}_* J
\end{equation}
if $\tau > + A$.
\item 
We dentote by $\varphi_{c H}^{\rho}$ the one parameter family 
of Hamilton diffeomorphisms generated by
the time dependent Hamiltonian $cH : X \times [0,1] \to \R$.
Then
\begin{equation}
J_{\tau,1} = (\varphi_{\chi(\tau) H}^{1})^{-1}_* J
\end{equation}
if $\tau > 0$.
\item
$J_{\tau,t} = J$ for any $\tau$ and for $t$ sufficiently close to $0$.
\end{enumerate}
We denote this source dependent (that is, $(\tau,t)$ dependent)  family 
of almost complex structures by $\mathcal J^{1;c,c'}$.
\par
Note that we used a one parameter family of almost complex structures 
$J^{(\rho)}$ to obtain $\varphi_*(b)$. We specify this choice 
as follows. We take an order preserving diffeomorphism
$
\theta :  (0,1) \to \R
$
and put
$
J^{(\rho)} = (\varphi_{\chi(\theta(\rho)) H}^{1})^{-1}_*J.
$
Now for a map $u : \R \times [0,1] \to X$ 
we consider the equation 
\begin{equation}\label{eq1520}
\frac{\partial u}{\partial \tau} + J_{\tau,t}
\left(\frac{\partial u}{\partial t}
- \chi(\tau) \frak X_{H_t}\right) = 0.
\end{equation}
We denote the $(\tau,t)$-dependent Hamilton vector field 
$\chi(\tau) \frak X_{H_t}$ by $\frak X^{1;c,c'}$
and put $\mathcal X^{1;c,c'}_{\tau,t}:  = \chi(\tau) \frak X_{H_t}$.
\par
The boundary condition at $t =1$ is
\begin{equation}
u(\tau,1) \in L.
\end{equation}
The boundary condition at $t=0$ is as follows.
We require that there are $k$ boundary marked points 
$z_{0,1}=(\tau_{0,1},0), \dots, z_{0,k}=(\tau_{0,k},0)$
with $\tau_{0,1} < \dots < \tau_{0,k}$.
We then require
\begin{equation}
u(\tau,t) \in M_i
\end{equation}
if $\tau \in (\tau_{0,i},\tau_{0,i+1})$.
(Here we put $\tau_{0,-1} = -\infty$, $\tau_{0,k+1} = \infty$ 
as convention.)
\par
We consider the configuration of such $u$ and several 
of the configurations as in Figure \ref{fig19FOOO} 
attached at $t=1$. 
If a disk $D^2_{\alpha}$ is attached at $(\tau,1)$ 
the map $u_{\alpha}$ on $D^2_{\alpha}$ 
is required to be $J^{(\rho)}$ holomorphic 
with $\theta(\rho) \le \tau$. See Figure \ref{Figure15-23} below.
\begin{figure}[h]
\centering
\includegraphics[scale=0.48]{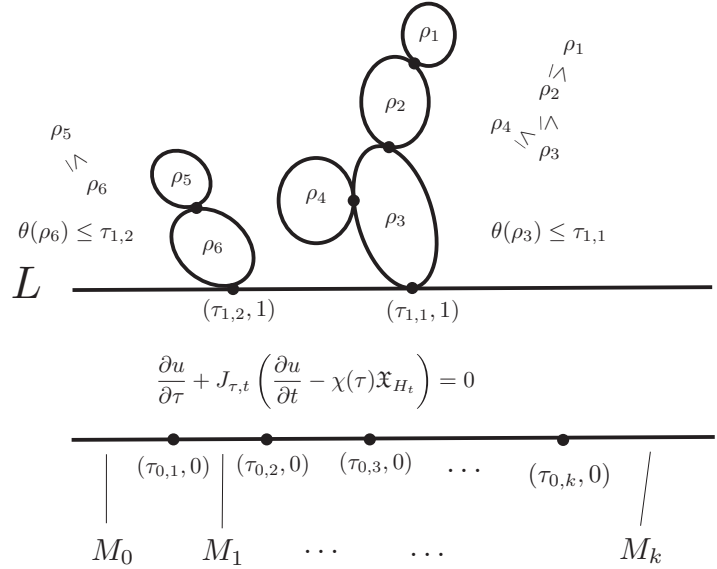}
\caption{An element of the moduli space used to define $\phi_k$.}
\label{Figure15-23}
\end{figure}
We use the compactifications of the moduli spaces 
of such configurations and their CF-perturbations to 
define $\phi_k$. 
Actually we put more marked points on $t=0$ and $t=1$ (and 
the boundaries of trees of disks attached to $t=1$)
and put bounding cochains (of $M_i$ or of $L$) there.
\par
The condition that $\{\phi_k\}$ defines a right filtered $A_{\infty}$
module homomorphism is written as the next formula.
\begin{equation}\label{form1611}
\aligned
&(-1)^{\deg'y}\phi_*(y;\hat d({\bf x}))
+ \sum_{i \in I({\bf x},2)}\phi_*(
\frak n(y;{\bf x}_{i;1}),{\bf x}_{i;2}) \\
&= \sum_{i \in I({\bf x},2)}\frak n_*(\phi_*(y;{\bf x}_{i;1});{\bf x}_{i;2}).
\endaligned
\end{equation}
Here we use notation (\ref{Swedler}).
As usual, this formula can be checked by studying the 
boundary of our moduli spaces (and using Stokes' formula 
\cite[Theorem 9.28]{fooo:springerbook} and  composition formula \cite[Theorem 10.21]{fooo:springerbook}).
The first term of the left hand side corresponds 
to the boundary where a polygon bubbles at the line $t=0$.
The second term of the left hand side corresponds
to the boundary where a strip escapes to $\tau \to -\infty$.
The right hand side corresponds
to the boundary where a strip escapes to $\tau \to +\infty$.
All the contributions from other boundaries cancel out 
by using bounding cochains attached.
See \cite[Figure 91]{FuFu6}.
\par
The estimate of the energy loss of $\phi_k$ 
follows from the inequality:
\begin{equation}\label{1612form}
\int_{\R \times [0,1]} u^* \omega
+ \lim_{\tau\to+\infty}\int_{[0,1]} H(t,u(\tau,t)) dt
\ge -\int_{0}^1 \inf(H_t) dt
\end{equation}
for an element $u$ of our moduli space.
This is an estimate which goes back to 
Chekanov \cite{chekanov} and was used in \cite{fooobook,fooo:bidisk}.
See \cite[Lemma 15.29]{FuFu6} for its proof.
\par
The upshot of this construction is 
that we can use $H_t$ to 
obtain $\mathcal J^{1;c,c'}$, $\frak X^{1;c,c'}_{\tau,t}$ 
and construct $t^1_{c',c}= \{\phi_k\}$.
Hereafter we write $\phi^1_{c,c';k}$ in place of $\phi_k$.
\par
The right filtered $A_{\infty}$ homomorphism
$t^1_{c,c'}: \frak Y(\varphi(L),\varphi_*(b))\to \frak Y(L,b)$
can be constructed by using $H'(x,t) = -H(x,1-t)$.
Namely 
$J^{1;c',c}_{\tau,t} = J^{1;c',c}_{-\tau,t}$
and $\frak X^{1;c',c}_{\tau,t} = - \frak X^{1;c',c}_{\tau,1-t}$.
\begin{rem}
There is a slight asymmetry between $L$ and $\varphi(L)$.
If we want them to be completely symmetric we take 
$L'$ and put $L = \varphi_0(L')$ and $\varphi(L) = \varphi_1(L')$.
\end{rem}
\par
We next discuss the construction of the cochain homotopy $s^1_{c}$
from $t^1_{c,c'}\circ t^1_{c',c}$ to the identity.
\par
We consider $A>0$ such that $J_{\tau,t} =
J$
if $\tau< -A$.
For $R > A$, we define a one parameter family of 
source dependent almost complex structures 
$J^{1;c}_{R;\tau,t}$
by
\begin{equation}\label{famiJeq}
J^{1;c}_{R;\tau,t}
= 
\begin{cases}
\aligned
J^{1;c,c'}_{R+\tau,t} \qquad &\tau <A-R \\
J^{1;c',c}_{-R+\tau,t} \qquad &\tau >R-A \\
(\varphi^{t})^{-1}_* J \qquad &A-R \le \tau \le R-A,
\endaligned
\end{cases}
\end{equation}
and a one parameter family of 
source dependent Hamiltonian vector fields 
$\frak X^{1;c}_{R;\tau,t}$ by
\begin{equation}\label{famiXeq}
\frak X^{1;c}_{R;\tau,t}
= 
\begin{cases}
\aligned
\frak X^{1;c,c'}_{R+\tau,t} \qquad &\tau <A-R \\
\frak X^{1;c',c}_{-R+\tau,t} \qquad &\tau >R-A \\
\frak X_{H_t}  \qquad & A-R \le \tau \le R-A.
\endaligned
\end{cases}
\end{equation}
We extend the families $J^{1;c}_{R;\tau,t}$ and $\frak X^{1;c}_{R;\tau,t}$
to $R \in [0,A]$ so that 
$J^{1;c}_{0;\tau,t} = J$ and $\frak X^{1;c}_{0;\tau,t} = 0$
and obtain $\mathcal J^{1;c}$ and $\frak X^{1;c}$.
\par
We change  (\ref{eq1520})   to
\begin{equation}\label{eq15202}
\frac{\partial u}{\partial \tau} + J^{1;c}_{R;\tau,t}
\left(\frac{\partial u}{\partial t}
- \frak X^{1;c}_{R,\tau,t}\right) = 0.
\end{equation}
Thus in the same way we obtain a
totality of the pair $(R,u)$ of $R \in [0,\infty)$ and 
a solution $u$ of (\ref{eq15202}) together with 
the configurations as in Figure \ref{fig19FOOO} 
attached at $t=1$.
\par
This system of moduli spaces has a system of Kuranishi structures 
and CF-perturbations which are compatible with 
the previously defined ones at the boundary.
We thus obtain:
\begin{equation}
\aligned
\psi^1_{c;k} : CF(L,M_0) &\otimes CF(M_0,M_1) 
\\
&\otimes \dots 
\otimes CF(M_{k-1},M_k) 
\to CF(L,M_k).
\endaligned
\end{equation}
We put $s^1_{c} = \{\psi^1_{c;k}\}$.
The system of $\Lambda_0$ module homomorphisms $s^1_{c}$
can be regarded as a pre-natural transformation 
between two right filtered $A_{\infty}$ modules (that is identified as 
filtered $A_{\infty}$ functors).
It is not a cocycle, that is, $s^1_{c}$ is not a 
right module homomorphism. 
The boundary $\frak M_1(s^1_{c})$ is defined by the 
next formula.
\begin{equation}
\aligned
\frak M_1(s^1_{c})_k({\bf x})
=
&\sum_{i \in I({\bf x},2)}\psi^1_{c;*}(
\frak n(y;{\bf x}_{i;1}),{\bf x}_{i;2})\\
&- \sum_{i \in I({\bf x},2)}\frak n_*(\psi^1_{c;*}(y;{\bf x}_{i;1});{\bf x}_{i;2})
+ (-1)^{\deg'y}\psi^1_{c,c;*}(y;\hat d({\bf x})).
\endaligned
\end{equation}
Here we use notation (\ref{Swedler}).
As in the case of (\ref{form1611}),
the three terms of the right hand sides 
corresponds to the three types of the boundary 
of the moduli space of the solutions of (\ref{eq15202}).
The first term corresponds
to the boundary where a strip escapes to $\tau \to -\infty$.
The second term corresponds
to the boundary where a strip escapes to $\tau \to +\infty$.
The third term corresponds 
to the boundary where a polygon bubbles at the line $t=0$.
Many of the other boundary components are cancel out 
by using bounding cochains.
It remains  two types of boundary components which correspond
to $R \to \infty$ and to $R \to 0$, respectively.
We thus obtain the required formula
$
\frak M_1(s^1_{c})+ \frak M_2(t^1_{c,c'},t^1_{c',c})
= {\rm id}.
$
The review of the proof of Theorem \ref{estimateHofLag}
is over. We will continue to prove
Theorem \ref{estimateHofLaginf2}.
\par
The next step is to obtain a pre-natural transformation 
$t^2_{c,c'} = \{\phi^2_{c,c';k}\}$, which satisfies:
\begin{equation}\label{eqtobesolved}
\frak M_1(t^2_{c,c'}) + \frak M_2(t^1_{c,c'},s^1_{c'})
+ \frak M_2(s^1_{c},t^1_{c,c'})
= 0.
\end{equation}
(See (\ref{form128}). Here we use the sign convention of 
$A_{\infty}$ structure. In (\ref{form128}) the sign 
convention of DG-categories is used.)
We remark that the second and third terms are 
obtained from a one parameter family of source dependent 
almost complex structures and Hamiltonian vector fields 
which is obtained by `composing' those 
that were used to obtain $t^1_{c,c'},s^1_{c'},s^1_{c},t^1_{c,c'}$.
Those are depicted  by Figure \ref{Figuret2}. 
\begin{figure}[h]
\centering
\includegraphics[scale=0.48]{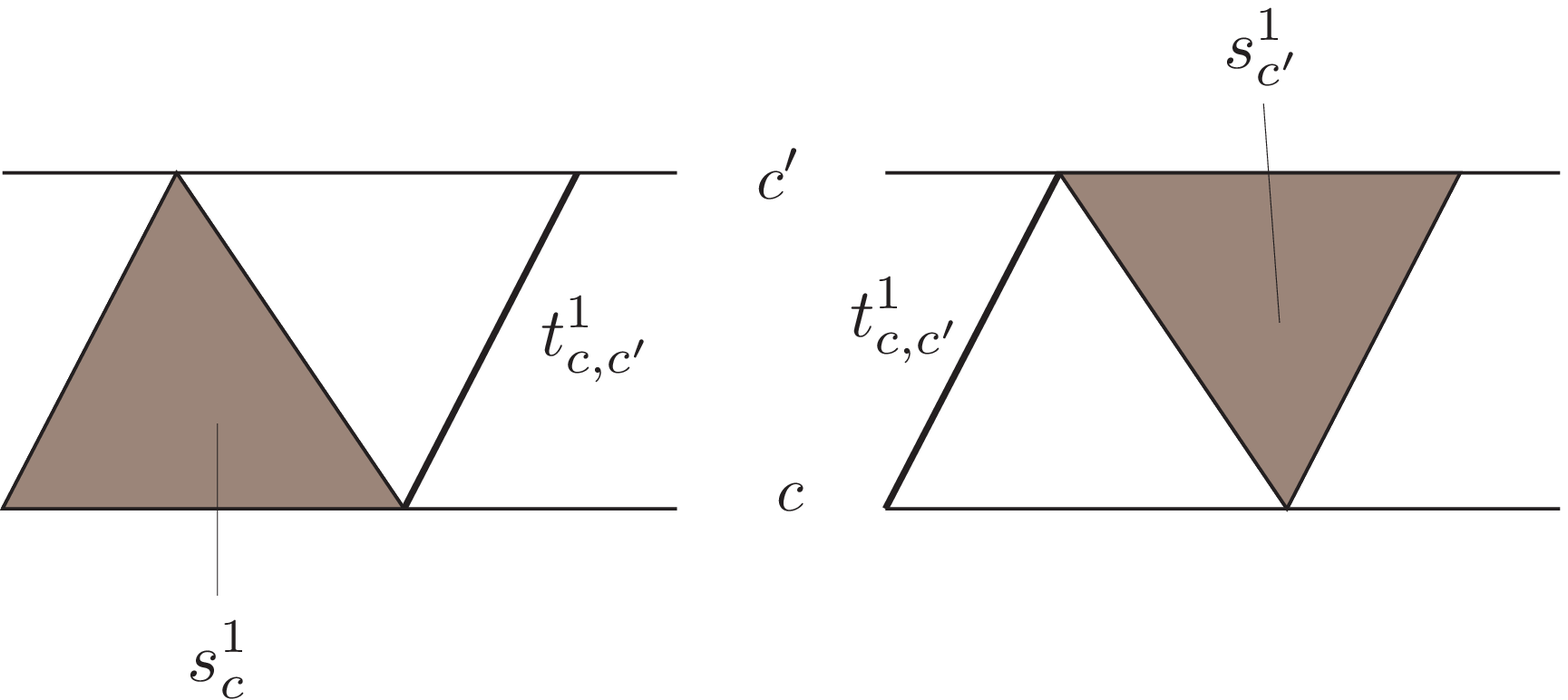}
\caption{Family of piecewise linear functions.}
\label{Figuret2}
\end{figure}
The gray regions of the figure show a one parameter families 
of piecewise linear functions as in Figure \ref{1parafamilypl},
which can be converted to a one parameter 
family of source dependent almost complex structures 
and Hamiltonian vector fields as we did during the construction of $s^1_{c}$.
\begin{figure}[h]
\centering
\includegraphics[scale=0.48]{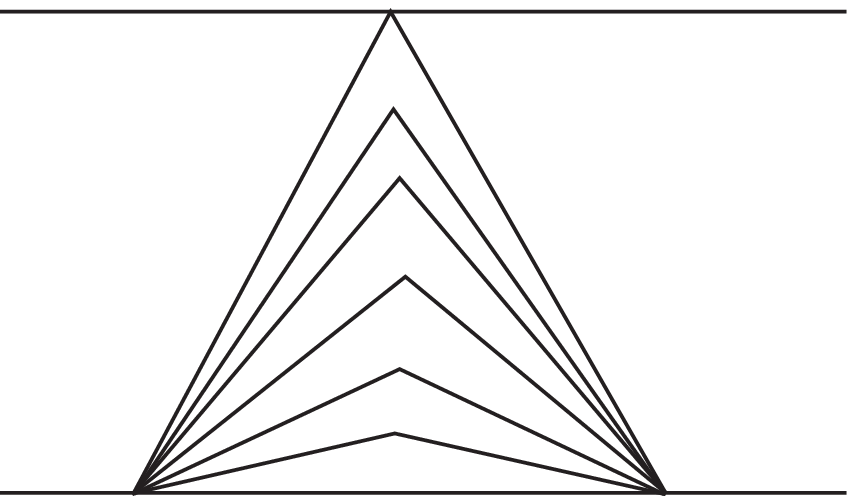}
\caption{Family of piecewise linear functions.}
\label{1parafamilypl}
\end{figure}
We join this family of piecewise linear functions
with an afine function representing $t^1_{c,c'}$.
Thus the third term of (\ref{eqtobesolved}) 
corresponds to the left of Figure 
\ref{Figuret2} and 
the second term of (\ref{eqtobesolved}) 
corresponds to the right of Figure \ref{Figuret2}.
The boundary of the two one parameter families 
intersect at one point which represents 
$t^1_{c,c'} \circ t^1_{c',c} \circ t^1_{c,c'}$.
The other ends correspond 
to ${\bf e}_c \circ t^1_{c,c'}$ 
and $t^1_{c,c'} \circ {\bf e}_{c'}$, respectively.
As (pre)natural transformations they 
coincide with each other, but, as piecewise linear functions 
or source dependent almost complex structures and 
Hamiltoniam vector fileds, they are different.
We join them by a one parameter family of 
piecewise linear functions depicted by Figure \ref{Fantoncell1}.
We call this (and similar) family as a phantom cell.
\begin{figure}[h]
\centering
\includegraphics[scale=0.48]{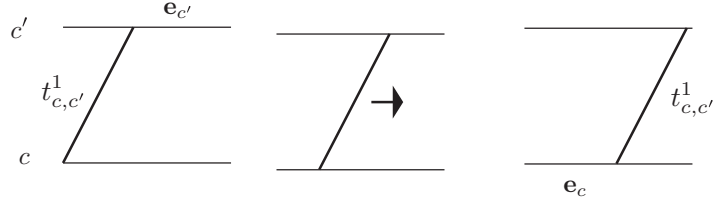}
\caption{A phantom cell.}
\label{Fantoncell1}
\end{figure}
We thus obtain a cell decomposition of $S^1$
into three 1-cells which parametrize piecewise linear functions.
We can find a 2-cell $\Delta^2$ and  a family of piecewise linear functions 
parametrized by $\Delta^2$ so that on the boundary $\partial \Delta^2 = 
S^1$ it becomes the 
above described $S^1$ parametrized family.
In fact we can take $s f + (1-s) f_0$ 
where $f_0$ is the unique affine map so that 
$f_0(0) = 0$ and $f_0(1) = 1$ at $1$, for example.
It gives a map from $\partial \Delta^2 \times [0,1]$ to 
the space of piecewise linear functions.
At $s=1$ the  function is independent  of $\partial \Delta^2$
factor. So it induces a map from the cone of $\partial \Delta^2$ 
to the space of piecewise linear functions.
\par
In the same way as (\ref{famiJeq}) and (\ref{famiXeq}),
from the family of piecewise linear functions
we can obtain a $\Delta^2$ parametrized family of source dependent families of 
almost complex structures and Hamiltonian vector fields.
Then for each point in the interior ${\rm Int}{\Delta^2}$ of $\Delta^2$
we obtain an equation similar to (\ref{eq15202})
where $R$ there is replaced by a point of ${\rm Int}{\Delta^2}$.
At the boundary $\partial{\Delta^2}$ the moduli space of its solutions 
becomes the (fiber) products of the moduli spaces 
which was used to define $t^1_{c,c'}$, $t^1_{c',c}$, $s^1_{c}$,
$s^1_{c'}$.
These moduli spaces have Kuranishi structures and 
CF-perturbations which are compatible with those 
used to define $t^1_{c,c'}$, $t^1_{c',c}$, $s^1_{c}$,
$s^1_{c'}$.
Therefore we  use those 
Kuranishi structures and 
CF-perturbations to define
\begin{equation}\label{mapfor619}
\aligned
\phi^2_{c,c';k} : CF(L,M_0) &\otimes CF(M_0,M_1) \\
&\otimes \dots 
\otimes CF(M_{k-1},M_k) \to CF(\varphi(L),M_k).
\endaligned
\end{equation}
This is the definition of $t^2_{c,c'}$.
As in the case of $s^1_{c}$ the moduli space 
we used to define $t^2_{c,c'}$ has 
4 types of boundaries.
Three of them correspond to $\frak M_1(t^2_{c,c'})$.
The last one is obtained 
from the $S^1 = \partial\Delta^2$ parametrized family.
Two among three 1-cells of  $\partial\Delta^2$ give 
$\frak M_2(t^1_{c,c'},s^1_{c'})$
and  $\frak M_2(s^1_{c},t^1_{c,c'})$,
respectively,
by construction.
We claim that the phantom cells do not contribute 
to the operation.
In fact, the moduli spaces of the solutions of the 
equation do not change if we move on a phantom 
1-cell. We can take our CF-perturbations 
so that this symmetry is preserved.
Then the integration along the fiber is zero 
in the chain level. (This argument is the same 
as the proof of strict unitality of the degree-$0$ form $1$.)
Thus we verified (\ref{eqtobesolved}).
The estimate of the energy loss is the same 
as (\ref{1612form}).
\par
To clarify the construction we will 
discuss one more step of the construction
explicitly.
We will construct $s^2_{c} = \{\psi^2_{c;k}\}$, that satisfies:
\begin{equation}\label{eqtobesolved2}
\frak M_1(s^2_{c}) + \frak M_2(t^2_{c,c'},t^1_{c',c})
+ \frak M_2(t^1_{c,c'},t^2_{c',c}) + \frak M_2(s^1_{c},s^1_{c})
= 0.
\end{equation}
The second, third and fourth terms of the 
left hand side is depicted by Figure \ref{Figures2} which is similar to 
Figure \ref{Figuret2}.

\begin{figure}[h]
\centering
\includegraphics[scale=0.45]{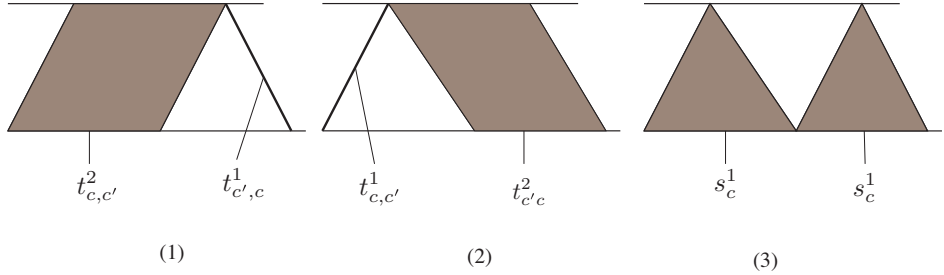}
\caption{2,3 and 4th terms of (\ref{eqtobesolved2}).}
\label{Figures2}
\end{figure}

(1)(2)(3) of Figure \ref{Figures2} correspond 2-cells parametrizing 
piecewise linear functions. They are glued to give a 
two dimensional cell complex depicted by Figure \ref{Figures2poly} below.

\begin{figure}[h]
\centering
\includegraphics[scale=0.45]{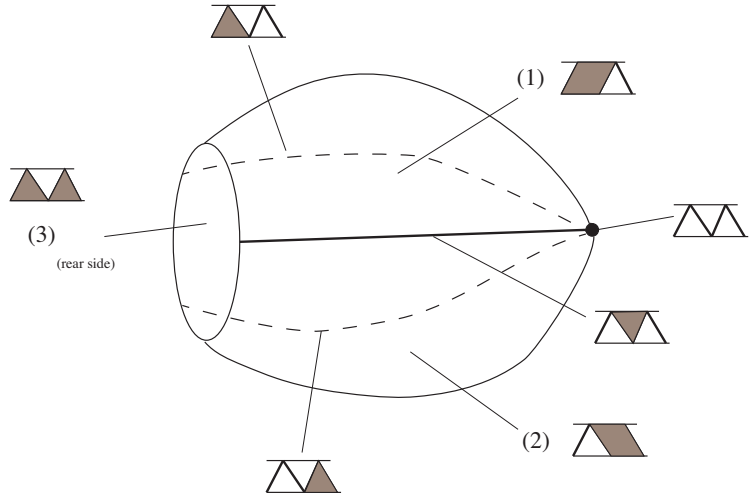}
\caption{A cell complex describing (\ref{eqtobesolved2}).}
\label{Figures2poly}
\end{figure}

The 2-dimensional cell complex in Figure \ref{Figures2poly} 
is homeomorphic to a 2-cell and its boundary is a 
1-dimensional cell complex homeomorphic to a circle 
and its 1-cells are phantom cells. We depict it in Figure \ref{Phantom2}
below.
(The arrows in the figure show how a part of the graph moves 
when we move on the circle in the counter clockwise direction.) 
\begin{figure}[h]
\centering
\includegraphics[scale=0.45]{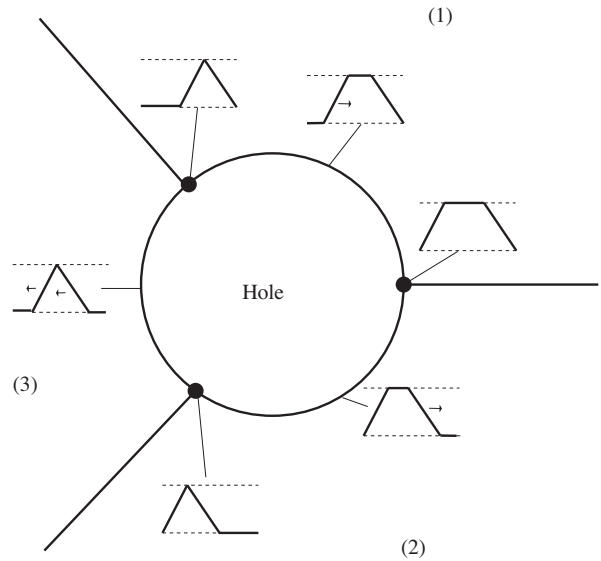}
\caption{A circle of phantom cells.}
\label{Phantom2}
\end{figure}
\par
We can fill this hole and construct a 2-cycle parametrizing a 
piecewise linear functions as follows.
\par
We observe that a piecewise linear function corresponding to a point 
of a phantom cell has a part (a union of intervals) where the function 
is constant, (that is $0$ or $1$). Suppose those intervals are 
$[a_1,b_1]$, \dots, $[a_m,b_m]$ with $b_i < a_{i+1}$.  
(In our particular situation there are at most 
two such intervals. We describe this process 
in such a way that we can use it in more general 
situation also.)
At time $s \in [0,1]$ we shrink the length of each 
such interval by the factor $(1-s)$.
Namely the first interval becomes $[a_1,b_1 - s(b_1-a_1)]$
The second interval becomes $[a_2 - s(b_1-a_1),b_2 - s(b_1-a_1)-
s(b_2-a_2)]$ and etc.
The exception is the case when $b_m = 1$. In that case we do not change 
$b_m$ but change $a_m$ to $a_m - \sum_{i=1}^{m-1} s(b_i-a_i)$.
We move the part where the piecewise linear functions are 
non-constant to the left.
See Figure \ref{fillhole}.
\begin{figure}[h]
\centering
\includegraphics[scale=0.45]{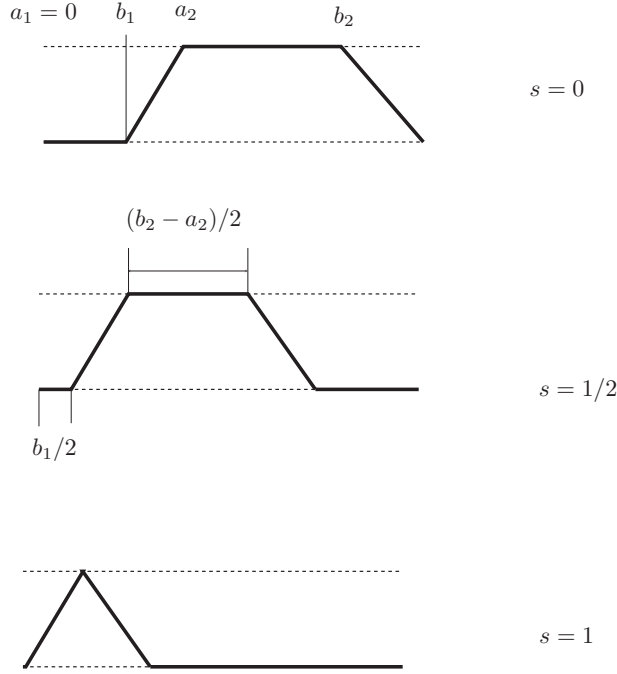}
\caption{Fill the hole.}
\label{fillhole}
\end{figure}
At $s=1$ certain parameters of a phantom cell become trivial 
and so the dimension of the cell complex strictly decreases.
Therefore we obtain a cell complex of dimension 2 (in our case) 
such that its boundary is $S^1$.
Note that the complex we use to fill the hole consists of phantom cells.
\par
We thus obtain a 2-dimensional cell complex which parametrizes  
piecewise linear functions and which is a {\it cycle}.
\par
Therefore we can find a 3-dimensional cell complex $C_3$
which parametrizes piecewise linear functions and whose boundary 
is the above 2 cycle. 
Thus we obtain a $C_{3}$ parametrized family of source dependent 
families of almost complex structures and of Hamiltonian 
vector fields. Then we obtain a moduli space of solutions of
an equation similar to (\ref{eq15202}) where 
$R$ is replaced by an element of $C_3$. 
\par
We extend Kuranishi structures and CF-perturbations 
to these moduli spaces.
Then we obtain  maps $\phi^2_{c;k}$ 
as in (\ref{mapfor619}).
We thus obtained $s^2_{c} = \{\psi^2_{c;k}\}$.
Using the fact that the boundary of $C_3$
is the union of 2-cells as in Figure \ref{Figures2poly}
and phantom cells we can prove (\ref{eqtobesolved2})
by Stokes' theorem (\cite[Theorem 9.28]{fooo:springerbook}), 
composition formula (\cite[Theorem 10.21]{fooo:springerbook})
and the fact that phantom cells do not contribute to the 
map (integration along the fiber) in the chain level.
\par
The construction of $s^k_{c}$, $s^k_{c'}$ 
and $t^k_{c,c'}$, $t^k_{c',c}$
for general $k$ is mostly the same as the case we already 
described. 
Let us discuss the case of $t^k_{c,c'}$.
The equation $t^k_{c,c'}$ to satisfy is
$$
\frak M_1(t^k_{c,c'}) + 
\sum_{\ell=1}^{k-1} \frak M_2(s^{\ell}_{c},t^{k-\ell}_{c,c'})
+ \sum_{\ell=1}^{k-1} \frak M_2(t^{k-\ell}_{c,c'},s^{\ell}_{c'}) = 0
$$
See (\ref{form128}).
By induction hypothesis, we have constructed cells of 
piecewise linear functions, 
corresponding to $s^{\ell}_{c},t^{k-\ell}_{c,c'},t^{k-\ell}_{c,c'},s^{\ell}_{c'}$.
We observe that those cells are glued 
together to give a cell complex $C_0$ whose boundary 
consists of phantom cells.
This is a consequence of 
$$
\frak M_1
\left(
\sum_{\ell=1}^{k-1} \frak M_2(s^{\ell}_{c},t^{k-\ell}_{c,c'})
+ \sum_{\ell=1}^{k-1} \frak M_2(t^{k-\ell}_{c,c'},s^{\ell}_{c'})
\right) = 0
$$
that follows from the induction hypothesis.
Therefore in the same way as we did in Figure \ref{fillhole}
we can shrink the length of the `flat' parts of the 
graph of piecewise linear functions and move the other parts to the 
left to obtain a cell complex $D$ such that 
$D$ consists of phantom cells and $\partial D = \partial C_0$.
Thus we glue $C_0$ and $-D$ along the boundary to obtain 
a cycle
$C_1$. Therefore we obtain a 
cell complex $C$ parametrizing piecewise linear functions 
such that $\partial C = C_1 = C_0 \cup -D.$
Therefore using $C$ parametrized family of source dependent 
almost complex structures and Hamiltonian vector fields 
we define moduli spaces. They have Kuranishi structures and 
CF-perturbations which are compatible with the previously 
given one.
We thus obtain  maps which define $t^k_{c,c'}$.
\par
Here we need to remark the following.
During the above construction we need to extend the CF-perturbation 
given at the boundary to the interior.
Such a theory is developed in detail in \cite[Chapter 17]{fooo:springerbook}.
The case of a Kuranishi structure with corners is studied there.
We defined a moduli space of $C$ parametrized equations.
The parameter space $C$ is a cell complex and is not necessary 
a manifold with corners.
To apply the theory of \cite[Chapter 17]{fooo:springerbook}
to our case we proceed as follows.
From the construction it is obvious that the cell complex $C$ 
we use has a triangulation so that the triangulations 
of various cell complexes (used to define  $s^{k}_{c},t^{k}_{c,c'},t^{k}_{c,c'},s^{k}_{c'}$) are compatible at the boundary.
In other words the triangulation of the boundary of the cell complex 
to define $s^k_{c}$,$t^k_{c,c'}$ etc. is a subdivision of the product of those 
used to define $s^{\ell}_{c}$,$t^{\ell}_{c,c'}$ etc.
We remark that a simplex is a manifold with corners.
So when we restrict our moduli space to the part corresponding to 
a simplex of the triangulation of $C$,
we obtain a Kuranishi structure with corners.
Therefore applying \cite[Chapter 17]{fooo:springerbook} 
to each simplex inductively from those adjoining the boundary 
to those in the interior we can construct 
CF-perturbations on the parts of the moduli spaces 
corresponding to various simplexes so that they coincide at the 
overlapped part.
\par
The last point to mention is `homotopy inductive limit argument'.
Again we need to stop the argument at the stage 
when only a finite number of moduli spaces are involved.
We can resolve this issue in the following way as usual.
Observe that the equation we want to solve is 
Maurer-Cartan equation. Therefore we can
define the notion of a gauge equivalence of infinite homotopy 
equivalences in the same way as the gauge equivalence of 
bounding cochains (See \cite[Section 4.3]{fooobook}).
We can also define the notion of an $(m,K)$-homotopy equivalence.
Namely an $(m,K)$-homotopy equivalence assigns
$t^k_{c,c'}$, $s^k_{c}$ etc. only for $k \le m$
and they satisfy Maurer-Cartan equation 
modulo $T^{K}$.
The argument we have described gives such $(m,K)$-homotopy equivalence
for any but finite $m$ and $K$, which we write  $(\hat t(m,K),\hat s(m,K))$.
Then we can show for $m < m'$, $K < K'$ 
$(\hat t(m,K),\hat s(m,K))$ is gauge equivalent to $(\hat t(m',K'),\hat s(m',K'))$
as $(m,K)$-homotopy equivalences.
Then by a homological algebra we can show that 
$(\hat t(m,K),\hat s(m,K))$ can be extended to an $(m',K')$-homotopy equivalence.
Therefore by an induction argument 
we can construct an infinite homotopy equivalence.
\par
The proof of Theorem \ref{estimateHofLaginf2} is now complete.
\qed

\section{Completions of filtered $A_{\infty}$ categories.}
\label{sec;more}

In Theorem \ref{Lagcatecomplete} we study Cauchy sequences of 
finite sets 
of Lagrangian submanifolds.
Actually the argument of its proof works in more general case, 
which we discuss here.
\begin{defn} 
Let $(X,\omega)$ be a symplectic manifold 
which is compact or tame.
We define a metric space $\overline{\frak L}$ 
as follows.
Its element is an equivalence class of a sequence $\{L^n\}$
such that 
\begin{enumerate}
\item
$L^n$ is a relatively spin Lagrangian submanifold of $X$.
We assume $L^n$ is unobstructed.
\item
$\sum_{n=1}^{\infty} d_{\rm Hof}(L^n,L^{n+1}) < \infty.$
\end{enumerate}
Two sequences $\{L_i^n\}$, $i=1,2$ are said to be equivalent 
if 
$$
\lim_{n\to \infty} d_{\rm Hof}(L_1^n,L_2^{n}) = 0.
$$
We define the Hofer metric $d_{\rm Hof}([\{L^n_1\}],[\{L^n_2\}])$
between two equivalence classes by
$$
d_{\rm Hof}([\{L^n_1\}],[\{L^n_2\}]) = \lim_{n\to \infty} d_{\rm Hof}(L_1^n,L_2^{n}).
$$
Clearly $\overline{\frak L}$ is a complete metric space.
\end{defn}
\begin{thm}\label{sepaobj}
For any separable subset $\mathbb L$ of $\overline{\frak L}$
we can associate a unital completed DG-category 
$\frak F(\mathbb L)$ with the following properties.
\begin{enumerate}
\item The set of objects of $\frak F(\mathbb L)$ is an equivalence classes of pairs 
$([\{L^n\}],[(b^n)])$ where $\{L^n\}$ is a Cauchy sequence converging to an element 
of $\mathbb L$  and $b^n$ is a sequence of bounding cochains of $L^n$ 
which has certain compatibility conditions explained during the proof.
\item $\frak F(\mathbb L)$ is well-defined up to  equivalence.
\item If $\mathbb L \subseteq \mathbb L'$ then 
$\frak F(\mathbb L)$ is  equivalent to the full subcategory of $\frak F(\mathbb L')$.
\item
If $\mathbb L$ is a finite set then $\frak F(\mathbb L)$ is  equivalent to 
one in Theorem \ref{Lagcatecomplete}. 
\end{enumerate}
\end{thm}
\begin{proof}
The proof is similar to the proof of Theorem \ref{Lagcatecomplete} so we will be 
rather sketchy.
Since $\mathbb L$ is separable we can find a sequence of finite 
sets of Lagrangian submanifolds $\mathbb L^n$ with the following 
properties.
\begin{enumerate} 
\item[(a)]
For any element $L^{\infty}$ of $\mathbb L$ there exists a 
sequence $L^n \in \mathbb L^n$ such that $L^{\infty} = \lim_{n\to \infty} L^n$.
\item[(b)]
For any $L^n \in \mathbb L^n$ there exists 
$L^{n+1} \in \mathbb L^{n+1}$ such that
$
d_{\rm Hof}(L^n,L^{n+1}) < \epsilon_n
$
with $\sum \epsilon_n < \infty$. ($\epsilon_n$ is independent of $L^n$.)
\item[(c)]
The sequence $\mathbb L^n$ is strongly 
transversal in the sense of Definition \ref{defn1555}.
\end{enumerate}
The existence of $\mathbb L^n$ satisfying (a)(b) is easy to show.
We can then replace 
$\mathbb L^n = \{L^n_i \mid i=1,\dots,m_n\}$
by $\mathbb L^{n \prime} = \{L^{n,\prime}_i \mid i=1,\dots, m_n\}$
inductively such that
elements of $\mathbb L^{n-1}$ is transversal to elements of 
$\mathbb L^{n \prime}$ and that
$
d_{\rm Hof}(L^n_i,L^{n,\prime}_i) < \epsilon'_n
$
with $\sum \epsilon'_n < \infty$.
This $\mathbb L^{n \prime}$ has property (c) also.
\par
We take $\epsilon_n$ as above and modify $\frak F(\mathbb L^n)$
to $\widehat{\frak F}(\mathbb L^n)$ as follows.
For each element $[\hat L]$ of $\mathbb L$  we take and fix 
a sequence $\hat L = (L_1,L_{2},\dots)$ 
such that $L_m \in \mathbb L^m$,
\begin{equation}\label{form171}
d_{\rm Hof}(L^m,L^{m+1}) < 2\epsilon_{m}
\end{equation}
and $\hat L$ represents $[\hat L]$.
Moreover we take and fix a Hamiltonian diffeomorphism 
$\varphi_m$ such that  $\varphi_m(L_m) = L_{m+1}$ and 
$d_{\rm Hof}({\rm id},\varphi_m) < 2\epsilon_m$.
We fix such a choice for each element of $\mathbb L$.
For $[\hat L], [\hat L']$ with $(L_1,\dots,L_m) = 
(L'_1,\dots,L'_m)$ we require that 
$
\varphi_i = \varphi'_i
$
holds for $i<m$. Here $[\hat L], [\hat L']$ are represented by
$(L_i)_{i=1}^{\infty}$, $(L'_i)_{i=1}^{\infty}$
respectively and 
$\varphi_i$, $\varphi'_i$ are Hamiltonian
diffeomorphisms which we associate to $[\hat L], [\hat L']$,
respectively.
\par
The set of objects of $\widehat{\frak F}(\mathbb L^n)$
is a pair $([\hat L],b)$ where $[\hat L]$ is an element of $\mathbb L$ 
and $b$ is a bounding cochain of $L^n$.
\par
The morphism module is defined by
$$
CF(((L^n,L^{n+1},\dots),b),((L^{n \prime},L^{n+1 \prime},\dots),b')) := CF((L^n,b),(L^{n \prime},b')),
$$
where the right hand side is the Floer's cochain module.
\par
Using (b)(c) and (\ref{form171}) we can proceed in the same way as  Step 1 of 
the proof of Theorem \ref{Lagcatecomplete} to obtain 
a $\#$-unital filtered $A_{\infty}$ functor $\Phi^n : \widehat{\frak F}(\mathbb L^n)
\to \widehat{\frak F}(\mathbb L^{n+1})$
with energy loss $\epsilon_n$.  
Note that $\Phi^n_{\rm ob}([\hat L],b) = ([\hat L],(\varphi_n)_* b)$,
where $\varphi_n$ is a Hamilton diffeomorphism which we 
fixed above for $[\hat L]$.
\par
By property (a) the sequence $\{L^n\}$ represents elements 
of $\mathbb L$ and any element of $\mathbb L$ is represented by a certain 
sequence $L^n$. The condition $\Phi^n_{\rm ob}(L^n,b^n) = (L^{n+1},b^{n+1})$
is the condition for $b^n$ mentioned in (1). Thus (1) is proved.
\par
The proof of well-defined-ness is the same as the various well-defined-ness 
proof of Theorem \ref{Lagcatecomplete}.
The claims (3)(4) of Theorem \ref{sepaobj} then is immediate from
the construction.
\end{proof}
\begin{rem}
There are many cases where the metric space $\overline{\frak L}$
is not separable.
An easy example is $T^2$. There is a continuum of Lagrangian $S^1$ which are geodesic 
with respect to the flat metric. They are unobstructed and two of such $S^1$ are not 
Hamiltonian isotopic with each other. Therefore the Hofer distance 
between them is 
infinite. So  $\overline{\frak L}$ contains uncountably many connected components.
In view of the study of family Floer homologies (See \cite{Ab,Tu,Ha}.) such 
continuum (of pairs of Lagrangian submanifolds and their bounding cochains) becomes a set of points of a certain rigid analytic 
space. (Note that the topology on this set coming from rigid analytic structure 
is  much different from the Hofer metric. The latter is finer structure 
in the sense that it distinguishes two Hamiltonian isotopic Lagrangian submanifolds).
So we need to combine two stories somehow to obtain a 
filtered $A_{\infty}$ category containing `all' the unobstructed Lagrangian submanifolds.
\par
In the Fano case such as certain toric manifolds there are cases where 
we can take sufficiently large separable set of Lagrangian submanifolds 
which can play a similar role as the set of `all' Lagrangian submanifolds.
For example in the case $X=S^2$ we may take the set of all 
oriented great circles, 
which is homeomorphic to $S^2$.
\end{rem}
\begin{rem}
The difficulty to generalize the proof of 
Theorem \ref{sepaobj} to the case when $\mathbb L$ is not separable, 
is the following.
During the proof we take an $\mathbb N$ parametrized family of finite sets 
of Lagrangian submanifolds and take an inductive limit with respect to the 
$\mathbb N$ parametrized inductive system. 
When we consider non-separable $\mathbb L$ we need to consider the following 
situation: 
$\mathbb L_1$, $\mathbb L_2$, $\mathbb L_3$, $\mathbb L_4$ are 
finite sets of Lagrangian submanifolds 
and $\Phi_{ji} :  \frak F(\mathbb L_i) \to \frak F(\mathbb L_j)$ 
for $(ji) = (21), (31), (42), (43)$ are defined.
In such a situation we can still take an inductive limit if 
$\Phi_{42}\circ \Phi_{21} = \Phi_{43}\circ \Phi_{31}$ holds 
as an {\it exact} equality. 
We can find $\Phi_{ji}$ such that $\Phi_{42}\circ \Phi_{21}$ 
is homotopic (or homotopy equivalent) to $\Phi_{43}\circ \Phi_{31}$.
However the exact equality is hard to achieve.
\end{rem}
\begin{prop}\label{completion}
For a completed DG-category $\mathscr C$,
there exists a completed DG-category 
$\mathscr C'$ equivalent to $\mathscr C$
such that $\frak{OB}(\mathscr C')$
is complete and Hausdorff.
\end{prop}
\begin{proof}
The proof is mostly the same as the proof of Theorem \ref{sepaobj}.
For $c,c' \in \frak{OB}(\mathscr C)$ we write 
$c \sim c'$ if $d_{\rm Hof,\infty}(c,c') = 0$.
This is an equivalence relation.
We pick up one representative from the equivalence classes
to obtain $C$. We replace $\mathscr C$ by its full subcategory 
the set of whose objects is $C$.
We thus may assume that $\frak{OB}(\mathscr C)$ is Hausdorff.
Let $\overline{\frak{OB}(\mathscr C)}$
be the completion of $\frak{OB}(\mathscr C)$
with respect to the Hofer infinite metric.
For each $\frak c \in \overline{\frak{OB}(\mathscr C)}$ we choose a 
sequence of elements $\{c^n(\frak c)\}$ with $c^n(\frak c) \in \frak{OB}(\mathscr C)$
such that 
\begin{equation}\label{Hausdddddd}
d_{\rm Hof,\infty}(c^n(\frak c),\frak c) < \epsilon_n.
\end{equation}
with $\epsilon_n \to 0$. Here $\epsilon_n$ is independent of $\frak c$.
In case $\frak c = c \in \frak{OB}(\mathscr C)$ we take 
$c^n(\frak c) = c$.
\par
We consider the completed DG-category $\mathscr C^n$
the set of whose objects is $\frak{OB}(\mathscr C)$
and for $\frak c,\frak c' \in \frak{OB}(\mathscr C)$
the morphism complex is defined by 
$$
\mathscr C^n(\frak c,\frak c'): = \mathscr C(c^n(\frak c),c^n(\frak c')).
$$
Structure operations are those of $\mathscr C$.
Using (\ref{Hausdddddd}) we can construct 
a $\#$-unital inductive system of completed DG-categories
$(\{\mathscr C^n\},\{\Phi^n\})$. 
It is easy to see that the inductive limit of this inductive
system is the required $\mathscr C'$.
\end{proof}

\section{A concluding remark: SYZ and KAM.}
\label{sec;KAM}

Let $X$ be an algebraic variety over $\C[[T]]$.
We put $\Lambda_0^{(m)} = \C[[T^{1/m}]]$ 
and $\Lambda^{(m)}$ be its field of fractions.
We put $X^{(m)} = X \times_{\C[[T]]}\Lambda_0^{(m)}$
and $X^{(m)}_{\Lambda} = X \times_{\C[[T]]}\Lambda^{(m)}$.
We fix a vector bundle $\mathcal E_{\Lambda^{(m)}}$ on $X^{(m)}_{\Lambda}$.
We consider the set of all vector bundles $\mathcal E$ on 
$X^{(n)}$ such that
$$
\mathcal E \times_{\Lambda_0^{(n)}} \Lambda^{(mnk)} 
\cong \mathcal E_{\Lambda^{(m)}} \times_{\Lambda^{(m)}} \Lambda^{(mnk)}
$$  
for some $k$.
We write $\mathcal E \sim \mathcal F$ 
if they become isomorphic on $\Lambda_0^{(n')}$ for certain $n'$.
Let $Y$ be the set of $\sim$ equivalence classes.
For $[\mathcal E], [\mathcal F] \in Y$ 
we may choose representatives such that $\mathcal E$ and $\mathcal F$ are vector 
bundles on $X^{(m)}$.
We define a metric $d([\mathcal E], [\mathcal F])$
so that $d([\mathcal E], [\mathcal F]) < \epsilon$
if and only if there exists $n$ such that
\begin{enumerate}
\item 
There exists an isomorphism 
$\varphi : \mathcal E \times_{\Lambda^{(m)}} \Lambda^{(mn)}
\to \mathcal F \times_{\Lambda^{(m)}} \Lambda^{(mn)}$.
\item
$T^{k/mn} \varphi$ and $T^{k/mn}\varphi^{-1}$ are defined on 
$X \times_{\rm Spec(\C[[T]])}\Lambda_0^{(nm)}$.
\item
$k/mn < \epsilon$.
\end{enumerate}
Then $(Y,d)$ becomes a metric space.
Let $\overline Y$ be its completion.
We might extend this definition to coherent sheaves or objects 
of its derived category.
It seems likely that such a metric space is 
related to the space of Lagrangian submanifolds and its 
Hofer distance via homological mirror symmetry.
\par
Suppose $\pi: X \to B$ is an SYG-fibration
(\cite{SYZ}).
Let $f: X \to X$ be a symplectic diffeomorphism
which commutes with $\pi$ and 
which is linear on each Lagrangian fiber.
We may expect that such a symplectic diffeomorphism 
is obtained as a monodromy of a maximal degenerate 
family of Calabi-Yau manifolds (\cite{galois}).
\par
We consider its Mirror $\hat X$ which could be a
scheme over $\Lambda_0^{(m)}$.
Homological Mirror symmetry predicts that the 
fibers of $\pi: X \to B$ 
(at a rational point of $B$) correspond to a 
skyscraper sheaf $\mathcal E_p$ on $\hat X\times_{\Lambda_0^{(m)}}\Lambda^{(m)}$.
For this object $\mathcal E_{\Lambda^{(m)}}: = \mathcal E_p$ we carry out the construction 
suggested above 
to obtain a metric space $\overline{Y_p}$.
It is expected that  $f$ via the mirror symmetry
becomes an automorphism (of the derived category of coherent sheaves) 
given by $\mathcal E \mapsto \mathcal E \otimes \mathcal L$ 
on $\overline{Y_p}$ for a certain ample line bundle $\mathcal L$. 
(See \cite{lecture}.)
Note that $\mathcal E_p\otimes \mathcal L$ is $\mathcal E_p$ 
for the skyscraper sheaf $\mathcal E_p$.
The sheaf on  $\hat X\times_{\Lambda_0^{(m)}}\Lambda^{(m)}$ which corresponds to the fiber 
is a fixed point of this automorphism.
\par
Now we take $f_{\epsilon}$ which is sufficiently close to $f$.
The KAM theorem seems to say that the Mirror $\hat f_{\epsilon}$
to $f_{\epsilon}$ 
(that is an isometry and $x \mapsto d(x,f_{\epsilon}(x))$ is bounded)
has a fixed point on  
$Y_p$ for $p$ in a subset of almost full measure.
\par
The space $\overline{Y_p}$ seems to be closely related to 
a space appearing in rigid analytic geometry.
It might be possible to expect that there is a certain 
fixed point theorem which implies the existence 
of a fixed point of $\hat f_{\epsilon}$ on $\overline{Y_p}$.
\par
It would be very interesting if one can relate 
such a fixed point theorem to the KAM theory and 
also relate the existence of KAM tori with 
the existence of a fixed point of $\hat f_{\epsilon}$ on $Y_p$.

\bibliographystyle{amsalpha}

\end{document}